\documentclass[12pt]{amsart}
\usepackage{amssymb}
\usepackage{amsmath}
\usepackage{amsthm}
\usepackage{times}
\usepackage{changepage}
\usepackage{caption}
\usepackage{geometry}
\usepackage[curve,all]{xy}
\xyoption{curve}
\xyoption{line}
\xyoption{arc}
\xyoption{color}
\usepackage{array}
\usepackage{enumitem}
\usepackage{graphicx}
\usepackage[dvipdfmx]{color}
\usepackage{color}
\usepackage{amsthm}
\newtheorem{theorem}{Theorem}[section]
\newtheorem{lemma}[theorem]{Lemma}
\newtheorem{corollary}[theorem]{Corollary}
\newtheorem{proposition}[theorem]{Proposition}
\newtheorem{prop}[theorem]{Proposition}

\usepackage{subfig}
 
\theoremstyle{definition}

\newtheorem{example}[theorem]{Example}

\theoremstyle{remark}
\newtheorem{remark}[theorem]{Remark}

\numberwithin{equation}{section}

\newcommand{\calP}{\mathcal{P}}

\newcommand{\la}{\langle}
\newcommand{\ra}{\rangle}

\def\deg{{\text{deg}}}

\def\I{{\text{I}}}
\def\II{{\text{II}}}
\def\III{{\text{III}}}
\def\IV{{\text{IV}}}

\begin{document}
\title [Coble surfaces]{Coble surfaces in characteristic two}

\author{Toshiyuki Katsura}
\address{Graduate School of Mathematical Sciences, The University of Tokyo,
Meguro-ku, Tokyo 153-8914, Japan}
\email{tkatsura@ms.u-tokyo.ac.jp}
\thanks{}

\author{Shigeyuki Kond\=o}
\address{Graduate School of Mathematics, Nagoya University, Nagoya,
464-8602, Japan}
\email{kondo@math.nagoya-u.ac.jp}
\thanks{Research of the first author is partially supported by JSPS Grant-in-Aid 
for Scientific Research (C) No.20K03530 and the second author by JSPS
Grant-in-Aid for Scientific Research (A) No.20H00112.}

\begin{abstract}
We study Coble surfaces in characteristic $2$, in particular, singularities of their canonical coverings.  As an application we classify Coble surfaces with finite automorphism group in characteristic $2$. There are exactly 9 types of such surfaces.
\end{abstract}

\maketitle

\section{Introduction}\label{sec1}

A {\it Coble surface} $S$ is a rational surface with $|-K_S|=\emptyset$ and 
$|-2K_S|=\{ B_1+\cdots + B_n\}$ where $B_i$ is a non-singular rational curve and 
$B_i\cap B_j =\emptyset$ $(i\not=j)$. This is called a terminal Coble surface of $K3$ type 
in Dolgachev and Zhang \cite{DZ}.
We call $B_1,\ldots, B_n$ the boundary components of $S$.  Morrison \cite{Morrison} showed that a Coble surface appears as a degeneration of Enriques surfaces.
Over the complex numbers, the moduli space of complex Enriques surfaces can be described as an open subset of an arithmetic quotient of a 10-dimensional bounded symmetric domain of type IV which is the complement of a Heegner divisor.  
A general point of the Heegner divisor corresponds to a Coble surface. This is a peculiar
phenomenon of Enriques surfaces which does not occur in case of $K3$ surfaces.
Coble surfaces inherit many properties of Enriques surfaces and complement 
the theory of Enriques surfaces.
Thus it is interesting to study Coble surfaces like as Enriques surfaces.

By definition any Coble surface has a canonical double covering $\pi : X\to S$ defined by
$-K_S$.  In case that the characteristic $p$ of the
base field $k$ is not equal to 2, $\pi$ is a separable double covering branched along a 
non-singular divisor $B_1+\cdots + B_n$ and hence $X$ is smooth.  In particular $X$ is a $K3$ surface.
Thus one can apply the theory of $K3$ surfaces to the case of Coble surfaces.
On the other hand, in case of $p=2$, $\pi$ is a purely inseparable $\mu_2$-covering, 
and $S$ has always singularities and might be non-normal.  
This situation is similar to the case of
Enriques surfaces in characteristic 2.  Recall that Enriques surfaces 
$Y$ in characteristic 2 are classified into three types of classical, $\mu_2$- or $\alpha_2$-surfaces according to ${\rm Pic}_Y^\tau \cong {\bf Z}/2{\bf Z}$, $\mu_2$ or $\alpha_2$, respectively (Bombieri and Mumford \cite{BM}).
A Coble surface in charcteristic 2 is an analogue of classical Enriques surfaces.  
The purpose of this paper is to study the double covering $\pi$ and, as an application,
to give the classification of Coble surfaces with finite automorphism group in characteristic 2.

In the following we assume the the ground field $k$ is an algebraically closed field in
characteristic 2.  First of all, we show that the canonical covering $X$ of a Coble surface 
is a $K3$-like surface, that is, the dualizing sheaf $\omega_X$ is trivial and 
${\rm H}^1(X, {\mathcal O}_X) = 0$ (Theorem \ref{K3-like}).  One can define a rational 1-form $\eta$ on $S$ naturally whose divisor is written as
$$(\eta) = -\sum_{i=1}^n B_i + B$$
with $(\sum_i B_i) \cap B=\emptyset$ and $B=2A$ for a divisor $A$ (Theorem \ref{eta}).  We call $A$ the conductrix of $S$.   We will show the following theorem which is an analogue of the case of Enriques surfaces.

\begin{theorem}\label{main1}
\begin{itemize}
\item[$({\rm 1})$]
Assume that $A=0$.  
Then the singularities of $X$ are all rational double points 
and the minimal non-singular model of $X$ is a supersingular $K3$ surface.
\item[$({\rm 2})$] Assume that $A\ne 0$.  Then $X$ is a rational surface, $A^2=-2$
and the isolated singularities of $X$ are all rational double points of type $A_1$ and the number of isolated singularities is equal to $4 -n$. 
\end{itemize}
\end{theorem}
We will study more details of conductrices, and can obtain the list of possible conductrices $A$, 
which is the same as the one of Enriques surfaces
given in Ekedahl and Shepherd-Barron \cite[Theorems 2.2, 3.1]{ES} for $A^2 = -2$.

Next we will discuss Coble surfaces with finite automorphism group.
A general Coble surface has an infinite group of automorphisms (e.g. Dolgachev and Kond\=o \cite[Theorem 9.6.1]{DK}) and hence it is natural to classify Coble surfaces with finite automorphism group as in the case of Enriques surfaces.  Recall that 
Katsura, Kond\=o and Martin \cite{KKM} classified classical Enriques surfaces in characteristic 2 with finite automorphism group: there are exactly 8 classes of such
Enriques surfaces.  
On the other hand, in case of characteristic $p \ne 2$, the second author \cite{Kon2} recently gave the classification of such Coble surfaces.  

Let $Y$ be an Enriques surface and let ${\rm Num}(Y)$ be the quotient group of the
N\'eron-Severi group by the torsion subgroup.  Then ${\rm Num}(Y)$ is a lattice of signature $(1,9)$ which is an important role to study Enriques surfaces.  Also
any non-singular rational curve on $Y$ has the
self-intersection number $-2$. 
On the other hand,
a non-singular rational curve on a Coble surface $S$ has the self-intersection number $-1, -2$ or $-4$ and the Picard number $\rho(S)$ is greater than 10.  Instead of the N\'eron-Severi lattice ${\rm NS}(S)$ we consider a lattice ${\rm CM}(S)$ of signature $(1,9)$
called Coble-Mukai lattice and defined by
the orthogonal complement of $B_1,\ldots, B_n$ in the quadratic space generated by
${\rm NS}(S)$ and ${1\over 2}B_1,\ldots, {1\over 2}B_n$.
An effective class $\alpha$ 
in ${\rm CM}(S)$
with $\alpha^2=-2$
 is called an effective root, and an effective root $\alpha$ is called irreducible if
$|\alpha -\beta|=\emptyset$ for any other effective root $\beta$.  The Coble-Mukai lattice and
effective irreducible roots work like as ${\rm Num}(Y)$ for an Enriques surface $Y$ and
non-singular rational curves on $Y$.

The following is the second main theorem of this paper.  In Table \ref{mainTable}, ``dual graph" means the one of effective irreducible roots, ``Type" means the type of the dual graph, ${\rm Aut}(S)$ is the group of automorphisms of $S$, $n$ is the number of boundary components and ``dim" is the dimension of the modli space of Coble surfaces given type.

\begin{theorem}\label{main2}
Coble surfaces with finite automorphism group in characteristic $2$ are classified 
as in the following Table {\rm \ref{mainTable}}.  The moduli space of Coble surfaces 
with $n$ boundary components of each type is irreducible and  of dimension given in the Table.  
Every Coble surface is a specialization of classical
Enriques surfaces with finite automorphism group.

\vspace{-0.5cm}

\hspace{-1.0cm}
%%%%%%%
\hspace{-1.0cm}\begin{table}[!htb]
 \begin{center}
  \includegraphics[width=180mm]{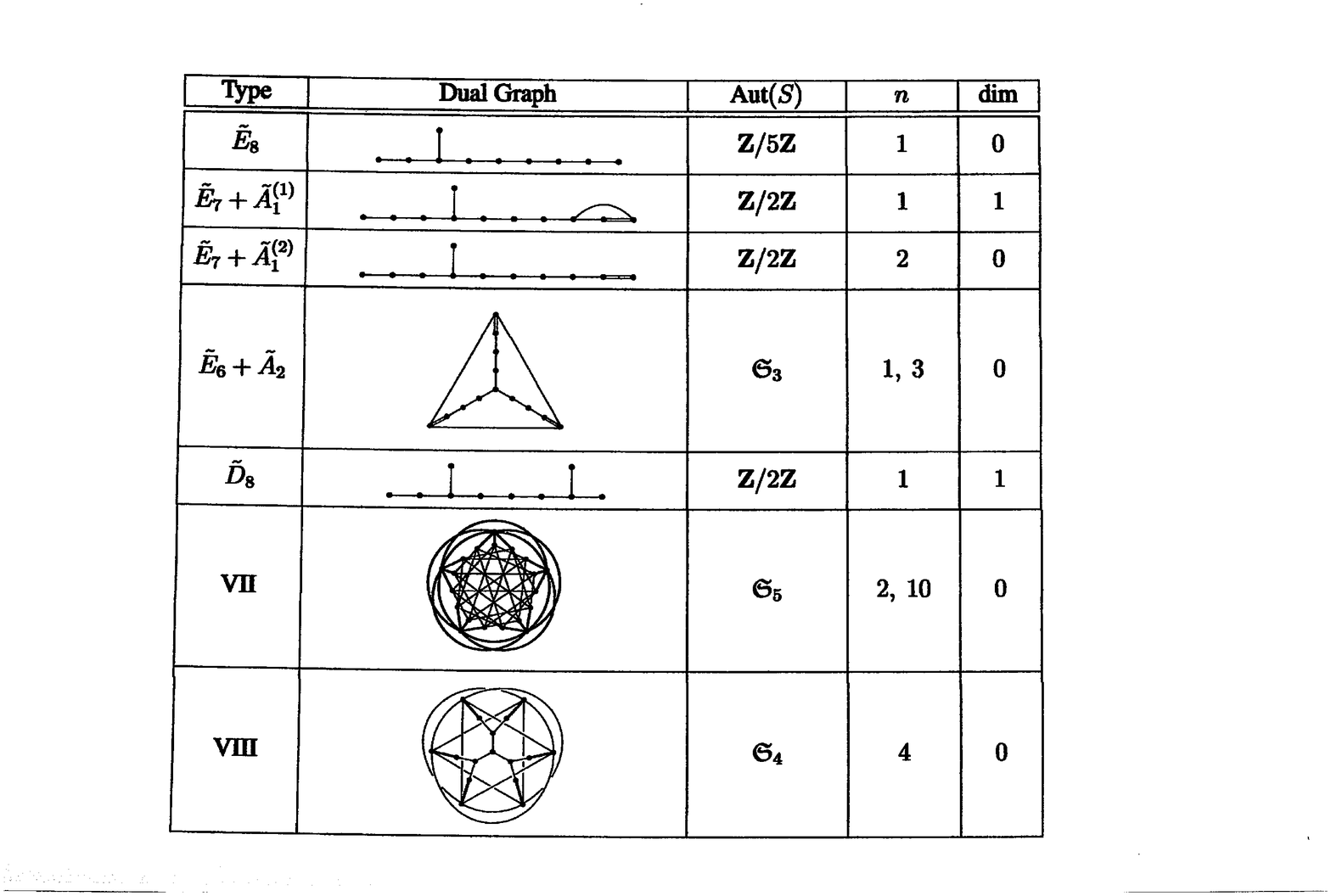}
 \end{center}
 \caption{Coble surfaces with finite automorphism group ($p=2$)}
 \label{mainTable}
\end{table}
%%%%%%%

%\vspace{-0.5cm}
%\begin{table}[hbt]
%\hspace*{-0.9cm}
%\begin{center}
%\includegraphics[scale=0.40,clip]{table1.eps}
%\end{center}
% \caption{Coble surfaces with finite automorphism group ($p=2$)}
% \label{mainTable}
%\end{table}

\end{theorem}

\vspace{-0.8cm}
We remark that a Coble surface of type VII appears in case $p=2, 5$ and
other Coble surfaces in the Table \ref{mainTable} appear only in characteristic 2.  
Moreover the minimal resolution of the canonical covering $X$ is the supersingular $K3$ surface with Artin invariant 1 if $S$ is of type VII (i.e. $A=0$) and rational in other cases ($A\ne 0$).

The proof of classification of possible dual graphs in Table \ref{mainTable} is the same as that for Enriques surfaces given in Katsura, Kond\=o and Martin \cite{KKM} by using the list of conductrices.  We can determine the possible number $n$ of boundary components of Coble surfaces from its dual graph.  
On the other hand, since Coble surfaces are rational, one can
construct Coble surfaces with finite automorphism group by blowing up ${\bf P}^2$, 
which allow us to determine the moduli space of such Coble surfaces.
On the other hand, in \cite{KKM}, the authors gave examples of families of Enriques 
surfaces with finite automorphism group by taking the quotients of rational surfaces by
derivations.  We show that every Coble surface in Table \ref{mainTable} is also obtained as a specialization of these Enriques surfaces.  

The plan of the paper is as follows.  In \S \ref{sec2} we summarize known results used in the paper.  In Section \ref{sec3} we study the canonical coverings of Coble surfaces in characteristic 2 and their singularities, and in Section \ref{sec4} discuss
the properties of the conductrices of Coble surfaces and give a classification of
possible conductrices.  Section \ref{sec5} devotes the classification of possible dual graphs of $(-2)$-curves  on Coble surfaces with finite automorphsim and the number of boundary components.  Finally in Section \ref{sec6} we give a construction of Coble surfaces with finite automorphism group and determine their moduli spaces..

The authors would like to thank N. Shepherd-Barron for pointing out
Lemma \ref{exclude}.

\section{Preliminaries}\label{sec2}

A surface is a non-singular complete algebraic surface defined over an algebraically
closed field $k$ of characteristic $p\geq 0$.  A non-singular rational curve $C$ on a surface $S$ is called $(-m)$-curve if $C^2=-m$.  We denote by $K_S$ the canonical divisor
of $S$.

\subsection{Vector fields}\label{derivation}

Assume ${\rm char}(k)=p > 0$ and let $S$ be a surface defined over $k$.
A rational vector field $D$ on $S$ is said to be $p$-closed if there exists
a rational function $f$ on $S$ such that $D^p = fD$. 
A vector field $D$ is of additive type (resp. of multiplicative type) if $D^p=0$ (resp. $D^p=D$).
Let $\{U_{i} = {\rm Spec} A_{i}\}$ be an affine open covering of $S$. We set 
$A_{i}^{D} = \{\alpha \in A_{i} \mid D(\alpha) = 0\}$. 
The affine surfaces $\{U_{i}^{D} = {\rm Spec} A_{i}^{D}\}$ glue together to 
define a normal quotient surface $S^{D}$.

Now, we  assume that $D$ is $p$-closed. Then,
the natural morphism 
\begin{equation}\label{naturalmapPI}
\pi : S \longrightarrow S^D 
\end{equation}
is a purely
inseparable morphism of degree $p$. 
If the affine open covering $\{U_{i}\}$ of $S$ is fine enough, then
taking local coordinates $x_{i}, y_{i}$
on $U_{i}$, we see that there exist $g_{i}, h_{i}\in A_{i}$ and 
a rational function $f_{i}$
such that the divisors defined by $g_{i} = 0$ and by $h_{i} = 0$ have no common components,
and such that
$$
 D = f_{i}\left(g_{i}\frac{\partial}{\partial x_{i}} + h_{i}\frac{\partial}{\partial y_{i}}\right)
\quad \mbox{on}~U_{i}.
$$
By Rudakov and Shafarevich \cite[Section 1]{RS}, the divisors $(f_{i})$ on $U_{i}$
give a global divisor $(D)$ on $S$, and the zero-cycles defined
by the ideal $(g_{i}, h_{i})$ on $U_{i}$ give a global zero cycle 
$\langle D \rangle $ on $S$. A point contained in the support of
$\langle D \rangle $ is called an isolated singular point of $D$.
If $D$ has no isolated singular point, $D$ is said to be divisorial.
Rudakov and Shafarevich \cite[Theorem 1, Corollary]{RS} 
showed that $S^D$ is non-singular
if $\langle D \rangle  = 0$, i.e. $D$ is divisorial.
When $S^D$ is non-singular,
they also showed a canonical divisor formula
\begin{equation}\label{canonical}
K_{S} \sim \pi^{*}K_{S^D} + (p - 1)(D),
\end{equation}
where $\sim$ means linear equivalence.
As for the Euler number $c_{2}(S)$ of $S$, we have a formula
\begin{equation}\label{euler}
c_{2}(S) = \deg \langle D \rangle  - K_{S}\cdot (D) - (D)^2
\end{equation}
(cf. Katsura and Takeda \cite[Proposition 2.1]{KT}). 

Now we consider an irreducible curve $C$ on $S$ and we set $C' = \pi (C)$.
Take an affine open set $U_{i}$ above such that $C \cap U_{i}$ is non-empty.
The curve $C$ is said to be integral with respect to the vector field $D$
if $g_{i}\frac{\partial}{\partial x_{i}} + h_{i}\frac{\partial}{\partial y_{i}}$
is tangent to $C$ at a general point of $C \cap U_{i}$. Then, Rudakov-Shafarevich
\cite[Proposition 1]{RS} showed the following proposition:

\begin{prop}\label{insep}
\begin{itemize}
\item[$({\rm i})$]
If $C$ is integral, then $C = \pi^*(C')$ and $C^2 = pC'^2$.
\item[$({\rm ii})$] 
If $C$ is not integral, then $pC = \pi^*(C')$ and $pC^2 = C'^2$.
\end{itemize}
\end{prop}

\subsection{Genus 1 fibrations}

We recall  elliptic and quasi-elliptic fibrations on a rational surface.  For simplicity, we call an elliptic or a quasi-elliptic fibration a genus 1 fibration. 
A genus 1 fibration is called relatively minimal if any fiber contains no $(-1)$-curves. 
Let $f : S \to {\bf P}^1$ be a genus 1 fibration on a rational surface and let $F$ be a general fiber.  A non-singular rational curve $C$ on $S$ is called a section (resp. a 2-section) if $C\cdot F=1$ (resp. $C\cdot F=2$).  

We use the symbols $\I_n$ $(n\geq 1)$, $\I_n^*$ $(n\geq 0)$, $\II$, $\III$, $\IV$, $\II^*$, 
$\III^*$, $\IV^*$ of relatively minimal singular fibers of $f$ in the sense of Kodaira.  
The dual graph of $(-2)$-curves in a singular fiber of type $\I_n$ $(n\geq 2)$, $\I_n^*$ $(n\geq 0)$, 
$\III$, $\IV$, $\II^*$, $\III^*$ or $\IV^*$ is an extended Dynkin diagram of type 
$\tilde{A}_{n-1}$, $\tilde{D}_{n+4}$, 
$\tilde{A_1}$, $\tilde{A_2}$, $\tilde{E_8}$, $\tilde{E_7}$ or $\tilde{E_6}$, respectively.
A fiber with multiplicity 2 is called a multiple fiber.
For a multiple singular fiber of type $F$, we write $2F$.
If, for example, $f$ has a multiple singular fiber of type $\III$ and a singular fiber of type $\IV^*$, then we say that $f$ has singular fibers of type $(2\III, \IV^*)$ or
of type $(2\tilde{A}_2, \tilde{E}_6)$.
If $f$ has a section and its Mordell-Weil group is torsion, then $f$ is called extremal.
We use the following classifications of extremal rational elliptic and rational quasi-elliptic fibrations.

\begin{prop}\label{Lang}{\rm (Lang \cite{La2}, \cite{La3})}
The following are the singular fibers of  relatively minimal extremal elliptic fibrations on rational surfaces in characteristic $2:$
$$({\rm II}^*),\ ({\rm II}^*, {\rm I}_1),\ ({\rm III}^*, {\rm I}_2),\ ({\rm IV}^*, {\rm IV}),\ ({\rm IV}^*, {\rm I}_3, {\rm I}_1),\ ({\rm I}_4^*),\ ({\rm I}_1^*, {\rm I}_4),$$
$$({\rm I}_9, {\rm I}_1, {\rm I}_1, {\rm I}_1),\ ({\rm I}_8, {\rm III}),\ ({\rm I}_6, {\rm IV}, {\rm I}_2),\ ({\rm I}_5, {\rm I}_5, {\rm I}_1, {\rm I}_1),\ ({\rm I}_3, {\rm I}_3, {\rm I}_3, {\rm I}_3).$$
\end{prop} 

\begin{prop}\label{Ito}{\rm (Ito \cite{Ito})}
The following are the singular fibers of relatively minimal quasi-elliptic fibrations on rational surfaces in characteristic $2:$
$$({\rm II}^*),\ ({\rm III}^*, {\rm III}),\ ({\rm I}_4^*),\ ({\rm I}_2^*, {\rm III}, {\rm III}), \ 
({\rm I}_0^*, {\rm I}_0^*),$$
$$({\rm I}_0^*, {\rm III}, {\rm III}, {\rm III}, {\rm III}),\ ({\rm III}, {\rm III}, {\rm III}, {\rm III}, {\rm III}, {\rm III}, {\rm III}, {\rm III}).$$
\end{prop} 

\noindent
Note that any quasi-elliptic fibration is extremal. 
Assume that $f$ is quasi-elliptic.  Then 
the closure of the set of singular points of irreducible fibers (of type II) of $f$ is
called the curve of cusps which is a 2-section of $f$.

\subsection{Coble surfaces}\label{surfaces}

Let $S$ be a Coble surface with boundary components $B_1,\ldots, B_n$.
By $|-2K_S| =\{ B_1+\cdots +B_n\}$ and the adjunction formula, $B_1^2=\cdots =B_n^2=-4$
and $K_S^2=-n$.  Since $|-2K_S| = \{B_1+\cdots+B_n\}$, the only curves with negative self-intersection on $S$ are $B_1,\ldots,B_n$, $(-2)$-curves and $(-1)$-curves.
Note that $(-2)$-curves do not intersect any boundary component $B_i$.

\begin{proposition}\label{basic}{\rm (Dolgachev and Kond\=o \cite[Proposition 9.1.2]{DK})}
Any Coble surface $S$ is a basic type, that is, there exists a birational morphism from
$S$ to ${\bf P}^2$.  Moreover $S$ admits a birational morphism $f:S\to {\bf P}^2$
such that the image of $B_1+\cdots +B_n$ is a curve $W$ of degree $6$.  
The fundamental points of $f^{-1}:{\bf P}^2\to S$ are among singular points of $W$ (including infinitely near points). 
\end{proposition}

Consider 
a pencil $|3mh - m(p_1+\cdots + p_9)|$ of curves of degree $3m$ on ${\bf P}^2$ 
with points $p_i$ of multiplicity $\geq m$ $(m=1$  or $2)$ where $h$ is a line on ${\bf P}^2$.  Blowing up its base points we obtain a rational genus 1 fibration called a Halphen
surface of index $m$.
A Halphen surface of index 1 is nothing but a Jacobian fibration.  
Moreover we have the following.

\begin{prop}\label{Halphen}{\rm (Dolgachev and Kond\=o \cite[Proposition 9.1.4]{DK})}
Let $f:S\to {\bf P}^1$ be a genus $1$ fibration on a Coble surface $S$.  
Then $S$ is obtained from a Halphen surface of index $2$ by 
blowing up all singular points {\rm (}and their infinitely near points in the case of type {\rm III, IV )} of one reduced singular fiber or from a Halphen surface of index $1$, i.e. a Jacobian fibration, by blowing up all singular points {\rm (}and their infinitely near points in the case of type {\rm III, IV )} of two reduced singular fibers. 
\end{prop}
\noindent

\begin{remark}\label{HalphenBlowUp}
It follows from Proposition \ref{Halphen} that a fiber of type II, III, IV, or ${\rm I}_n$ is blown up.
The configurations of fibers blown up are given in Figure 1.
In the figure, the dotted lines are $(-1)$-curve,
and $B_i$ means $(-4)$-curves. In case of type ${\rm I}_n$,  the curves $E_i$ ($i = 1, \cdots, n$) are $(-1)$-curves
and $B_i$ ($i = 1, \cdots, n$) are the $(-4)$-curves. 
The fiber is given by 
$\sum_{i = 1}^{n}B_i + 2\sum_{i = 1}^{n}E_i$.
In case of type ${\rm II}$, the fiber is given by $B+2E$.
In case of type ${\rm III}$, $E_1$ is a $(-2)$-curve and 
the fiber is given by $B_1 + B_2  + 2(E_1 + 2E_2)$.
In case of type ${\rm IV}$, the fiber is given by $B_1 + B_2 + B_3 + 3B_4 + 4(E_1 + E_2 + E_3)$.

%%%%%
\begin{figure}[!htb]
%\begin{center}
\hspace{-4.0cm}  \includegraphics[width=180mm]{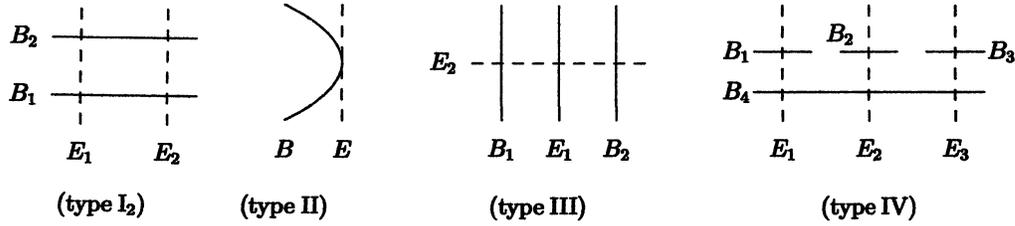}
% \end{center}
 \vspace{-5.0cm}
 \caption{Blowing-up the singular points of a fiber}
 \label{Halsing}
\end{figure}

%%%%%

\end{remark}

\begin{prop}\label{Boundary}{\rm (Dolgachev and Kond\=o \cite[Corollary 9.1.5]{DK})}
Let $S$ be a Coble surface.  Then $K_S^2 \geq -10$.  In particular the number of boundary components of $S$ is at most $10$.
\end{prop}

\subsection{Coble-Mukai lattices.}\label{CMlattice}

For a Coble surface $S$, instead of its N\'eron-Severi group, we consider a lattice ${\rm CM}(S)$ of rank 10, called Coble-Mukai lattice, which is defined by the orthogonal complement of
the boundary components $B_1, \ldots, B_n$ in the quadratic space $\widetilde{\rm Pic}(S)$ generated by 
${\rm Pic}(S)$ and ${1\over 2}B_1, \ldots, {1\over 2}B_n$. 
Then ${\rm CM}(S)$ together with the intersection form is a lattice (Dolgachev and Kond\=o \cite[\S 9.2]{DK}).  

We call an effective class $\alpha\in {\rm CM}(S)$ with $\alpha^2=-2$ an effective root. 
We say that an effective root $\alpha$ is irreducible if $|\alpha-\beta| = \emptyset$ for any other effective root $\beta$.

\begin{lemma}\label{roots}{\rm (Dolgachev and Kond\=o \cite[Lemma 9.2.1]{DK})} Let $\alpha$ be an effective irreducible root. Then $\alpha$ is either the divisor class of a $(-2)$-curve or the ${\bf Q}$-divisor class of an effective root of the form 
$2E+{1\over 2}B_j+{1\over 2}B_k$, where $E$ is a $(-1)$-curve intersecting two different boundary components $B_j,B_k$. 
\end{lemma}

We call an effective irreducible root $\alpha$ a $(-2)$-root
if it is represented by
a $(-2)$-curve and a $(-1)$-root if it is represented by 
$2E+{1\over 2}B_j+{1\over 2}B_k$.  An easy calculation shows the following Lemma.

\begin{lemma}\label{roots2}{\rm (Kond\=o \cite[Lemma 3.3]{Kon2})} 
\begin{itemize}
\item[{\rm (1)}]
Two $(-1)$-roots have intersection number $1$ if and only if the associated $(-1)$-curves do not meet and the roots share precisely one boundary component.
\item[{\rm (2)}]
The intersection number of a $(-1)$-root and a $(-2)$-root is always even.
\end{itemize}
\end{lemma}

Let ${\rm CM}(S)_{\bf R}^+$ be the connected component of $\{ x\in {\rm CM}(S)\otimes {\bf R} \ : \ x^2 > 0\}$
containing a nef class.
Denote by $\overline{{\rm CM}(S)}_{\bf R}^+$
the closure of ${\rm CM}(S)_{\bf R}^+$ in ${\rm CM}(S)\otimes {\bf R}$.
Let $W(S)$ be the subgroup of ${\rm O}({\rm CM}(S))$ generated by all reflections $s_{\alpha}$, where 
$\alpha$ is an effective irreducible root. 
The group $W(S)$ naturally acts on $\overline{{\rm CM}(S)}_{\bf R}^+$.
We denote by $D(S)$ the domain defined by
$$D(S)=\{ x\in \overline{{\rm CM}(S)}_{\bf R}^+\ : \ x\cdot \alpha \geq 0 \ {\rm for \ any \ effective \ irreducible \ root} \ \alpha \}.$$
Obviously, the automorphism group ${\rm Aut}(S)$ leaves the set of curves $\{B_1,\ldots,B_n\}$ invariant and hence acts on ${\rm CM}(S)$ and on $D(S)$. 
The following proposition holds.

\begin{prop}\label{NefCone}{\rm (Dolgachev and Kond\=o \cite[Proposition 9.2.2]{DK})}
Let $x \in \overline{{\rm CM}(S)}_{\bf R}^+ \cap {\rm CM}(S)$.  Then $x$ is nef if and only if $x\in D(S)$.
In other word, the intersection of the nef cone of $S$ with $\overline{{\rm CM}(S)}_{\bf R}^+$ is a fundamental domain of $W(S)$.
\end{prop}

A vector $f\in D(S)\cap {\rm CM}(S)$ is called isotropic if $f^2=0$ and primitive
if $f=mf'$, $m \in {\bf Z}$, $f'\in {\rm CM}(S)$ implies $m=\pm 1$.

\begin{lemma}\label{fibration-isotropic}
The set of genus $1$ fibration on a Coble surface $S$ bijectively corresponds to the set 
of primitive isotropic vectors in $D(S)\cap {\rm CM}(S)$.  
\end{lemma}
\begin{proof}
By the assumption, $f\cdot K_S=0$ and hence the proof of Cossec, Dolgachev and Liedtke \cite[Proposition 2.2.8]{CDL}
works well in this case, too.
\end{proof}

\begin{remark}\label{two-double}
Let $f: S\to {\bf P}^1$ be a genus 1 fibration obtained from a Halphen surface
$h : H \to {\bf P}^1$.  
When we consider $f$ in ${\rm CM}(S)$, we can say that the type of the extended Dynkin 
diagram of reducible 
singular fibers of $f$
is the same as those of $h$.  
Let $F$ be the fiber of $f$ corresponding to a reducible singular fiber $F_0$ of $h$.
If $F_0$ is not blown up, then $F = F_0$.  Assume that $F_0$ is blown up.
If $F_0$ is of type ${\rm I}_n$ $(n\geq 2)$, then $F$ consists of
$n$ $(-4)$-curves $B_1,\ldots, B_n$ and $n$ $(-1)$-curves $E_1,\ldots, E_n$ with
$B_i\cdot E_i = B_{i+1}\cdot E_i=1$.  Here we use the same notation as in Remark \ref{HalphenBlowUp}.
Then 
$$F= B_1+ \cdots + B_n + 2E_1+\cdots + 2E_n=
({1\over 2}B_1+{1\over 2}B_2 + 2E_1) + \cdots + ({1\over 2}B_{n} + {1\over 2}B_1 + 2E_{n}).$$  
Thus $F$ is the sum of $n$ $(-1)$-roots forming a dual graph of type $\tilde{A}_{n-1}$.
Similarly, if $F_0$ is of type ${\rm III}$ (resp. of type ${\rm IV}$), then
$$F= 2(({1\over 2}B_1+{1\over 2}B_2 + 2E_2) + E_1)\quad 
(resp. \ \ 2(\sum_{i=1}^3 {1\over 2}B_i+{1\over 2}B_4 + 2E_i))
$$
is the double of the sum of a $(-1)$-root and a $(-2)$-root (resp. three $(-1)$-roots) 
forming a dual graph of type $\tilde{A}_1$ (resp. $\tilde{A}_2$).
Note that $F$ looks like a multiple fiber in case of type ${\rm III}, {\rm IV}$.
\end{remark}

As a corollary of Proposition \ref{NefCone} and Lemma \ref{fibration-isotropic}, we obtain the following Proposition.

\begin{prop}\label{FiniteIndex}{\rm (Dolgachev and Kond\=o \cite[Theorem 9.8.1]{DK})}
Let $S$ be a Coble surface.  Suppose $W(S)$ is of finite index in ${\rm O}({\rm CM}(S))$.  Then ${\rm Aut}(S)$ is finite.
\end{prop}

Consider a genus 1 fibration on a Coble surface $f : S \to {\bf P}^1$.  Then the Mordell-Weil group of the Jacobian fibration of $f$ acts on $S$ faithfully as automorphisms.  
This implies the following Proposition.

\begin{prop}\label{MW-Dolgachev}
Assume that the automorphism group of a Coble surface $S$ is finite.  Then any genus $1$ fibration on $S$ is extremal.  
\end{prop}

An effective irreducible root $r$ is called a special 2-section of 
a genus 1 fibration $f:S\to {\bf P}^1$
if $r\cdot F=2$ where $F$ is a general fiber of $f$.  In this case the fibration $f$ is called special.

In case of Enriques surfaces, the following is known.

\begin{proposition}\label{Cossec}{\rm (Cossec \cite[Theorem 4]{C}, Lang \cite[Theorem A.3]{La1})}
If an Enriques surface $Y$ contains a $(-2)$-curve, then there exists a special genus $1$ fibration on $Y$.
\end{proposition}

The following lemma is the analogue of Proposition \ref{Cossec} for Coble surfaces.
For the proof we refer the reader to \cite[Lemma 3.13]{Kon2}.

\begin{lemma}\label{specialfibration} Let $S$ be a Coble surface.  
Assume that $S$ has an irreducible effective root.  
Then there exists a special genus $1$ fibration on $S$.  
\end{lemma}

\subsection{Vinberg's criterion}

Let $L$ be a lattice of signature $(1,r)$.
We recall Vinberg's criterion,
which guarantees that a group generated by a finite number of reflections is
of finite index in ${\rm O}(L)$.

Let $\Delta$ be a finite set of $(-2)$-vectors in $L$.
Let $\Gamma$ be the graph of $\Delta$, that is,
$\Delta$ is the set of vertices of $\Gamma$ and two vertices $\delta$ and $\delta'$ are joined by $m$-tuple lines if $\langle \delta, \delta'\rangle=m$.
We assume that the cone
$$K(\Gamma) = \{ x \in L\otimes {\bf R} \ : \ \langle x, \delta_i \rangle \geq 0, \ \delta_i \in \Delta\}$$
is a strictly convex cone. Such a $\Gamma$ is called non-degenerate.
A connected parabolic subdiagram $\Gamma'$ in $\Gamma$ is a  Dynkin diagram of type 
$\tilde{A}_m$, $\tilde{D}_n$ or $\tilde{E}_k$ (see Vinberg \cite[p. 345, Table 2]{V}).  If the number of vertices of 
$\Gamma'$ is $r+1$, then $r$ is called the rank of $\Gamma'$.  A disjoint union of connected parabolic subdiagrams is called a parabolic subdiagram of $\Gamma$.  
We denote by $\tilde{K_1}\oplus \cdots \oplus  \tilde{K_s}$ a parabolic subdiagram 
which is a disjoint union of 
connected parabolic subdiagrams of type $\tilde{K_1}, \ldots, \tilde{K_s}$, where
$K_i$ is $A_m$, $D_n$ or $E_k$. The rank of a parabolic subdiagram is the sum of the ranks of its connected components.  Note that the dual graph of reducible singular fibers of a genus 1 fibration on 
$X$ gives a parabolic subdiagram.  
We denote by $W(\Gamma)$ the subgroup of ${\rm O}(L)$ 
generated by reflections associated with $\delta \in \Gamma$.

\begin{theorem}\label{Vinberg}{\rm (Vinberg \cite[Theorem 2.3]{V})}
Let $\Delta$ be a set of $(-2)$-vectors in $L$
and let $\Gamma$ be the graph of $\Delta$.
Assume that $\Delta$ is a finite set, $\Gamma$ is non-degenerate and $\Gamma$ contains no $m$-tuple lines with $m \geq 3$.  Then $W(\Gamma)$ is of finite index in ${\rm O}(L)$ 
if and only if every connected parabolic subdiagram of $\Gamma$ is a connected component 
of some parabolic subdiagram in $\Gamma$ of rank $r-1$ {\rm (}= the maximal one{\rm )}.
\end{theorem} 
\noindent

\begin{remark}\label{Vinbergremark}
Note that $\Gamma$ as in the above proposition is automatically non-degenerate if it contains the components of the reducible fibers of a special extremal genus 1 fibration and a special $2$-section of this fibration. Indeed, these curves will generate $L\otimes {\bf Q}$ and hence $K(\Gamma)$ is strictly convex.
\end{remark}

\begin{prop}\label{Namikawa}{\rm (Cossec, Dolgachev, Liedtke \cite[Proposition 0.8.16]{CDL})}
Let $\Delta$ be a finite set of irreducible effective roots on a Coble surface $S$ and let
$\Gamma$ be the graph of $\Delta$. 
Assume that $W(\Gamma)$ is of finite index in ${\rm O}({\rm CM}(S))$.  Then $\Delta$ is the set of all irreducible effective roots on $S$.
\end{prop}

\section{The canonical covering of a Coble surface in characteristic 2}\label{sec3}
\subsection{The canonical covering}\label{sec4-1}
In this subsection, we examine the structure of the canonical covering of a Coble surface
in characteristic 2.
For a Coble surface $S$, let $\{U_i\}_{i \in I}$ be an affine open covering of $S$,
and we denote by $A_i$ the affine coordinate ring of $U_i$. Let $f_i$ be a defining
equation of the divisor $B_1 + \cdots +B_n$. Taking an enough fine affine open covering,
we may assume that each affine open set $U_i$ meets at most one component of the boundary
and that $f_i$ is one element of local coordinates of $U_i$ when the boundary meets $U_i$. 
Let $\{f_{ij}\}$ be a cocycle
which gives the invertible sheaf $\mathcal{O}_S(-K_S)$. Then, by definition
there exist $g_i \in \Gamma (U_i, {\mathcal O}_S^*)$ ($i\in I$) such that
\begin{equation}\label{cycle}
     f_{ij}^2 =(f_i/f_j)(g_i/g_j).
\end{equation}
We define the covering $\tilde{U}_i$ of $U_i$ by
\begin{equation}\label{covering}
       z_i^2 = f_ig_i
\end{equation}
with new variables $z_i$ ($i \in I$).
Then, we have $z_i = f_{ij}z_j$ and we get a finite flat covering $X$ of $S$:
$$
     \pi : X \longrightarrow S.
$$
We call $X$ a canonical covering of $S$. 

\begin{theorem}\label{K3-like} 
The canonical covering $X$ is a K3-like surface, that is, $\omega_X\cong {\mathcal O}_X$ and ${\rm H}^1(X, {\mathcal O}_X) = 0$.
\end{theorem}
\begin{proof}
We denote by $\omega_S$ the dualizing sheaf of $S$. We have $\omega_S \cong {\mathcal O}_S(K_S)$.
The coordinate ring of $\tilde{U}_i$ is isomorphic to
$$
      A_i[z_i]/(z_i^2 -f_ig_i)\cong A_i \oplus A_iz_i.
$$
Since $z_j/z_i = f_{ij}^{-1}$, we see
\begin{equation}\label{direct}
    \pi_{*}{\mathcal O}_{X}\cong {\mathcal O}_S \oplus \omega_S.
\end{equation}
By the duality theorem of finite flat morphism, we have
$$
\begin{array}{rl}
  \pi_{*}(\omega_X) &\cong \pi_{*}(\pi^{!}\omega_S) 
          \cong {\mathcal H}om_{{\mathcal O}_S}(\pi_{*}{\mathcal O}_X, \omega_S) \cong 
           {\mathcal H}om_{{\mathcal O}_S}({\mathcal O}_S \oplus \omega_S, \omega_S)\\
           & \cong {\mathcal H}om_{{\mathcal O}_S}({\mathcal O}_S, \omega_S) \oplus 
            {\mathcal H}om_{{\mathcal O}_S}(\omega_S, \omega_S)\\
     & \cong \omega_S \oplus {\mathcal O}_S \cong \pi_{*}({\mathcal O}_X)
\end{array}.
$$           
Therefore, we have $\omega_X \cong {\mathcal O}_X$. By the Serre duality theorem,
we have
$$
\dim {\rm H}^2(X, {\mathcal O}_X) = \dim {\rm H}^0(X, \omega_X) =
\dim {\rm H}^0(X, {\mathcal O}_X) = 1.
$$
Since $\pi$ is a finite morphism, we have $R^i\pi_{*}{\mathcal O}_X = 0$
($i \geq 1$). Therefore,
we have
$$
\chi({\mathcal O}_X) = \chi(\pi_{*}{\mathcal O}_X) = \chi({\mathcal O}_S \oplus \omega_S) 
= 2\chi({\mathcal O}_S) = 2.
$$
Hence, we have $\dim {\rm H}^1(X, {\mathcal O}_X) = 0$.
\end{proof}

Now, let $g: R \longrightarrow {\bf P}^1$ be a genus 1 fibration on a rational surface $R$.
Let $t$ be a local coordinate of affine line ${\bf A}^1 \subset {\bf P}^1$ and 
assume that over the point $P \in {\bf P}^1$ defined by $t = 0$, we have 
a singular fiber of type
${\rm I}_n$, ${\rm II}$, ${\rm III}$ or ${\rm IV}$.

We take a rational one form $dt/t$ on ${\bf P}^1$ and we examine the poles of 
$g^{*}(dt/t)$ under the blowing-ups to make $(-4)$-curves for a Coble surface.

In the case of type ${\rm I}_n$, we blow up at the singlar points of $g^{-1}(P)$.
Then, the pull-back of $g^{*}(dt/t)$ has poles of order 1 along the proper transforms of
$g^{-1}(P)$, which are $(-4)$-curves, and it is regular along the exceptional curves
(see Figure \ref{Halsing}).

In the case of type ${\rm II}$, we blow up at the singlar point of the curve $g^{-1}(P)$
with a cusp.
Then, the pull-back of $g^{*}(dt/t)$ has a pole of order 1 along the proper transform of
$g^{-1}(P)$, which are $(-4)$-curve, and it is regular along the exceptional curve.

In the case of type ${\rm III}$, we blow up at the singlar point of $g^{-1}(P)$.
Then, the pull-back of $g^{*}(dt/t)$ has poles of order 1 along the proper transforms of
$g^{-1}(P)$, which are $(-3)$-curves, and it has not a pole
along the exceptional curve.
We again blow up at the singular point of the fiber (non-irreducible) curve.
Then, we have two $(-4)$-curves, where the pull back of the rational 1-form
has poles of order 1, and no poles along other exceptional curves.

In the case of type ${\rm IV}$, we blow up at the singlar point of $g^{-1}(P)$.
Then, the pull-back of $g^{*}(dt/t)$ has poles of order 1 along the proper transforms of
$g^{-1}(P)$, which are three $(-3)$-curves, and also it has a pole of order 1 
along the exceptional curve.
We again blow up at the three singular points of the fiber (non-irreducible) curve.
Then, we have four $(-4)$-curves, where the pull back of the rational 1-form
has poles of order 1, and no poles along other exceptional curves.

\begin{lemma}\label{1-form1}
Let $S$ be a Coble surface and let 
$f: S \longrightarrow {\bf P}^1$ be a genus $1$ fibration. 
Let $t$ be a local coordinate of ${\bf A}^1 \subset {\bf P}^1$ and
assume that for the point $P_0$ defined by $t = 0$ the singular fiber
$f^{-1}(P_0)$ contains at least one $(-4)$-curve. Then, on the fiber $f^{-1}(P_0)$
the rational $1$-form $f^{*}(dt/t)$ has poles of order $1$ along the $(-4)$-curves and is regular on other part of $f^{-1}(P_0)$.
Namely, the rational $1$-form $f^{*}(dt/t)$ has poles of order $1$ along the boundary
contained in $f^{-1}(P_0)$ and is regular on other part of $S \setminus f^{-1}(\infty)$.
\end{lemma}

\begin{proposition}\label{1-form2}
Let $S$ be a Coble surface and let $f: S \longrightarrow {\bf P}^1$ be a genus $1$
fibration. 
Under the notation above, we assume that at the point $P_0$ defined by $t = 0$, 
the fiber $f^{-1}(P_0)$ is either a multiple fiber or a singular fiber with at least
one $(-4)$-curve, and that at the point $P_{\infty}$ defined by $t = \infty$ 
the fiber $f^{-1}(P_{\infty})$ is a singular fiber with at least one $(-4)$-curve.
Then. $f^{*}(dt/t)$ has poles of order $1$ along $(-4)$-curves and is regular
on other part.
\end{proposition}
\begin{proof}
Assume $f : S \longrightarrow {\bf P}^1$ has a multiple fiber at the point $P_0$.
Then, we have an expression $t = uh^2$.
Here, $h$ is a defining equation of the half fiber and $u$ is a unit. 
Therefore, we have
$f^{*}(dt/t) = du/u$ and  $f^{*}(dt/t)$ is regular on the multiple fiber.
Using this fact and Lemma \ref{1-form1}, we complete our proof.
\end{proof}

Now, we define
$$
    \eta = d(f_ig_i)/(f_ig_i)\quad \mbox{on}~U_i.
$$
Then, by (\ref{cycle}), $\eta$ becomes a rational 1-form on $S$.
The rational 1-form $\eta$ has the poles of order 1
along the boundary $B_1 + \cdots + B_n$. 
When $f_i$ is part of local coordinates, $g_if_i$ is also 
part of local coordinates.
Therefore, $\eta$ has neither isolated zeros nor zero divisors on 
the boundary $B_1+ \cdots + B_n$. Therefore, over the boundary 
the covering $X$ is non-singular and the divisorial part of the rational 1-form $\eta$
is written as
\begin{equation}\label{B}
 (\eta) = - \sum_{i = 1}^{n} B_i + B
\end{equation}
with an effective divisor $B$. Note $(\sum_{i = 1}^{n} B_i) \cap B = \emptyset$.
We call $B$ a bi-conductrix.

\begin{theorem}\label{forms}
$\eta = f^{*}(dt/t)$
\end{theorem}
\begin{proof}
Let a boundary component $B_k$ intersect with affine open set $U_i$. Then, we have
$$
\eta = d(f_ig_i)/(f_ig_i) = df_i/f_i + dg_i/g_i.
$$
Since the boundary component $B_k$ is defined by $f_i = 0$ on $U_i$ and is contained
in a fiber as a component with odd multiplicity, we have $t = u_if_i$ with a unit $u_i$.
Therefore, we have
$$
  f^{*}(dt/t) = d(u_if_i)/(u_if_i) = df_i/f_i + du_i/u_i.
$$
Therefore, on $U_i$ we have
$$
 \eta - f^{*}(dt/t) = dg_i/g_i - du_i/u_i,
$$
and $\eta - f^{*}(dt/t)$ is a regular 1-form on $S$.
However, since $S$ is a rational surface, we have no non-zero regular 1-form on $S$.
Therefore, we have $\eta = f^{*}(dt/t)$.
\end{proof}
 
From here on, for the 1-form $f^{*}(dt/t)$ on $S$, we sometimes write it as $dt/t$
if no confusion occurs.

\subsection{The conductrix}\label{sec4-2}
In this subsection, we examine the bi-conductrix $B$ and we show that 
there exists an effective divisor $A$ on $S$ such that $B = 2A$. The arguments
are parallel to the one in Ekedahl, Shepherd-Barron \cite{ES} and
Cossec, Dolgachev, Liedtke \cite{CDL}. But we give here a down-to-earth
proof to understand the case of Coble surfaces precisely.
The following lemma is well-known (cf. Cossec, Dolgachev, Liedtke, loc. cit., for instance).
\begin{lemma}\label{Cohen-Macauley}
Let $S$, $X$ be algebraic surfaces and let $f : X \longrightarrow S$ be a finite
morphism. Assume $S$ is regular. Then, $f$ is flat if and only if $X$ is Cohen-Macauley.
Moreover, if $X$ is normal, 
% and $\dim X \leq 2$
then $X$ is Cohen-Macauley.
\end{lemma}

Let $S$ be a Coble surface and let $X$ be the canonical covering. Let $X'$ be
the normaization of $X$:
$$
\begin{array}{ccr}
   X' & \stackrel{\nu}{\longrightarrow} & X \\
   \pi \circ \nu \searrow &     &\swarrow \pi \\
     &   S.   &
\end{array}
$$
Since $X'$ is normal, 
$X'$ is Cohen-Macauley
and $\pi\circ \nu$ is flat by Lemma \ref{Cohen-Macauley}. Therefore, we have an
exact sequence
\begin{equation}\label{exactL'}
0 \longrightarrow {\mathcal O}_S \longrightarrow (\pi \circ \nu)_{*}{\mathcal O}_{X'}
\longrightarrow {\mathcal L}' \longrightarrow 0
\end{equation}
with an invertible sheaf ${\mathcal L}'$ on $S$. Therefore, by (\ref{direct})
we have a commutative diagram of exact sequences
$$
\begin{array}{ccccccccc}
  0 & \longrightarrow & {\mathcal O}_S &\longrightarrow &\pi_{*}{\mathcal O}_X & \longrightarrow & \omega_S &\longrightarrow & 0 \\
  
        &   & \parallel & & \varphi \downarrow & & \psi \downarrow & & \\
 0 &\longrightarrow &{\mathcal O}_S &\longrightarrow &(\pi \circ \nu)_{*}{\mathcal O}_{X'}&
\longrightarrow &{\mathcal L}'& \longrightarrow & 0.
\end{array}
$$

Since $\pi \circ \nu$ is a purely inseparable morphism of degree 2, 
on the affine open set $U_i$ we have a commutative diagram of exact sequences
$$
\begin{array}{ccccccccc}
   0 & \longrightarrow & A_i &\longrightarrow & A_i \oplus A_iz_i& \longrightarrow & A_iz_i&\longrightarrow & 0 \\
  
        &   & \parallel & & \varphi \downarrow & & \psi \downarrow & & \\
 0 &\longrightarrow & A_i &\longrightarrow &  A_i \oplus A_iw_i&
\longrightarrow & A_iw_i & \longrightarrow & 0.
\end{array}
$$     
Here, $w_i$ is an element such that $w_i^2 = h_i \in A_i$. Let $\varphi(z_i) = b_i + a_iw_i$ ($a_i, b_i \in A_i$).
Since $a_i w_i = \varphi(z_i) - b_i = \varphi(z_i -b_i)$, the annihilator
$Ann_{A_i}(A_i \oplus A_iw_i)/\varphi(A_i\oplus A_iz_i)$ is given by the ideal
$(a_i) = a_iA_i$. The Cartier divisors on $U_i$ defined by the ideal $(a_i)$ ($i \in I$)
glue together to give a divisor $A$ on $S$. We call $A$ a conductrix of $S$.
Since $\psi$ is given by the multiplication by $a_i$ on $U_i$, $\psi$ induces
an isomorphism $\omega_S \cong {\mathcal O}_S(-A)\otimes {\mathcal L}'$.
Therefore, we have
\begin{equation}\label{L'}
      {\mathcal L}'\cong \omega_S \otimes {\mathcal O}_S(A).
\end{equation}
\begin{theorem}\label{eta}
  $(\eta) = - \sum_{i= 1}^{n}B_i + 2A$, $B = 2A$. 
\end{theorem}
\begin{proof}
By (\ref{covering}), we have
$$
g_if_i = \varphi (g_if_i) = \varphi (z_i^2) = (b_i + a_iw_i)^2 = b_i^2 + a_i^2w_i^2 = b_i^2 + a_i^2h_i.
$$
Therefore, we have $a_i^2h_i = g_if_i -b_i^2$, and $g_if_i - b_i^2$ is divisible by $a_i^2$.
Since $X'$ is normal, $dh_i$ has only isolated zeros.  Therefore, by
$a_i^2dh_i = d(g_if_i -b_i^2) = d(g_if_i)$, we see $d(g_if_i)$ is exactly divisible
by $a_i^2$. Hence, by (\ref{B}), we see $B = 2A$.
\end{proof}

\subsection{The conductrix of the surface by Frobenius base change}\label{case1case2}
Let $F : {\bf P}^1 \longrightarrow {\bf P}^1$ be the Frobenius morphism.
Since $\pi^*f^*(dt/t) = \pi^*(\eta) = \pi^*(d(g_if_i)/(g_if_i)) = \pi^*(dz_i^2/z_i^2) = 0$,
we have $\pi^*f^*(dt) = 0$. Therefore, there exists an element $s \in k(X)$
such that $t = s^2$. Since $X'$ is normal, we have the following diagram:
$$
\begin{array}{ccl}\label{diagram1}
  \quad  X' & \stackrel{\nu}{\longrightarrow} & X \\
\psi_0  \downarrow  &       & \downarrow \pi\\
 \quad   S\times_{{\bf P}^1} {\bf P}^1 & \stackrel{pr_1}{\longrightarrow} & S\\
pr_2 \downarrow   &      & \downarrow f  \\
    \quad \quad  {\bf P}^1   & \stackrel{F}{\longrightarrow} & {\bf P}^1.
\end{array}
$$
We set $S_F = S\times_{{\bf P}^1} {\bf P}^1$.
Over $U_i$,  if $U_i$ exists in the open set of $S$  given by $t \neq \infty$,
the covering $pr_1 : pr_1^{-1}(U_i) \longrightarrow U_i$ is given by
the homomorphism
$$
  A_i = {\mathcal O}_X(U_i) \longrightarrow 
  pr_{1*}{\mathcal O}_{X_F}(U_i) = A_i[s]/(t - s^2) 
  = A_i \oplus A_{i}s.
$$
Around the point defined by $t = \infty$, the situation is similar.
Note that $S'$ is the normalization of $S_F$.
We calculate the conductrix for $S_F$.
We have exact sequences
$$
\begin{array}{ccccccccc}
  0 & \longrightarrow & {\mathcal O}_S &\longrightarrow &pr_{1*}{\mathcal O}_{S_F} & \longrightarrow & \quad {\mathcal L} &\longrightarrow & 0 \\
  
        &   & \parallel & & \varphi' \downarrow \quad \quad  & & \psi' \downarrow  & & \\
 0 &\longrightarrow &{\mathcal O}_S &\longrightarrow &(pr_1\circ \psi_0)_{*}{\mathcal O}_{S'}&
\longrightarrow & \quad  {\mathcal L}'& \longrightarrow & 0\\
     & & & & \parallel & & & &  \\
     & & & & (\pi \circ \nu)_{*}{\mathcal O}_{S'} & & & &  
\end{array}
$$  
On $U_i$ we have exact sequences
$$
\begin{array}{ccccccccc}
   0 & \longrightarrow & A_i &\longrightarrow & A_i \oplus A_{i}s& \longrightarrow & A_{i}s&\longrightarrow & 0 \\
  
        &   & \parallel & & \varphi' \downarrow & & \psi' \downarrow & & \\
 0 &\longrightarrow & A_i &\longrightarrow &  A_i \oplus A_iw_i&
\longrightarrow & A_iw_i & \longrightarrow & 0.
\end{array}
$$ 
Let $\varphi'(s) = b'_i + a'_iw_i$ ($a'_i, b'_i \in A_i$).
Since $a'_i w_i = \varphi'(s) - b'_i = \varphi'(s -b'_i)$, the annihilator
$Ann_{A_i}(A_i \oplus A_iw_i)/\varphi'(A_i\oplus A_is)$ is given by the ideal
$(a'_i) = a'_iA_i$. The Cartier divisors on $U_i$ defined by the ideal $(a'_i)$ ($i \in I$)
glue together to give a divisor $A'$ on $S$. We call $A'$ a conductrix of $S_F$.
As before, we have
$$
t = \varphi' (t) = \varphi' (s^2) = (b'_i + a'_iw_i)^2 = {b'_i}^2 + {a'_i}^2w_i^2 
= {b'_i}^2 + {a'_i}^2h_i.
$$
Therefore, we have ${a'_i}^2h_i = t -{b'_i}^2$, and $t - {b'_i}^2$ is divisible by ${a'_i}^2$.
Since $S'$ is normal, $dh_i$ has only isolated zeros.  Therefore, by
${a'_i}^2dh_i = d(t -{b'_i}^2) = dt$, we see $dt$ is exactly divisible
by ${a'_i}^2$. 

We denote by $P_0$ (resp. $P_{\infty}$) the point on
${\bf P}^1$ defined by $t = 0$ (resp. $t = \infty$). We have the following two cases:

(Case 1) Halphen surface of index 1. 
Both fibers $f^{-1}(P_0)$ and $f^{-1}(P_{\infty})$ are the fibers
which contain $(-4)$-curves for the Coble surface $S$.

(Case 2) Halphen surface of index 2. 
The fiber $f^{-1}(P_0)$ is the multiple fiber and the fiber $f^{-1}(P_{\infty})$
is the fiber which contains $(-4)$-curves for the Coble surface $S$.

In fact, for any Coble surface, by a suitable choice of the local parameter $t$,
either (Case 1) or (Case 2) holds.

Case 1. By the structure of blowing-ups which we explain in the beginning 
of Subsection \ref{sec4-1},
the fiber $f^{-1}(P_0)$ is expressed as 
$$
     \sum_{i = 1}^{\ell}B_i + 2G_0
$$
with a positive interger $\ell  < n$ and an effective divisor $G_0$. 
Here, we arrange some boundary components $B_i$ suitably.
The fiber $f^{-1}(P_{\infty})$ is also expressed as

$$
     \sum_{i = \ell + 1}^{n}B_i + 2G_{\infty}
$$
with an effective divisor $G_{\infty}$. Then, using these notation, we have
\begin{equation}\label{f^{*}(dt/t)}
       (f^{*}(dt/t)) = 2A'  - \sum_{i = 1}^{n}B_i - 2(G_0 + G_{\infty}).
\end{equation}
On the other hand, we have 
\begin{equation}\label{dt/t=eta}
    (f^{*}(dt/t)) = (\eta) = 2A - \sum_{i=1}^{n}B_i.
\end{equation}
Therefore, we have $A = A' - G_0 - G_{\infty}$.

Case 2. We denote by $G_0$ the half fiber of the fiber defined by $t = 0$.
The fiber defined by $t = \infty$ is also expressed as

$$
     \sum_{i = 1}^{n}B_i + 2G_{\infty}
$$
with an effective divisor $G_{\infty}$.
Then, we have  again the equalities (\ref{f^{*}(dt/t)}) and (\ref{dt/t=eta}).
Therefore, we have $A = A' - G_0 - G_{\infty}$.
Summarizing these results, we get the following theorem (cf. Remark \ref{G_0,G} below).
\begin{theorem}[Comparisson theorem]
Under the notation above, $A = A' - G_0 - G_{\infty}$ holds.
\end{theorem}

\begin{remark}
In the case of an Enriques surface $Y$ in characteristic 2, 
we can consider a similar situation.
If $Y$  is classical, the genus 1 fibration has two multiple fibers
$2G_0$ and $2G_{\infty}$. In this case we have $A = A' - G_0 - G_{\infty}$.
If $Y$ is supersingular, then the genus 1 fibration has only one multiple fiber
$2G_0$. In this case we have $A = A' - 2G_0$. Therefore, $A$ is equal
to $A'$ modulo half fibers and we can determine $A$ from $A'$ uniquely.
\end{remark}

\begin{remark}\label{G_0,G} 
We give here an explicit form of $G_0$ and $G_{\infty}$ for the fiber with $(-4)$-curves. For the non-half fiber case, we write $G_0$ and $G_{\infty}$ as $G$ symbolically. 
The configuration of fibers which contain $(-4)$-curves are given in Figure 1.
For each case, we write up $G$.  
In case of type ${\rm I}_n$, $G = \sum_{i = 1}^{n}E_i$.
In case of type ${\rm II}$, $G=E$.
In case of type ${\rm III}$, $G = E_1 + 2E_2$.
In case of type ${\rm IV}$, $G = B_4 + 2E_1 + 2E_2 + 2E_3$. 
\end{remark}   

\subsection{The singularities of the canonical covering}\label{sec4-3}
We first compare the canonical covering $X$ with the surface 
$S_F = S\times_{{\bf P}^1} {\bf P}^1$
which is obtained by the Frobenius pull-back of $S$ as in the previous section.
In this section, for the local parameter $t$ on ${\bf A}^1 \subset {\bf P}^1$
we assume either (Case 1) or (case 2) in Subsection \ref{case1case2}, if otherwise mentions.

\begin{proposition}\label{compare-surface}
Let $\pi : X \longrightarrow S$ be a canonical covering of a Coble surface $S$, and let $f : S \longrightarrow {\bf P}^1$ be a genus $1$ fibration. 
Under the assumption for the parameter $t$, $\pi$ factors through the Frobenius pull-Back
$S_F = S \times_{{\bf P}^1}{\bf P}^1$ and $X$ is isomorphic to $S_F$
except two fibers defined by $t = 0$ and $t = \infty$.
\end{proposition}
\begin{proof}
We have $\pi_{*}{\mathcal O}_{X} \cong {\mathcal O}_S \oplus \omega_S$.  Since
$f_{*}\omega_S \cong {\mathcal O}_{{\bf P}^1}(-1)$, we have
$$
f_{*}\pi_{*}{\mathcal O}_{X} \cong f_{*}{\mathcal O}_S \oplus f_{*}\omega_S\cong 
{\mathcal O}_{{\bf P}^1} \oplus {\mathcal O}_{{\bf P}^1}(-1).
$$
Take an affine open subset $U$ of ${\bf P}^1$ and let $R$ be the coordinte ring of $U$:
$U = {\rm Spec}~R$. Then we have
$$
   f_{*}\pi_{*}{\mathcal O}_{X}(U) \cong R \oplus Rz  
   \quad (\exists z \in f_{*}\pi_{*}{\mathcal O}_{X}(U)).
$$   
Since $\pi$ is purely inseparable of degree 2, $z$ is not separable over $R$. Therefore
we have either $z^2 = a$ ($\exists a \in R$, $a\neq 0$)  or 0. Since $Y$ is reduced,
the case $z^2 = 0$ doesn't occur. Therefore, we have $z^2 = a$ ($\exists a \in R$, $a\neq 0$).

Put $C = {\rm Spec}~f_{*}\pi_{*}{\mathcal O}_{X}$. Then, we have a commutative diagram
$$
\begin{array}{ccc}
     X   &  \stackrel{\pi}{\longrightarrow}   &  S  \\
     \downarrow  &   &   \downarrow f  \\
          C  & \stackrel{F}{\longrightarrow} & {\bf P}^1.
\end{array}
$$
The morphism $F$ becomes the Frobenius morphism and $f\circ \pi$ factors through 
the Frobenius morphism. 
It is easy to see that 
$F_{*}{\mathcal O}_{{\bf P}^1} \cong {\mathcal O}_{{\bf P}^1} \oplus {\mathcal O}_{{\bf P}^1}(-1)$.
Therefore, for $pr_1 : S_F \longrightarrow S$, we have
\begin{equation}\label{OX_F}
pr_{1*}{\mathcal O}_{S_F} 
\cong {\mathcal O}_S \otimes_{{\mathcal O}_{{\bf P}^1}}({\mathcal O}_{{\bf P}^1} \oplus {\mathcal O}_{{\bf P}^1}(-1)) \cong {\mathcal O}_S \oplus f^{*}{\mathcal O}_{{\bf P}^1}(-1).
\end{equation}
Put $V = S \setminus \{f^{-1}(P_0) \cup f^{-1}(P_{\infty})\}$. 
Then, on the open set $V$, we have
\begin{equation}\label{OY}
    \pi_{*}{\mathcal O}_X\vert_{V} \cong ({\mathcal O}_S \oplus \omega_S)\vert_{V} 
    \cong ({\mathcal O}_S \oplus f^{*}{\mathcal O}_{{\bf P}^1}(-1))\vert_{V}.
\end{equation}
Comparing (\ref{OX_F}) with (\ref{OY}), we get the proof of the latter part.
\end{proof}

We state here some conditions for the non-emptiness of the conductrix $A$.
\begin{lemma}\label{hascusp}
If the genus $1$ fibration $f : S \longrightarrow {\bf P}^1$ is quasi-elliptic,
then the conductrix $A$ contains the curve of cusps as a simple component.
In particular, $A \neq 0$.
\end{lemma}
\begin{proof}
The general fiber is of the form $t = y^2 + x^3$ (cf. Bombieri and Mumford \cite{BM}).
Therefore, by $dt/t = x^2dx/t$ we see that divisorial part $(\eta) = (dt/t) $
 contains the curve of cusps as a double component. Therefore, $A$ contains
 the curve of cusps as a simple component.
\end{proof}

\begin{lemma}\label{Dynkin}
Let $Z$ be a fundamental divisor of ADE configulation on a Coble
surface $S$. Then,
$$
   {\rm H}^0(Z, {\mathcal O}_Z) = 1, ~{\rm H}^1(Z, {\mathcal O}_Z) = 0.
$$
\end{lemma}
\begin{proof}
Since the lattice generated by the irreducible components of $Z$
in the N\'eron-Severi group ${\rm NS}(S)$ is negative definite, we have
${\rm H}^0(S, {\mathcal O}_S(K_S + Z)) = 0$. Therefore, by the Serre duality
we have ${\rm H}^2(S, {\mathcal O}_S(-Z)) = 0$.
Since $Z$ is a fundamental divisor of ADE configulation,
we have $Z^2 = -2$, and since any $(-2)$-curve does not intersect the boundary
divisor $\sum_{i= 1}^n B_i$, we have $K_S\cdot Z = 0$.
Therefore, by the Riemann-Roch theorem we have ${\rm H}^1(S, {\mathcal O}_S(-Z)) = 0$.
Considering the exact sequence
$$
0\longrightarrow {\mathcal O}_S(-Z) \longrightarrow {\mathcal O}_S \longrightarrow {\mathcal O}_Z
\longrightarrow 0,
$$
we have a long exact sequence
$$
\begin{array}{c}
 0 \longrightarrow k  \longrightarrow  {\rm H}^0(Z, {\mathcal O}_Z) \longrightarrow 
 {\rm H}^1(S, {\mathcal O}_S(-Z)) \\
 \longrightarrow {\rm H}^1(S, {\mathcal O}_S)
 \longrightarrow {\rm H}^1(Z, {\mathcal O}_Z) \longrightarrow {\rm H}^2(S, {\mathcal O}_S(-Z)).
\end{array}
$$
Since $S$ is rational, we have ${\rm H}^1(S, {\mathcal O}_S)= 0$.
Therefore, we have the results.
\end{proof}

\begin{proposition}\label{SchroeerThm}
Assume $S$ contains a divisor whose the dual graph is of type $D_4$.  Then the central component is contained in $A$.  In particular, $A\ne 0$.
\end{proposition}
\begin{proof}
Let $E_0, E_1, E_2, E_3$ be the components of $Z$ with $E_0\cdot E_i=1$ $(i=1,2,3)$.
Then the divisor $Z=2E_0+E_1+E_2+E_3$ is a fundamental divisor of type $D_4$.
By Lemma \ref{Dynkin}, it is rational type in the sence of Schr\"oer \cite{S}.
It now follows from Schr\"oer \cite[Proposition 4.6]{S} that $E_0 \subset A$.
\end{proof}

\begin{corollary}\label{non-emptyA}
If a genus $1$ fibration $f : S \longrightarrow {\bf P}^1$ has 
a singular fiber of type ${\rm I}_m^*$, ${\rm II}^{*}$,
${\rm III}^{*}$ or ${\rm IV}^{*}$,
then we have the conductrix $A\neq 0$.
\end{corollary}
\begin{proof}
Obviously each of these singular fibers contains a fundamental divisor of type $D_4$.
Thus the assertion follows from Proposition \ref{SchroeerThm}.
\end{proof}

\begin{lemma}\label{rational elliptic}
Let $g : Y \longrightarrow {\bf P}^1$ be a relatively minimal rational elliptic surface
without multiple fiber, and let $F : {\bf P}^1 \longrightarrow {\bf P}^1$ 
be the Frobenius morphism.
Then, the Frobenius pull-back $Y_F = Y \times_{{\bf P}^1} {\bf P}^1$
is a K3-like surface.
\end{lemma}
\begin{proof}
We consider the cartesian diagram
$$
\begin{array}{lccc}
      Y_F = & Y \times_{{\bf P}^1} {\bf P}^1 & \stackrel{pr}{\longrightarrow} & Y \\
            & \downarrow   &       &   \downarrow g \\
            & {\bf P}^1 &   \stackrel{F}{\longrightarrow}  & {\bf P}^1.
\end{array}
$$
Let $g^{-1}(P)$ be a fiber of $g$.
It is easy to see that
$F_{*}{\mathcal O}_{{\bf P}^1} = {\mathcal O}_{{\bf P}^1} \oplus {\mathcal O}_{{\bf P}^1}(-1)$.
Therefore, we have
$$
      pr_{*}{\mathcal O}_{Y_F} \cong 
({\mathcal O}_{{\bf P}^1} \oplus {\mathcal O}_{{\bf P}^1}(-1))\otimes_{{\mathcal O}_{{\bf P}^1}}{\mathcal O}_{Y} = {\mathcal O}_{Y}\oplus {\mathcal O}_{Y}(-g^{-1}(P)).
$$
Since the dualizing sheaf $\omega_Y$ of $Y$ is given by ${\mathcal O}_{Y}(-g^{-1}(P))$ and
the morphism $pr$ is a finite morphism, the dualizing sheaf $\omega_{Y_F}$ of $Y_F$ is given by
$$
\begin{array}{cl}
   pr_{*}(\omega_{Y_F}) &\cong  {\mathcal H}om_{{\mathcal O}_Y}(pr_{*}{\mathcal O}_{Y_F}, \omega_Y)\\
     &\cong {\mathcal H}om_{{\mathcal O}_Y}({\mathcal O}_{Y}\oplus {\mathcal O}_{Y}(-g^{-1}(P)), {\mathcal O}_{Y}(-g^{-1}(P)))\\
     & \cong {\mathcal O}_{Y}(-g^{-1}(P))\oplus {\mathcal O}_{Y} \cong pr_{*}({\mathcal O}_{Y_F})
\end{array}
$$
Therefore, we have $\omega_{Y_F}\cong {\mathcal O}_{Y_F}$. Since 
$$
\chi ({\mathcal O}_{Y_F}) = \chi(pr_{*}({\mathcal O}_{Y_F}))
= \chi({\mathcal O}_Y) + \chi({\mathcal O}_{Y}(-g^{-1}(P)))= 2\chi({\mathcal O}_Y) = 2,
$$
we have ${\rm H}^1 (Y_F, {\mathcal O}_{Y_F}) = 0$. Therefore, $Y_F$ is 
a K3-like surface. 
\end{proof} 

\begin{lemma}\label{non-rational singularity}
Let $X$ be a normal K3-like surface. If $X$ has a non-rational singularity,
then $X$ is a rational surface.
\end{lemma}
\begin{proof}
Let $\mu : \tilde{X} \longrightarrow X$ be a resolution of singularities.
We denote by $\cup_{i= 1}^kE_i$ the support of the exceptional divisors of the resolution $\mu$.
Here, $E_i$'s are the irreducible components. Then, by Hidaka, Watanabe \cite[Lemma 1.1]{HW}
the dualizing sheaf of $\tilde{X}$ is given by 
$\omega_{\tilde{X}}\cong {\mathcal O}_{\tilde{X}}(- \sum_{i = 1}^{k}r_i E_i)$
and some $r_i$ is positive by the fact that one of singularities is non-rational.
Therefore, we have $\dim {\rm H}^0(\tilde{X}, \omega_{\tilde{X}}^m) = 0$
for $m > 0$.
Since ${\rm H}^1(X, {\mathcal O}_X) = 0$, we have $q(X) = 0$. Therefore, we have
$q(\tilde{X}) = 0$. By Castelnuovo's criterion of rationality, we see $\tilde{X}$
is a rational surface.
\end{proof}

\begin{proposition}
Assume the conductrix $A = 0$. Then, the canonical
covering $X$ always has some singular points and the singular points are all rational
double points.
\end{proposition}
\begin{proof}
Assume $X$ is non-singular. Then, since $\pi: X \longrightarrow S$ is 
purely inseparable and finite,
we have the Betti number $b_2(X) = b_2(S)\geq 11$, $\omega_X \cong {\mathcal O}_X$
and ${\rm H}^1(X, {\mathcal O}_X) = 0$. Therefore, $X$ must be a K3 surface 
by the classification theory of algebraic surfaces in positive characteristic.
By Proposition \ref{Boundary}, we have $K_S^2 \geq -10$. Therefore, the Betti number
$b_2(X) = b_2(S) \leq 20$, which contradicts the topological invarinant of K3 surfaces.
 
Since the covering $X \longrightarrow S$ is locally define by (\ref{covering}),
the singular points are double points.
By $A = 0$, $X$ has only isolated singularities. Therefore, $X$ is normal.
Let $f: S \longrightarrow {\bf P}^1$ be a genus 1 fibration. If this fibration
is quasi-elliptic, the conductrix $A$ contains the curve of cusps, which contradicts $A = 0$.
Therefore, $f : S \longrightarrow {\bf P}^1$ is an elliptic surface.
We may assume that this fibration has a 2-section.

Suppose that one of isolated singularities is non-rational.
We consider the Frobenius base change
$$
\begin{array}{ccc}
     S_{F}   &  \stackrel{\pi}{\longrightarrow}   &  S   \\
    \tilde{f} \downarrow  &   &   \downarrow f  \\
     
     {\bf P}^1  & \stackrel{F}{\longrightarrow} & {\bf P}^1.
\end{array}
$$
Since $S_{F}$ is birationally equivalent to $X$ which 
has a non-rational singular point, $S_F$
is a rational surface by Lemma \ref{non-rational singularity}. 
Now, we consider the relative Jacobian surface 
$J(S) \longrightarrow {\bf P}^1$ and the Frobenius base change
$J(S)_F$. We may assume that $J(S) \longrightarrow {\bf P}^1$ 
is relatively minimal. Since $J(S)$ is a rational surface by Lemma \ref{rational elliptic}, 
$J(S)_F$ is a K3-like surface. Since $X$ is normal, that is, the conductrix $A$
for $S$ is empty, by the same argument for the classical Enriques surfaces
in Matsumoto \cite[Proposition 4.3]{Matsumoto} we see that $S_{F}$ is 
birationally equivalent to $J(S)_F$ and so $J(S)_F$ is rational 
and that $J(S)_F$ has non-rational
singularity. Since $J(S) \longrightarrow {\bf P}^1$ is a
rational elliptic surface, we have a list of singular 
fibers in Lang \cite{La3}. By a result of Schr\"oer \cite{S}, the only case
where a non-rational singular point appears in $J(S)_F \longrightarrow {\bf P}^1$ 
is only Case 9C in Lang \cite{La3}.
In this case, the singular fiber is of type II with discriminant 12. 
Therefore, the rational ellptic surface $J(S) \longrightarrow {\bf P}^1$ 
has only one singular fiber of type II
and all other fibers are supersingular elliptic curves by Lang \cite{La3}.
Since the types of fibers of the relatively minimal model of
$f : S \longrightarrow {\bf P}^1$ are same
as the ones of fibers of $J(S) \longrightarrow {\bf P}^1$, and since 
a multiple fiber of Coble surface is tame, any supersingular 
elliptic curve cannot become a multiple fiber.
Therefore,
we have no chance to make a Coble surface
by blowing up the relatively minimal elliptic rational surface. Hence,
if $A = 0$, then the canonical covering of the Coble surface $S$ has
only rational singular points.
\end{proof}

\begin{corollary}
Assume the conductrix $A = 0$. 
Then, the minimal non-singular model $\tilde{X}$ of the canonical
covering $X$ is a K3 surface.
\end{corollary}
\begin{proof}
Let $\mu :\tilde{X} \longrightarrow X$ be the minimal resolution of singularities.
This corollary follows from $\omega_{\tilde{X}}\cong \mu^{*}\omega_X \cong \mu^{*}{\mathcal{O}_X}$
and ${\rm H}^1(\tilde{X}, {\mathcal O}_{\tilde{X}}) \cong 
{\rm H}^0(X, R^1\mu_{*}{\mathcal O}_{\tilde{X}}) = 0$.
\end{proof}

Now, we consider the case with the conductrix 
$A \neq 0$.
\begin{proposition}\label{singularityY'}
Assume the conductix $A \neq 0$.
Then, the singularities of the normalization $X'$ of the canonical covering $X$
are all rational, and $X$ is a rational surface.
\end{proposition}
\begin{proof} 
The proof is similar to the one for Enriques surfaces (cf. Cossec, Dolgachev, Liedtke \cite{CDL}).
We denote by $\tilde{X}$ the minimal resolution of singularities of $X'$:
$$
\begin{array}{ccccc}
\tilde{X} & \stackrel{\mu}{\longrightarrow} & X'  & \stackrel{\nu}{\longrightarrow} & X \\
   &   & \searrow &    & \swarrow \pi \\
      &  &  & S & 
\end{array}
$$
By Lemma \ref{Cohen-Macauley}, $\pi \circ \nu$ is a flat morphism.
Therefore, by Cossec, Dolgachev, Liedtke \cite[Proposition 0.2.22]{CDL},
$X'$ is Gorenstein and again by Cossec, Dolgachev, Liedtke [Proposition 0.28, loc. cit.]
and (\ref{L'}), we have
\begin{equation}\label{dualizingY'}
\omega_{X'} \cong (\pi\circ \nu)^{*}(\omega_S \otimes {{\mathcal L}'}^{-1})
\cong (\pi \circ \nu)^{*}{\mathcal O}_S(-A).
\end{equation}
We denote by $\cup_{i = 1}^{k}E_i$ the support of the exceptional divisors
of the resolution $\mu$. Here, $E_i$'s are the irreducible components.
Then, by Hidaka, Watanabe \cite[Lemma 1.1]{HW}, we have
$$
\begin{array}{cl}
     \omega_{\tilde{X}} &\sim \mu^{*}\omega_{X'} - \sum_{i= 1}^{k}r_iE_i   \quad (r_i \geq 0)\\
              & \sim \mu^{*}(\pi \circ \nu)^{*}{\mathcal O}_S(-A) - \sum_{i= 1}^{k}r_iE_i
\end{array}
$$
Therefore, we have the pluri-genus $P_m({\tilde{X}}) = 0$ for $m \geq 1$.
Since $\pi\circ \nu \circ \mu$ is a purely inseparable morphism of degree 2,
there exists a rational map $S^{(1/2)} \longrightarrow \tilde{X}$. By the universality
of Albanese variety, there exists a surjective morphism from the Albanese variety
${\rm Alb}(S^{(1/2)})$ to ${\rm Alb}(\tilde{X})$. Since $S$ is rational, $S^{(1/2)}$ is also
rational and so ${\rm Alb}(S^{(1/2)})=0 $. Therefore, we have ${\rm Alb}(\tilde{X})= 0$, 
that is,
the irregularity $q(\tilde{X}) = 0$. Therefore, by Castelnuovo's criterion of rationality
we see $\tilde{X}$ is a rational surface. In particular, we have
${\rm H}^1(\tilde{X}, {\mathcal O}_{\tilde{X}}) = 0$.

Now, we consider the Leray-spectral sequence for $\mu : \tilde{X} \longrightarrow X'$:
$$
  E_{2}^{i,j} = {\rm H}^{j}(X', R^i\mu_{*}{\mathcal O}_{\tilde{X}}) 
  \Longrightarrow {\rm H}^{i + j}(\tilde{X}, {\mathcal O}_{\tilde{X}}).
$$
Then, we have an exact sequence
$$
{\rm H}^1(X', {\mathcal O}_{X'})\longrightarrow {\rm H}^{1}(\tilde{X}, {\mathcal O}_{\tilde{X}})
\longrightarrow {\rm H}^{0}(X', R^1\mu_{*}{\mathcal O}_{\tilde{X}})
\longrightarrow {\rm H}^2(X', {\mathcal O}_{X'}).
$$
Since ${\rm H}^{1}(\tilde{X}, {\mathcal O}_{\tilde{X}}) = 0$ and
$$
{\rm H}^2(X', {\mathcal O}_{X'}) \cong {\rm H}^0(X', \omega_{X'})
\cong {\rm H}^0(X', (\pi \circ \nu)^{*}{\mathcal O}_S(-A))= 0,
$$
we have ${\rm H}^{0}(X', R^1\mu_{*}{\mathcal O}_{\tilde{X}}) = 0$,
that is, $R^1\mu_{*}{\mathcal O}_{\tilde{X}}= 0$. This means that
all singular points of $X'$ are rational.
\end{proof}

\section{Properties of the conductrix}\label{sec4}
Let $S$ be a Coble surface, $\pi : X \longrightarrow S$ the canonical covering
and $f : S \longrightarrow {\bf P}^1$ a genus 1 fibration. 
We assume this fibration satisfies the condition of (Case 1) or (Case 2) in Subsection \ref{case1case2}.
In this section, we assume the conductrix $A \neq 0$,  if otherwise mentions.
As for the general theory of conductrix, in particular for Enriques surfaces
in characteristic 2, we have the theory by Ekedahl and Shepherd-Barron \cite{ES}.
In the case of Coble surfaces, the situation is a bit different from their paper,
and we show the details, although arguments sometimes overlap with the ones 
in Ekedahl and Shepherd-Barron, loc. cit.
We use here the notation in Section \ref{sec3}. 
\begin{proposition}
The intesection number $A\cdot K_S = 0$. In particular, the boundary components $B_i$ don't
intersect the conductrix $A$, and each component of $A$ is contained
in the Coble-Mukai lattice ${\rm CM(S)}$.
\end{proposition}
\begin{proof}
We have $-2K_S  \sim \sum_{i= 1}^{n}B_i$. 
Therefore, this result follows from 
$B \cap \sum_{i= 1}^{n}B_i = \emptyset$ and $B = 2A$.
\end{proof}
\begin{proposition}\label{contain}
The conductrix $A$ contains neither a fiber nor a half fiber.
\end{proposition}
\begin{proof}
We consider the divisorial part $(\eta)$ of the rational 1-form $\eta$, and
by the multiplication by $\eta$, we have an exact sequence:
$$
   0 \longrightarrow {\mathcal O}((\eta))\stackrel{\times \eta}{\longrightarrow} \Omega_S.
$$
We get an inclusion ${\rm H}^{0}(S, {\mathcal O}_S((\eta))) \hookrightarrow 
{\rm H}^{0}(S, \Omega_S^1)$.
Since $S$ is a rational surface, we have ${\rm H}^{0}(S, \Omega_S^1)= 0$. Therefore, we have
${\rm H}^{0}(S, {\mathcal O}_S((\eta)))= 0$.

Assume $f : S \longrightarrow {\bf P}^1$ has a multiple fiber $f^{-1}(P_0)$.
If $A$ contains a half fiber,  
$B = 2A$ contains the fiber $f^{-1}(P_0)$. In this case $\sum_{i=1}^{n}B_i$ is contained
in $f^{-1}(P_{\infty})$. Then, there exists a rational function on $S$
such that it has a pole of order 1 along $f^{-1}(P_0)$ and has a zero of order 1
along $f^{-1}(P_{\infty})$. It is a non-zero function in ${\rm H}^{0}(S, {\mathcal O}_S((\eta)))$,
a contradiction. If $A$ contain a fiber $F_1$, then $B = 2A$ contains $2F_1$.
Since $A \cap \sum_{i=1}^{n}B_i = \emptyset$, $F_1$ is different from $f^{-1}(P_{\infty})$
and $\sum_{i=1}^{n}B_i$ is contained in $f^{-1}(P_{\infty})$. Therefore,
there exists a rational function such that it has a pole of order 2
along $F_1$ and has a zero of order 1 along $f^{-1}(P_{\infty})$.  
It is a non-zero function in ${\rm H}^{0}(S, {\mathcal O}_S((\eta))$,
a contradiction.

Assume $f : S \longrightarrow {\bf P}^1$ has no multiple fiber.
If $A$ contains a fiber $F_1$, then $B = 2A$ contains $2F_1$.
The fiber $F_1$ is neither $f^{-1}(P_0)$ nor $f^{-1}(P_{\infty})$
as above. Therefore, there exists a rational function
such that it has a pole of order 2
along $F_1$ and has a zero of order 1 along $f^{-1}(P_{0})$ and $f^{-1}(P_{\infty})$,
respectively.  
It is a non-zero function in ${\rm H}^{0}(S, {\mathcal O}_S((\eta))$,
a contradiction.
\end{proof}

\begin{lemma}\label{vanishing1}
${\rm H}^i(S, \omega_S\otimes {\mathcal O}_S(A)) = 0$  for any $i$.
\end{lemma}
\begin{proof}
By (\ref{exactL'}) and (\ref{L'}), 
we have a long exact sequence:
$$
0 \longrightarrow {\rm H}^0(S, {\mathcal O}_S) \longrightarrow {\rm H}^0(S, (\pi \circ \nu)_{*}{\mathcal O}_{X'}) \longrightarrow {\rm H}^0(S, \omega_S\otimes {\mathcal O}_S(A)) 
\longrightarrow {\rm H}^1(S, {\mathcal O}_S).
$$
Since ${\rm H}^0(S, {\mathcal O}_S)\cong k$, ${\rm H}^0(S, (\pi \circ \nu)_{*}{\mathcal O}_{X'})\cong k$
and ${\rm H}^1(S, {\mathcal O}_S) = 0$, we have 
${\rm H}^0(S, \omega_S\otimes {\mathcal O}_S(A)) = 0$.
By the Serre duality theorem, we have
$$
   {\rm H}^2(S, \omega_S\otimes {\mathcal O}_S(A))\cong {\rm H}^0(S, {\mathcal O}_S(-A))= 0.
$$
Therefore, by the Riemann-Roch theorem, we have
$$
0 \geq \chi(\omega_S\otimes {\mathcal O}_S(A)) = \frac{A^2}{2} + 1.
$$
Therefore, we have 
$A^2 \leq -2$.

Since $K_S^2 = -n$, by Noether's formula and Igusa's formula (cf. Igusa \cite{I}), we have
$$
\begin{array}{cl}
  12 + n &= c_2(S) = \deg \la \eta \ra + (\eta)\cdot K_S - (\eta)^2 \\
    &  =\deg \la \eta \ra + (-\sum_{i= 1}^{n}B_i + 2A)\cdot K_S - (-\sum_{i= 1}^{n}B_i + 2A)^2\\
    & = \deg \la \eta \ra  + 2n -4A^2.
\end{array}
$$
Therefore we have
\begin{equation}\label{eta>}
 \deg \la \eta \ra = 12 - n + 4A^2.
\end{equation}
Since $\deg \la \eta \ra \geq 0$ and $n \geq 1$, we have $A^2 > -3$.
Hence, we conclude $A^2 = -2$, and
we have ${\rm H}^1(S, \omega_S\otimes {\mathcal O}_S(A)) = 0$
\end{proof}

By the proof of this lemma, we have the following.
\begin{corollary}\label{A^2}
$A^2 = -2$ and $\deg \la \eta \ra = 4 - n$.
\end{corollary}

\begin{proposition}\label{numerically connected}
$A$ is numerically $1$-connected.
\end{proposition}
\begin{proof}
Suppose we have a decomposition $A = A_1 + A_2$ with divisors $A_1 > 0$ and $A_2 > 0$.
Then, we have an inclusion
$$
 {\rm H}^0(S, \omega_S \otimes {\mathcal O}_S(A_1)) \hookrightarrow 
 {\rm H}^0(S, \omega_S \otimes {\mathcal O}_S(A)).
$$
Therefore, by Lemma \ref{vanishing1} we have 
${\rm H}^0(S, \omega_S \otimes {\mathcal O}_S(A_1)) = 0$.
By the Serre duality theorem, we have
${\rm H}^2(S, \omega_S \otimes {\mathcal O}_S(A_1))\cong {\rm H}^0(S, {\mathcal O}_S(-A_1))= 0$.
Therefore, by the Riemann-Roch theorem, we have 
$0 \geq \chi (\omega_S \otimes {\mathcal O}_S(A_1))= A_1^2/2  + 1$. 
Therefore, $A_1^2 \leq -2$. Similarly, we have $A_2^2\leq -2$.
By Corollary \ref{A^2}, we have
$-2 = A^2 = A_1^2 + 2A_1\cdot A_2 + A_2^2$, we conclude $A_1\cdot A_2 \geq 1$.
\end{proof}
\begin{proposition}
All isolated singularities of $X$ are rational singularities of type ${\rm A}_1$ and
the number of the singularities are less than or equal to $3$.
\end{proposition}
\begin{proof}
By Lemma \ref{A^2}, we have $\deg \la \eta \ra \leq 3$.
The idea of the proof comes from Ekedahl and Shepherd-Barron \cite[Proposition 0.5]{ES}.
Let $P \in \tilde{U}_i \subset X$ be a singular point.
The covering $\pi_i: \tilde{U}_i \longrightarrow U_i$ is given by $z_i^2 = f_ig_i$.
By putting the square part of $f_ig_i$ into $z_i$, let the equation
become $z_i^2 = k_i$. Let $(x_i, y_i, z_i)$
be local coodinates at the point $\pi_i(P)$.
Then, since $P$ is a singular point, $k_i$ has no terms of degree less than 2.
If it has a term of degree 2, then we have an expression $z_i^2 = x_iy_i + {\rm higher \ terms}$.
This is a rational singularity of type ${\rm A}_1$. If $k_i$ has no terms of degree 2,
then considering partial differentiations $k_{ix_i}$, $k_{iy_i}$, we have
$$
    \dim_{k} k[[x_i,y_i]]/(k_{ix_i}, k_{iy_i}) \geq 4.
$$
Therefore, we have $\deg \la \eta \ra \geq 4$, a contradiction.
Hence, all isolated singularities of $Y$ is rational singularities of type ${\rm A}_1$
and the number of isolated singularities is less than or equal to 3.
\end{proof}

\begin{corollary}
Under the assumptin $A \neq 0$,
the sum of the number of isolated singularity and the number of boundary components
is equal to $4$.
In particular, if a Coble surface $S$ has a quasi-elliptic fibration, then 
the sum of the number of isolated singularity and the number of boundary components
is always equal to $4$.
\end{corollary}
\begin{proof}
The former part  follows from $\deg \la \eta \ra + n = 4$. The latter part
follows from Lemma \ref{hascusp}.
\end{proof}

\begin{lemma}\label{OA}
  $\dim {\rm H}^0(S, {\mathcal O}_S(A)) =\dim {\rm H}^1(S, {\mathcal O}_S(A)) = 1$ and 
  $\dim {\rm H}^2(S, {\mathcal O}_S(A)) = 0$.
\end{lemma}
\begin{proof}
Since $A$ is a connected effective divisor such that $A_1^2 \leq -2$ holds 
for any effective divisor $A_1 \subset A$, 
we have $\dim {\rm H}^0(S, {\mathcal O}_S(A)) = 1$.
By an exact sequence 
$0 \longrightarrow \omega_S \otimes {\mathcal O}_S(-A) \longrightarrow \omega_S$,
we have an exact sequence 
$$
0 \longrightarrow {\rm H}^0(S,\omega_S \otimes {\mathcal O}_S(-A)) \longrightarrow 
{\rm H}^0(S, \omega_S).
$$
Since $S$ is rational, we have ${\rm H}^0(S, \omega_S)= 0$. Therefore, we have
${\rm H}^0(S,\omega\otimes {\mathcal O}_S(-A)) = 0$. By the Serre duality theorem
we have ${\rm H}^2(S, {\mathcal O}_S(A)) = 0$. Hence, by the Riemann-Roch theorem
we have $\dim {\rm H}^1(S, {\mathcal O}_S(A)) = 1$
\end{proof}
\begin{theorem}\label{structureA}
The irreducible components of the conductrix $A$ are all $(-2)$-curves.
If different two irreducible components intersect, then they intersect each other 
at a point transeversely.
\end{theorem}
\begin{proof}
Let $C$ be an irreducible curve contained in $A$. 
In the same way as in the proof of Proposition \ref{numerically connected}
we see $C^2 \leq -2$. On $S$, there exists no curve with self-intersection number $-3$.
The irreducible curves with self-intersection number $-4$ are only
boundary components $B_i$. Since $A$ does not contain $(-4)$-curves, we have
$C^2 = -2$. Since $C \subset A$, we have $C\cdot K_S = 0$,
and by the adjunction formula, $C$ is a non-singular rational curve
with $C^2 = -2$. 

Let $C_1$ and $C_2$ be two different irreducible curves which are contained in $A$.
If $C_1\cdot C_2 \geq 3$, then we have $(C_1 + C_2)^2 \geq 2$.
By the Riemann-Roch theorem, we have
$\chi({\mathcal O}_S(C_1 + C_2)) \geq 2$. Therefore we have
$\dim {\rm H}^0(S, {\mathcal O}_S(C_1 + C_2)) \geq 2$.
On the other hand, by the exact sequence
$0 \longrightarrow {\mathcal O}_S(C_1 + C_2) \longrightarrow {\mathcal O}_S(A)$ we have 
an inclusion 
${\rm H}^0(S, {\mathcal O}_S(C_1 + C_2)) \hookrightarrow {\rm H}^0(S, {\mathcal O}_S(A))$.
By Lemma \ref{OA}, we see $\dim {\rm H}^0(S, {\mathcal O}_S(C_1 + C_2)) = 1$, a contradiction.

If $C_1 \cdot C_2 = 2$, then we have $(C_1 + C_2)^2 = 0$. 
If $C_1$ and $C_2$ are contained in a same fiber $F_1$, then there exists a rational number
$\lambda$ such that $C_1 + C_2 = \lambda F_1$, which contradicts Proposition \ref{contain}.
Since $C_1 \cdot C_2 = 2$, they cannot be containd in different fibers.
Therefore, the genus 1 fibration is quasi-elliptic, and 
we may assume that $C_1$ is the curve of cusps and $C_2$ is contained in a fiber $F_1$.
By the general theory of quasi-elliptic surface, we have $C_1 \cdot F_1 = 2$.
Therefore, $C_2$ must be a simple irreducible component of $F_1$.
However, since $A$ is given by a part of the divisor $(dt/t)$, 
any simple irreducible component cannot be contained in $A$.
Hence, we have $0 \leq C_1 \cdot C_2 \leq 1$.
\end{proof}

We give here a remark for the case that 
the genus 1 fibration $f : S \longrightarrow {\bf P}^1$
is an elliptic surface.

\begin{proposition}\label{ellipticcase}
Let $f : S \longrightarrow {\bf P}^1$ be an elliptic fibration.
If $A$ contains a component of the multiple fiber $2F_1$, then $F_1 \supset A$ holds.
If $A$ contains a component of a non-multiple fiber $F_1$, then $F_1 \supset A$ holds.
\end{proposition}
\begin{proof}
Since $A$ is connected and does not contain any points in a general fiber,
$A$ consists of some components of one fiber, say $F_1$ or $2F_1$ if it
is a multiple fiber.
By Theorem \ref{structureA} $F_1$ contains at least two irreducible curves.
By Proposition \ref{contain}, $A$ does not contain $F_1$.
Therefore,
there exist an effective divisor $A_1$ and a non-negative divisor 
$A_2$ such that $A_1$ and $A_2$
have no common components and such that $F_1 - A = A_1 - A_2$.
On the one hand, we have
$$
(F_1 - A)^2 = F_1^2 -2F_1\cdot A + A^2 = -2.
$$
On the other hand, we have
$$
(F_1 - A)^2 = (A_1 - A_2)^2 = A_1^2 - 2A_1\cdot A_2 + A_2^2 \leq A_1^2 + A_2^2.
$$
If $A_2$ is an effective divisor, we have $A_1^2 \leq -2$ and $A_2^2 \leq -2$.
Therefore, we have $(F_1 - A)^2\leq -4$, a contradiction.
Therefore, we have $A_2 = 0$.
\end{proof}

For a rational number $a$, we denote by $[a]$ the integral part of $a$.

\begin{theorem}\label{ellipticsimple} 
Let $f : S \longrightarrow {\bf P}^1$ be an elliptic fibration
of a Coble surface $S$. Assume that the conductrix $A \neq 0$ and that $A$ is contained
in a simple fiber $F_1 = \sum_{i = 1}^{\ell}m_iC_i$ with irreducible components $C_i$.
Then, $A = \sum_{i = 1}^{\ell}[\frac{m_i}{2}]C_i$.
\end{theorem}
\begin{proof}
By the assumption, the fiber is neither $f^{-1}(P_{0})$ nor $f^{-1}(P_{\infty})$.
Therefore, the zero divisor of $dt$ on the open set defined by $t \neq {\infty}$
is equal to $2A$. By a suitable translation of $t$,
we may assume the singular fiber $F_1$
is given by $t = 0$, and let $A = \sum _{i = 1}^{\ell}\ell_i C_i$.
Then, on a general point $P$ of $C_i$ we have the expression $t = uf_i^{m_i}$. Here,
$f_i$ is a defining equation of $C_i$ at $P$ and $u$ is a unit.
If $m_i$ is odd, then we have $dt = f_i^{m_i - 1}(u df_i + f_i du)$. Since we can regard
$f_i$ is one element of local coodinates at $P$, the zero divisor which is given by
$dt$ at $P$ is $(m_i - 1)C_i$. Therefore we have $\ell_i = (m_i - 1)/2$.
Therefore, the possiblity of the types of singular fibers for $A \neq 0$ is only
type ${\rm I}_{\nu}^{*}$, ${\rm II}^{*}$, ${\rm III}^{*}$ and ${\rm IV}^*$.

In case that $m_i$ is even, we have $dt = f_i^{m_i} du$. Therefore, 
the order of the zero divisor
which is given by $dt$ is greater than or equal to $m_i$. Therefore, $\ell_i$ is
greater than or equal to $m_i/2$. 
We will show that if $m_i$ is even, then $\ell_i = m_i/2$.
Now, we consider the case of type ${\rm II}^{*}$. Then, the fiber is given by
$$
 C_1 + 2C_2 + 3C_3 + 4C_2 + 5C_5 + 6C_6 + 4C_7 + 2C_8 + 3C_9.
$$
Here, $C_i^2 = -2$ ($1\leq i \leq 9$), $C_i \cdot C_{i + 1} = 1$ ($ 1\leq i \leq 7$) 
and $C_6\cdot C_9 =1$.
The other intersection numbers are all 0. Considering the argument above, we set
$$
 A = (1 + a_2)C_2 + C_3 + (2 + a_4)E_4 + 2C_5 + (3 + a_6)E_6 + (2 + a_7)E_7 + (1 + a_8)E_8 + C_9.
$$
with non-negative integer $a_i$ ($i = 2, 4, 6, 7, 8$).
Since we have $A^2 = -2$, calculating the self-intersection numbers of both sides,
we have an equation:
$$
a_2(1 + a_2) + a_4(1 + a_4)+ a_6^2/ 2 + a_6 + (a_6 - a_7)^2/2 + (a_8 - a_7)^2/2 + a_8^2/2  = 0.
$$
Since $a_i \geq 0$ for all $i$, we have $a_2 = a_4 = a_6 = a_7 = a_8 = 0$.
This shows the claim for the type ${\rm II}^*$.
The proofs of the other cases are similar.
\end{proof}

Now, we examine the properties of $(-2)$-curves on a Coble surface $S$, 
following the method in Ekedahl and Shepherd-Barron \cite[Definition-Lemma 0.8]{ES}.
As before, we denote by $X'$ the normalization of the canonical covering
of $S$. We donote by $\tilde{X}$ the minimal resolution of $X'$.
Since the singularities of $X'$ are only rational double points,
the minimal dissolution $\tilde{S}$ gives the minimal resolution of singularities, 
and we get the following commutative diagram:
$$
\begin{array}{ccccc}
\tilde{C} & \subset & \tilde{X} &\stackrel{\mu}{\longrightarrow} & X' \\
  &  & \quad  \downarrow \tilde{\rho}&  & \quad \downarrow \rho \\
  & & \tilde{S} & \stackrel{\theta}{\longrightarrow} & S\\
  & & \cup &   &   \cup \\
   & & C' & \longrightarrow  & C
\end{array}
$$
Here, $C$, $C'$ and $\tilde{C}$ are irreducible curves such that
$\tilde{\rho}({\tilde{C}}) = C'$ and $\theta (C') = C$. The morphism $\theta$
is a morphism given by a finite number of blowing-ups.
We denote by $r$ the number of blowing-ups on the curve $C$ which includes
blowing-ups at the infinitesimal near points. 
We also denote by $s$ the degree of the morphism $\rho\circ \mu : \tilde{C} \longrightarrow C$.
This number $s$ is equal to the degree of 
the morphism $\tilde{\rho}: \tilde{C} \longrightarrow C'$ and it is either 1 or 2.
By (\ref{dualizingY'}) we have $\omega_{X'} \cong \rho^{*}({\mathcal O}_S(-A))$. 
Since the singularities of $X'$ are rational double points, we have
$\omega_{\tilde{X}} \cong \mu^{*}\omega_{X'}$. Therefore, we have
$$
\omega_{\tilde{X}} \cong \mu^{*}\rho^{*}({\mathcal O}_S(-A)).
$$
Since $(C')^2 = C^2 - r$,  we have
$$
\left(\frac{2}{s}\tilde{C}\right)^2 =({\tilde{\rho}}^{-1}(C'))^2 = \deg\ \tilde{\rho}\cdot (C')^2 = 2(C^2 - r).
$$
Therefore, we have 
$$
 \tilde{C}^2 = (C^2 - r)s^2/2.
$$
By the adjunction formula, we have
$$
\omega_{\tilde{C}} \cong (\omega_{\tilde{X}}\otimes {\mathcal O}_{\tilde{X}}(\tilde{C}))\vert_{\tilde{C}}
\cong (\mu^{*}\rho^{*}({\mathcal O}_S(-A))\otimes {\mathcal O}_{\tilde{X}}(\tilde{C}))\vert_{\tilde{C}}.
$$
We denote by $g(\tilde{C})$ the genus of the curve $\tilde{C}$.
Taking the degrees of both sides, we have
$$
\begin{array}{cl}
 2(g(\tilde{C}) - 1)  & = \mu^{*}\rho^{*}({\mathcal O}_S(-A))\cdot\frac{s}{2}\mu^{*}\rho^{*}({\mathcal O}_S(C)) + {\tilde{C}}^2\\
  & = 2\cdot \frac{s}{2}\{(- A)\cdot C\} + (C^2 - r)s^2/2\\
  & = -s A\cdot C + (C^2 - r)s^2/2.
\end{array}
$$
Therefore, we have
\begin{equation}\label{genusC}
g(\tilde{C}) = - sA\cdot C/2 + (C^2 - r)s^2/4 + 1.
\end{equation}
Note that this formula holds even if $A = 0$.
\begin{remark} In Ekedahl, Shepherd-Barron \cite[Definition-Lemma 0.8]{ES},
they give similar results to the theorem below. But the $A$ which they treat
is not the conductrix. Therefore, we give here the complete proof of the theorem for
our conductrix $A$ to make the situation clear. However, in the case of Enriques surfaces 
the difference between these two A's is just the factor of the dualizing sheaf,
which is numerically trivial. Therefore, the numerical results are same
for these two A's. In fact, in Katsura, Kond\=o and Martin \cite[Lemma 3.4]{KKM},
these facts are used.
\end{remark}

\begin{theorem}
Let $C$ be a $(-2)$-curve on $S$. 
\begin{itemize}
\item[{\rm (i)}] If $s = 1$, then $A\cdot C = 1 - \frac{r}{2}$ holds.
\item[{\rm (ii)}] If $s = 2$, then $A\cdot C = - 1 - r$ holds.
\end{itemize}
\end{theorem}
\begin{proof}
Since $C$ is a non-singular curve of genus 0 with $C^2 = -2$, we get the results
form (\ref{genusC}).
\end{proof}

\begin{corollary}\label{conductrix-detail}
Let $C$ be a $(-2)$-curve.
\begin{itemize}
\item[{\rm (1)}] If $A \not\supset C$, then $s = 1$ and only the following cases occur.
\begin{itemize}
\item[{\rm (i)}] $A\cdot C = 0$ and $r = 2$.
\item[{\rm (ii)}] $A\cdot C = 1$ and $r = 0$.
\end{itemize}
\item[{\rm (2)}] If $A \supset C$, then only the following cases occur.
\begin{itemize}
\item[{\rm (a)}] Case $s = 1$.
\begin{itemize}
\item[{\rm (i)}] $A\cdot C = -2$, $A = C$ and $r = 6$.
\item[{\rm (ii)}] $A\cdot C = -1$ and $r = 4$.
\item[{\rm (iii)}] $A\cdot C = 0$ and $r = 2$.
\item[{\rm (iv)}] $A\cdot C = 1$ and $r = 0$.
\end{itemize}
\item[{\rm (b)}] Case $s = 2$.
\item[{\rm (i)}] $A\cdot C = -2$, $A = C$ and $r = 1$.
\item[{\rm (ii)}] $A\cdot C = -1$ and $r = 0$.
\end{itemize}
\end{itemize}
\end{corollary}
\begin{proof}
(1) By assumption  $A \not\supset C$, we have $A\cdot C\geq 0$. Therefore, we have $s = 1$.
Therefore, by $0 \leq A \cdot C = 1 - \frac{r}{2}$, we have $r \leq 2$. 
Since $A\cdot C$ is an integer and $r \geq 0$, we have $r = 2$ or $0$, and
we get our results.

(2) If $A = C$, then we have clearly $A\cdot C = -2$ and $r = 6$ if $s = 1$,
$r = 1$ if $s = 2$. Now, we assume $A \neq C$ and $A \supset C$. 
Since $A$ is numerically connected, we have
$(A - C) \cdot C \geq 1$. Therefore,  we have $A\cdot C \geq -1$ by $C^2 = -2$.
If $s = 1$, then we have
$- 1 \leq A \cdot C = 1 - \frac{r}{2}$ and we have $r \leq 4$.
Therefore, we have $r = 0, 2$, or $4$.
If $ s = 2$, we have $-1 \leq A \cdot C = -1 -r$. Therefore, we have $r = 0$,
and we get our result.
\end{proof}

\begin{corollary}\label{-1and 1}
Let $C$ be a $(-2)$-curve on $S$. Then, if $A \neq C$, then we have $-1 \leq A \cdot C \leq 1$.
If $s = 2$, then we have $C \subset A$.
\end{corollary}
\begin{proof}
This follows from Corollary \ref{conductrix-detail}.
\end{proof}

The following corollary is shown in Ekedahl and Shepherd-Barron \cite[Definition-Lemma 0.8 (iii)]{ES}.
\begin{corollary}\label{AC=1}
$({\rm i})$ Let $C_1$ be a $(-2)$-curve on $S$ with $A\cdot C_1 = 1$. Then, any $(-2)$-curve $C_2$
which intersects $C_1$ only at a point transversely has $s = 2$ with $A\cdot C_2 = -1$
unless $A = C_2$.
Moreover, $C_2$ is an irreducible component of $A$.

$({\rm ii})$ If two $(-2)$-curves $C_1$, $C_2$ on $S$ which meet transversely at a point 
have $s$-invariant $1$, then their intersection is blown-up.

\end{corollary}
\begin{proof}
We denote by $C'_1$ (resp. $C'_2$) the proper transform of $C_1$ (resp. $C_2$)
for the morphism $\theta$. 

(i) By Corollary \ref{conductrix-detail}, for $C_1$  we have $r = 0$ and $s=1$. 
This means that 
we don't blow up at the intersection point of $C_1$ and $C_2$, and that there exists
an irreducible curve $\tilde{C}_1$ on ${\tilde{X}}$ such that 
$\tilde{\rho}^{-1}(C'_1) = 2\tilde{C}_1$.
Therefore, we have $C'_1\cdot C'_2 = C_1 \cdot C_2 = 1$ and
$$
2 = \deg(\tilde{\rho})(C'_1\cdot C'_2) =  \tilde{\rho}^{-1}(C'_1)\cdot \tilde{\rho}^{-1}(C'_2)
  = 2 (\tilde{C_1}\cdot \tilde{\rho}^{-1}(C'_2)).
$$
Therefore, we have $\tilde{C_1}\cdot \tilde{\rho}^{-1}(C'_2)= 1$. Hence, we have
$s = 2$ for the curve $C_2$ and 
by Corollary \ref{conductrix-detail} we conclude $A\cdot C_2 = -1$ and $C_2 \subset A$.

(ii) Suppose that their intersection is not blown-up. Then, we have $C'_1\cdot C'_2 = 1$.
Since we have $s = 1$ for both $C_1$ and $C_2$, there exist
irreducible curves $\tilde{C}_1$ and $\tilde{C}_2$ on ${\tilde{X}}$ such that 
$\tilde{\rho}^{-1}(C'_1) = 2\tilde{C}_1$ and $\tilde{\rho}^{-1}(C'_2) = 2\tilde{C}_2$.
Therefore, we have
$$
  2 = \deg~ \tilde{\rho} ~(C'_1\cdot C'_2) = \tilde{\rho}^{-1}(C'_1)\cdot  \tilde{\rho}^{-1}(C'_2) = 4 (\tilde{C}_1\cdot \tilde{C}_2),   
$$
a contradiction.
\end{proof}

\begin{corollary}\label{eachtype} 
Let $F$ be a fiber such that $F\cap A \neq \emptyset$.
Assume that $F$ is of type ${\rm I}_m^{*}$, ${\rm II}^{*}$,
${\rm III}^{*}$ or ${\rm IV}^{*}$.
Then, all irreducible components except irreducible components at the end
of $F$ are contained in $A$.
\end{corollary}
\begin{proof}
Let $C$ be an irreducible component of $F$ which does not exist at the end of $F$.
Suppose that $C$ is not contained in $A$. By choosing a suitable irreducible component,
we may assume that $C$ intersects $A$.
Since $A$ is connected, 
there exists an irreducible
component $G$ of $F$ such that $G$ intersects $C$ and is not contained in $A$.
Since $A\cdot C > 0$, we have $A\cdot C = 1$ by Corollary \ref{-1and 1}.
Therefore, by Corollary \ref{AC=1} we have $G \subset A$, a contradiction.
\end{proof}

\begin{corollary}\label{endcomponent} 
Let $f: S \longrightarrow {\bf P}^1$ be an elliptic fibration of a Coble surface $S$. 
Assume this fibration has a $(-2)$-curve $C_s$ as a 2-section.
Let $2F$ be a multiple fiber such that $F\cap A \neq 0$ and 
 $2F$ be of type $2{\rm I}_m^{*}$, $2{\rm II}^{*}$,
$2{\rm III}^{*}$ or $2{\rm IV}^{*}$.
Then the simple irreducible component $C$ of the half fiber $F$ which intersects
the special 2-section $C_s$ is contained in $A$.
\end{corollary}
\begin{proof}
Suppose that $C$ is not contained in $A$. Then, we have $A \cdot C_s = 0$.
By Corollary \ref{eachtype},
we have $A\cdot C > 0$. Therefore, by Corollary \ref{conductrix-detail}, we have
$A \cdot C = 1$. Since $C\cdot C_s = 1$,  by Corollary \ref{AC=1}
we have $A\cdot C_s = -1$, a contradiction.
\end{proof}

\begin{remark}\label{non-emptyA}
Let $f: S \longrightarrow {\bf P}^1$ be an elliptic fibration of a Coble surface $S$.
If a simple fiber $F$ contains the conductrix $A$, then as we see in the proof 
of Theorem \ref{ellipticsimple}, the conductrix $A$ for the simple fiber $F$
does not contain any simple irreducible component of $F$. Therefore,
in the case of elliptic fibration with a $(-2)$-curve of 2-section,
we can distinguish the type of conductrix of multiple fiber from the one of
simple fiber by Corollary \ref{endcomponent}.
\end{remark}

Let $A = \sum_i m_iC_i$ be a conductrix.  
The graph of $A$ is the dual graph of irreducible curves $\{C_i\}$ with 
the multiplicity $m_i$.  We use graphs of conductrices in the following 
Lemma \ref{exclude}, Tables \ref{TableQ-elliptic}, \ref{TableElliptic1}, \ref{TableElliptic2}.

We learned the following lemma from Shepherd-Barron. In Ekedahl, Shepherd-Barron \cite{ES},
they condidered
the weight of the conductrix $A$ which is a function on the N\'eron-Severi group 
${\rm NS}(S)$ defined 
by $w(x) = x\cdot A$. The weight satisfies several conditions in which a vertex 
of weight $0$ is adjacent to at most two other vertices of weight $0$,
from which the lemma follows. We give here a direct proof.

\begin{lemma}\label{exclude}
For the multiple fiber of type ${\rm I}_4^{*}$ of elliptic fibration,
the following graph with given multiplicities is not represented by a conductrix.

\xy
(-40,20)*{};
@={(10,10),(20,10),(30,10),(40,10), (50, 10), (5,15),(5, 5), (55, 15), (55, 5)}@@{*{\bullet}};
(10,10)*{};(50,10)*{}**\dir{-};
(5,5)*{};(10,10)*{}**\dir{-};
(5,15)*{};(10,10)*{}**\dir{-};
(50,10)*{};(55,15)*{}**\dir{-};
(50,10)*{};(55,5)*{}**\dir{-};
(10,14)*{1};
(2,15)*{0};
(2,5)*{0};
(20,14)*{1};
(30,14)*{2};
(58,5)*{1};
(40,14)*{2};
(50,14)*{2};
(58,15)*{1};
\endxy

%Here, the numbers are the multiplicities of the irreducible components of $A$.
\end{lemma}
\begin{proof}
Assume that this graph is represented by a conductrix $A$.
The graph contains a Dynkin diagram of type $D_4$. 
The four curves $C_i$ ($1\leq i \leq 4$) 
satisfy $C_i \subset A$ and  $A\cdot C_i = 0$. Therefore, by Corollary \ref{conductrix-detail}
they have all $s = 1$ and $r = 2$. Therefore, by Corollary \ref{conductrix-detail}
the intersection points are blown-up. However, 
the midle curve in the Dynkin diagram of $D_4$ has 3 intersetion points and it must have $r\geq 3$, a contradiction.
\end{proof}

Now, we can classify all possible conductrices for Coble surfaces.
In Ekedahl and Shepherd-Barron \cite{ES}, they already classify
the conductrices up to a multiple of the Kodaira-N\'eron cycle
of the special fiber as a general theory. To make clear the results 
for Coble surfaces,
we list up the possibilities of the conductrices for Coble surfaces,
although our list is a part of the one in Ekedahl and Shepherd-Barron \cite{ES}.
Our results are not up to a multiple of the Kodaira-N\'eron cycle
of the special fiber, but give the precise divisors.
Under the results of Propositions \ref{contain}, \ref{numerically connected}, 
\ref{ellipticcase},
Lemma \ref{hascusp} and Corollary \ref{endcomponent}, we can calculate 
the possibilities of the conductrices for Coble surfaces
by using Corollaries \ref{-1and 1}, \ref{AC=1} (i) and \ref{eachtype} 
by a tedious routine calculation. Note that since Coble surfaces are rational,
the maximun number of irreducible components of the singular fibers is 9.

In case of quasi-elliptic fibration, 
there are exactly two types of conductrices for each singular fiber of type ${\rm III}$, 
${\rm I}_0^*$, ${\rm I}_2^*$, ${\rm I}_4^*$, ${\rm III}^*$, ${\rm II}^*$ according to whether
the fiber is simple or multiple. 
Hence, we get Table \ref{TableQ-elliptic}, where 
the hollow vertices are the curves of cusps.

%%%%%
\begin{table}[!htb]
% \begin{center}
\hspace{-3.0cm}  \includegraphics[width=160mm]{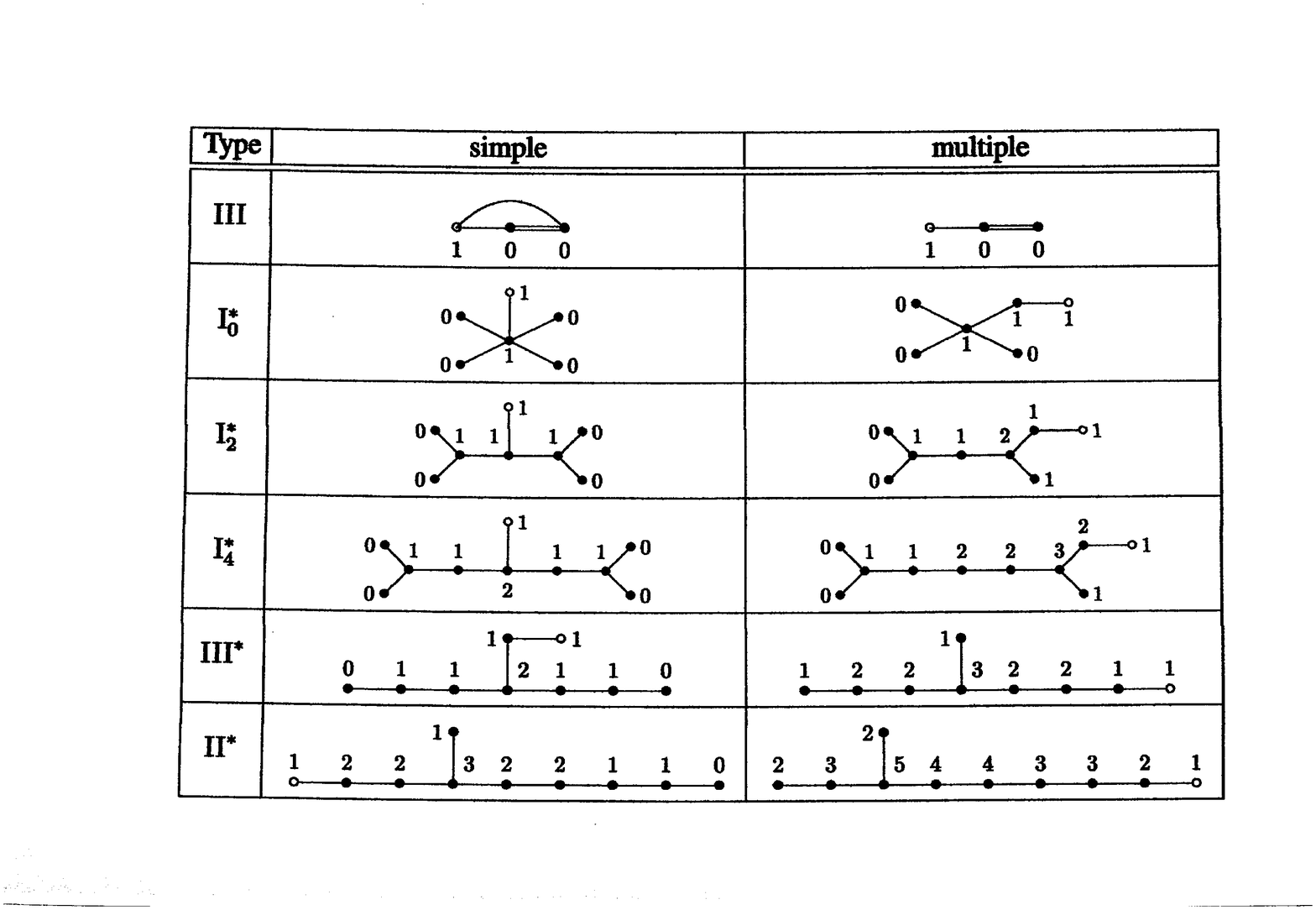}
% \end{center}
\vspace{-1.0cm}
 \caption{Conductrices for quasi-elliptic fibrations}
 \label{TableQ-elliptic}
\end{table}
%%%%%

%%%%%
\begin{table}[!htb]
% \begin{center}
\hspace{-2.9cm}  \includegraphics[width=160mm]{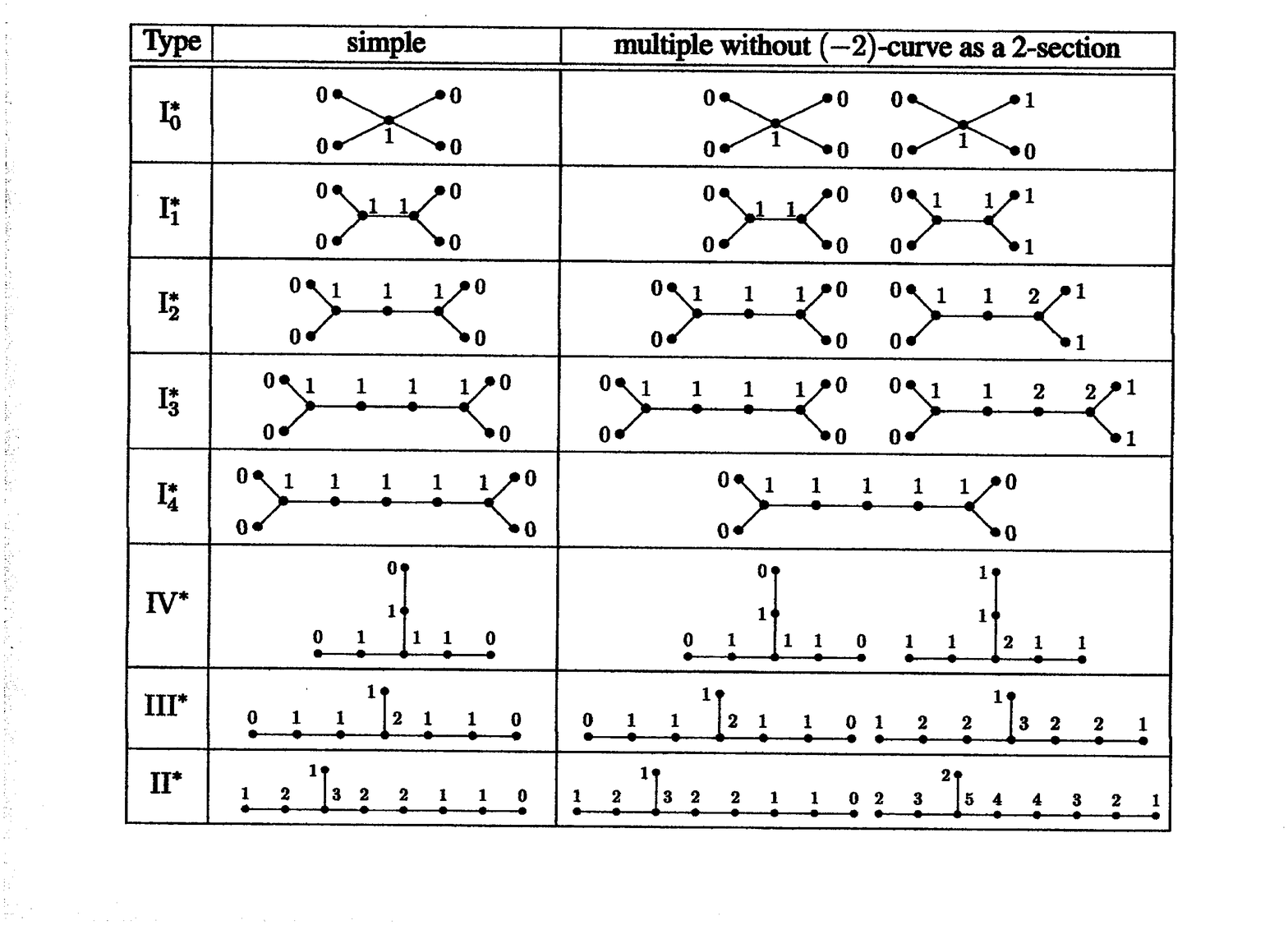}
% \end{center}
 \vspace{-1.0cm}
 \caption{Conductrices for elliptic fibrations ($A\ne 0$)}
 \label{TableElliptic1}
\end{table}

%%%%%

In case of elliptic fibration, if the singular fiber is simple,
we already calculated the conductrix $A$
in Theorem \ref{ellipticsimple}. Namely,
if the singular fiber of type ${\rm I}_0^*$, 
${\rm I}_1^*$, ${\rm I}_2^*$, ${\rm I}_3^*$, ${\rm I}_4^*$, 
${\rm IV}^*$, ${\rm III}^*$ or ${\rm II}^*$ is simple, we have exactly one type
of conductrix for each singular fiber. 
If the singular fiber of type ${\rm I}_0^*$, 
${\rm I}_1^*$, ${\rm I}_2^*$, ${\rm I}_3^*$, 
${\rm IV}^*$, ${\rm III}^*$ or ${\rm II}^*$ is multiple
and the  fibration has no $(-2)$-curve of 2-section, then
there exist two possibilities of the type of conductrix for each singular fiber.
If the singular fiber of type ${\rm I}_4^{*}$ is multiple, then we have exactly one type
of conductrix. Because one possibility is excluded by Lemma \ref{exclude}.
Hence, we get Table \ref{TableElliptic1}.
For a multiple fiber of elliptic fibration, if the fibration has 
a $(-2)$-curve of 2-section, then the situation is different.
If the elliptic fiberation with the multiple fiber of type
type ${\rm I}_0^*$, ${\rm I}_1^*$ or
${\rm IV}^*$ has a $(-2)$-curve of 2-section, then 
we have exactly one type of conductrix for each multiple fiber,
and there exists no elliptic fibration with multiple fiber of type
${\rm I}_3^*$, ${\rm I}_4^*$, ${\rm III}^*$ or ${\rm II}^*$
such that the fibration has a $(-2)$-curve of 2-section.
This follows from Lemma \ref{exclude2} below.
Hence, we get Table \ref{TableElliptic2}.

In Tables \ref{TableQ-elliptic}, \ref{TableElliptic1} and \ref{TableElliptic2},
the numbers are the multiplicities of the irreducible components
of $A$, ``simple'' means a simple fiber and ``multiple'' means
a multiple fiber.
We will give two examples of concrete calculations
of conductrices of multiple fibers of elliptic fiber space below.

\begin{lemma}\label{exclude2}
There are no special elliptic fibrations on Coble surfaces with a multiple fiber $2F$
of type $2{\rm II}^*, 2{\rm III}^*$, $2{\rm I}_m^*$ $(m = 2, 3, 4)$.
\end{lemma}
\begin{proof}
Let $C_s$ be a $(-2)$-curve of 2-section. Then, in Table 3,
the graphs of conductrices for simple fibers 
are excluded by Corollary \ref{endcomponent}.
Let $A \supset C$ be the curve of the fiber $2F$ which intersects $C_s$. 
The curve $C$ is the one which exists at the end of $F$ and the multiplicity in $A$
is 1.
Then, since $A\cdot C_s > 0$,
by Corollary \ref{-1and 1} we have $A\cdot C_s = 1$. Therefore, by Corollary \ref{AC=1} (i)
we have $A\cdot C = -1$. Hence, in Table 3, we have no possibilities for the multiple fibers
of type $2{\rm II}^*$, $2{\rm III}^*$, $2{\rm I}_m^*$ ($m = 2, 3, 4$)
\end{proof}

%%%%%
\begin{table}[!htb]
% \begin{center}
 \hspace{-2.5cm} \includegraphics[width=160mm]{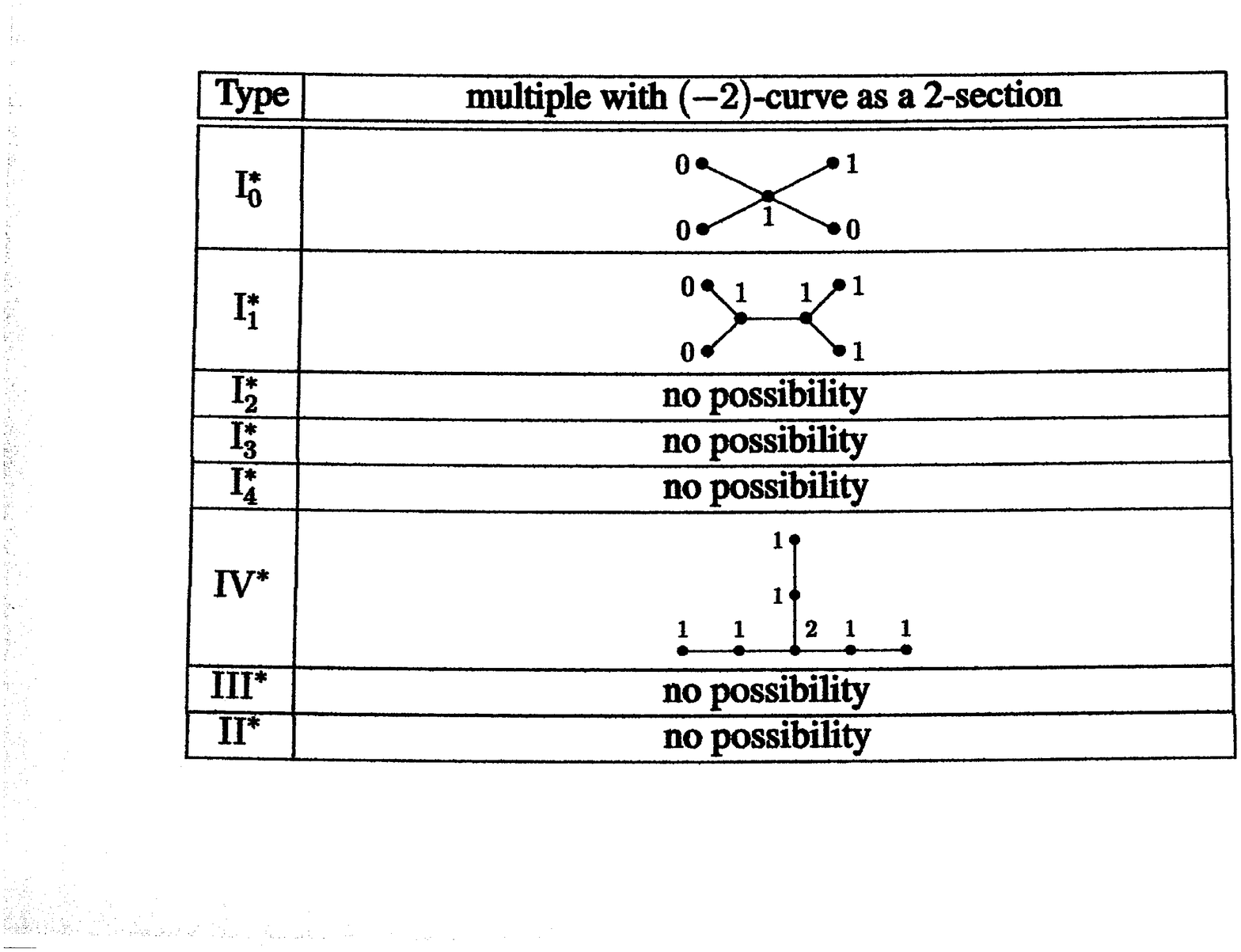}
% \end{center}
 \vspace{-2.0cm}
 \caption{Conductrices of multiple fibers for elliptic fibrations with 2-section ($A\ne 0$)}
 \label{TableElliptic2}
\end{table}
%%%%%

\begin{example}
We consider the multiple fiber $2F$ of type ${\rm I}_1^{*}$ 
of an elliptic fibration $f : S \longrightarrow {\bf P}^1$ and
calculate the conductrix in this case.
We have
$$
    F = E_1 + E_2 + E_3 + E_4 + 2E_5 + 2E_6
$$
with irreducible components $E_i$ ($1 \leq i \leq 6$).
Here, $E_i$ ($1 \leq i \leq 6$) are $(-2)$-curves
such that $E_1 \cdot E_5 = E_2 \cdot E_5 = 1$, $E_3 \cdot E_6 = E_4 \cdot E_6 = 1$
and $E_5\cdot E_6 = 1$.
We assume there exists a special 2-section $C$, and we may assume $C$ intersects
$E_4$. Let $A$ be the conductrix. Then, by Proposition \ref{ellipticcase}
we have $F \supset A$ and by Corollary \ref{eachtype} we have the expression
$$
  A =a_1E_1 + a_2E_2 + a_3E_3 + a_4E_4 + (1 + a_5)E_5 + (1 + a_6)E_6 \quad (0\leq \forall a_i\leq 1).
$$
Since $A \not\supset C$ and $A \cdot C > 0$, we have $A\cdot C = 1$ by
Corollary \ref{conductrix-detail}. Therefore,  we have $a_4 = 1$ and 
$A \cdot E_4 = -1$ by Corollary \ref{AC=1}.
Therefore, we have  $(1 + a_6) -2 = -1$, that is, we have $a_6 = 0$.
If $a_3 = 0$, then by $A \not\supset E_3$ and $A\cdot E_3 > 0$, we have
$A\cdot E_3 = 1$ as above. Therefore, we have $A \cdot E_6 = - 1$.
However, $A \cdot E_6 = (1 + a_5) + 1 - 2= a_5\geq 0$, a contradiction.
Therefore, we have $a_3 = 1$. Now, we have
$$
   A\cdot E_6 = (1 + a_5) + 1 + 1 - 2 = 1 + a_5.
$$
By Corollary \ref{-1and 1}, we have $a_5 = 0$ and $A\cdot E_6 = 1$.
Therefore, again by Corollary \ref{AC=1} we have
$ -1 = A\cdot E_5 = a_1 + a_2 + 1 -2$.
Therefore, we have $a_1+ a_2 = 0$. Since $a_1$ and $a_2$ are non-negative,
we have $a_1 = a_2 = 0$, and we conclude $A = E_3 + E_4 + E_5 + E_6$.
\end{example}

\begin{example}
We consider the multiple fiber $2F$ of type ${\rm IV}$ 
of an elliptic fibration $f : S \longrightarrow {\bf P}^1$ and
show that the conductrix $A$ is empty for this fiber.
We have
$$
    F = E_1 + E_2 + E_3
$$
with irreducible components $E_i$ ($1 \leq i \leq 3$).
Here, $E_i$ ($1 \leq i \leq 3$) are $(-2)$-curves
such that $E_i \cdot E_{j} = 1$ ($i \neq j$).
Suppose $A$ is non-empty.
Since $A\subset F$ and $A \neq F$, we may assume
either $A = E_1 + E_2$ or $A = E_1$ without loss of generality. 
Then, since we have $A \cdot E_{3}\leq 1$
by Corollary \ref{conductrix-detail}, (1), the first case is excluded.
Therefore, we have $A\cdot E_i = 1$ ($i = 2, 3$).
Therefore, again by Corollary \ref{conductrix-detail}, we have $s = 1$ and $r = 0$
for the curves $E_2$ and $E_3$. However, by Corollary \ref{AC=1} (ii),
the intersection point of $E_2$ and $E_3$ is blown-up, and we have $r \geq 1$
for these 2 curves, a contradiction. Hence, we have $A = 0$.
\end{example}

\section{Possible dual graphs and the number of boundary components}\label{sec5}

In this section we will determine the dual graph of effective irreducible roots on
Coble surfaces with finite automorphism group and their numbers of boundary components.

We begin with showing some lemmas.

\begin{lemma}\label{D4diagram}
Assume  that 
a Dynkin diagram of type $D_4$ is represented by effective irreducible roots.  Then it can be realized by only  $(-2)$-curves.
\end{lemma}
\begin{proof}
By Lemma \ref{roots2} (2), 
a diagram of type $D_4$ is generated by only $(-2)$-curves or by only $(-1)$-roots.
If $D_4$ is generated by only $(-1)$-roots, then the central $(-1)$-root of $D_4$ should contain three boundary components by Lemma \ref{roots2} (1) which is a contradiction.
\end{proof}

\begin{lemma}\label{ellII*}
If $A\ne 0$, then a special extremal elliptic fibration with a multiple fiber 
 of type $\tilde{A}_2$ cannot exist.
\end{lemma}
\begin{proof}
It follows from  Proposition \ref{Ito} and the assumption $A\ne 0$ 
that the fibration has
singular fibers of type $({\rm IV}, {\rm IV}^*)$ or their blow-ups.  Obviously the first
case does not occur.  Hence the fiber $2F$ of type $\tilde{A}_2$ is 
obtained by blowing-up a fiber of type ${\rm IV}$, that is, 
the fiber $F$ consists of $(-1)$-roots.
If a special $2$-section $\alpha$ is a $(-2)$-curve, then
$F\cdot \alpha \geq 2$ which contradicts the fact that $F$ is a multiple fiber.
If $\alpha$ is a $(-1)$-root, then
$\alpha$ contains a $(-4)$-curve $B$ not contained in $F$. Since $B$ is contained in  
a singular fiber, $B$ should be a component of the fiber of type ${\rm IV}^*$ which is impossible.
\end{proof}

\begin{lemma}\label{multiple-fiber1}
Let $f : S\to {\bf P}^1$ be an elliptic fibration on a Coble surface $S$ and $F$ a fiber of
$f$.  Assume that $F$ is of type ${\rm I}_n$ or the blowing-ups of the 
singular points of a fiber of type ${\rm I}_n$ $(n\geq 2)$.
Then $\alpha\cdot F\geq 2$ for any effective irreducible root $\alpha$ with $\alpha\cdot F > 0$.
\end{lemma}
\begin{proof}
First assume that $F$ is of type ${\rm I}_n$.  By Lemma \ref{roots2} (2), 
$\alpha\cdot F \geq 2$ if $\alpha$ is a $(-1)$-root.  If $\alpha$ is a $(-2)$-curve with 
$\alpha\cdot F =1$, then $\alpha$ and any component of $F$ are not contained 
in the conductrix $A$.  Therefore three points on the component 
of $F$ meeting $\alpha$ are blown-up by Lemma \ref{AC=1} (ii) which contradicts Corollary
\ref{conductrix-detail}.  Next assume that $F$ is the blowing-ups of the 
singular points of a fiber of type ${\rm I}_n$.  
Then $F$ is the sum of $(-1)$-roots (Remark \ref{two-double}) and hence $F\cdot \alpha \geq 2$ if $\alpha$ is a 
$(-2)$-curve (Lemma \ref{roots2}).  If $\alpha$ is a $(-1)$-root, then the $(-1)$-curve appeared in 
$\alpha$ meets a $(-4)$-curve in $F$ or
a $(-1)$-curve in $F$.  In both cases, $F\cdot \alpha \geq 2$.
\end{proof}

\begin{theorem}\label{dualgraph}
Let $S$ be a Coble surface in characteristic $2$ with a finite automorphism group.
Then the dual graph of irreducible effective roots on $S$ is one of the dual graphs given in Theorem {\rm \ref{main2}}.  
\end{theorem}
\begin{proof}
Let $S$ be a Coble surface with finite automorphism group.  Let $A$ be the conductrix of $S$ including the case $A=0$ and let $n$ be the number of boundary components.
It follows from Lemma \ref{specialfibration} that there exists a special genus 1 fibration on $S$ which
is extremal by Proposition \ref{MW-Dolgachev}.  Note that every configration of effective irreducible roots forming an extended Dynkin diagram is a fiber of a genus 1 fibration on $S$
(Lemma \ref{fibration-isotropic}).  These data can allow us to control effective irreducible roots.
Now starting from a special genus 1 fibration together with the possibilities of the conductrix $A$ (Tables \ref{TableQ-elliptic}, \ref{TableElliptic1} and \ref{TableElliptic2}), 
we arrive at a contradiction or $S$ contains a dual graph of irreducible effective roots in Theorem \ref{main2}.  We observe that  all dual graphs satisfy the condition in Theorem \ref{Vinberg}, and hence the automorphism group of the Coble surface with one of these dual graphs is finite by Proposition \ref{FiniteIndex}.  
Proposition \ref{Vinbergremark} implies that these are all effective irreducible 
roots on the Coble surface $S$.
This is a story of the proof.

We have shown the same properties of $A$ as in the case of Enriques surfaces holds in the previous section and have prepared Lemmas \ref{exclude2}, \ref{multiple-fiber1}, \ref{ellII*} which were used in the proof of the case of Enriques surfaces.  
Thus the proof of \cite[Theorem 4.1]{KKM} works well in this case, and hence the assertion
on the dual graphs has been proved except the non-existence of the one case  
(of course, we should consider $(-1)$-roots, but all arguments are easy exercise.  We leave the details to the reader).
The dual graph  
we want to exclude is given by the following
Figure \ref{graphD4D4}.
\begin{figure}[!htb]
 \begin{center}
\xy
(-45,25)*{};
@={(0,10),(10,0),(10,10),(20,10),(30,10),(40,10),(50,10),(10,20),(60,10),(50,20),(50,0)}@@{*{\bullet}};
(0,10)*{};(60,10)*{}**\dir{-};
(10,0)*{};(10,20)*{}**\dir{-};
(50,0)*{};(50,20)*{}**\dir{-};
(30,6)*{N};
\endxy
 \end{center}
 \caption{}
 \label{graphD4D4}
\end{figure}

\noindent
The diagram contains a Dynkin diagram of type $D_4$.  It follows from lemma \ref{D4diagram} and \ref{roots2} (2) that all vertices are represented by $(-2)$-curves.  
Note that there exists a parabolic subdiagram of type $\tilde{D}_4\oplus \tilde{D}_4$
which corresponds to a quasi-elliptic fibration with singular fibers of 
type $({\rm I}_0^*, {\rm I}_0^*)$ by Propositions \ref{Lang}, \ref{Ito}.  
Since $N$ is represented by a $(-2)$-curve by Lemma \ref{roots2} (2), 
the fibration is a Halphen surface of index 2.  This contradicts the existence of
two multiple fibers of type ${\rm I}_0^*$.
Thus this case does not occur.
\end{proof}

\begin{theorem}\label{Numberboundcomp}
Let $S$ be a Coble surface in characteristic $2$ with a finite automorphism group.
Then the number of boundary components of $S$ is given as in Theorem {\rm \ref{main2}}.
\end{theorem}
\begin{proof}
We prove the assertion for each type of dual graphs.

\smallskip

(1) \ {\bf The case of type $\tilde{E}_8$.}  
The dual graph is given by the following Figure \ref{graphE8}.

\begin{figure}[!htb]
 \begin{center}
\xy
(-45,25)*{};
@={(-10,10),(0,10),(10,10),(20,10),(30,10),(40,10),(50,10),(60,10),(70,10),(10,20)}@@{*{\bullet}};
(-10,10)*{};(70,10)*{}**\dir{-};
(10,10)*{};(10,20)*{}**\dir{-};
(60,6)*{E_9};(50,6)*{E_8};(40,6)*{E_7};
(30,6)*{E_6};(20,6)*{E_5};(10,6)*{E_3};
(0,6)*{E_2};(-10,6)*{E_1};(70,6)*{E_{10}};(15,20)*{E_{4}};
\endxy
 \end{center}
 \caption{}
 \label{graphE8}
\end{figure}
It follows from Lemma \ref{D4diagram} and \ref{roots2} (2) that all vertices are represented by $(-2)$-curves. 
There exists a parabolic subdiagram of type $\tilde{E}_8$ which corresponds to  
a quasi-elliptic fibration with a singular fiber of type ${\rm II}^*$
(Lemmas \ref{fibration-isotropic}, \ref{exclude2}).
If the fibration is obtained from a Jacobian fibration by blowing up
the singular points of
two irreducible fibers, then the vertex $E_{10}$ is a $(-1)$-root which contradicts Lemma \ref{roots2} (2).
Thus the fibration is induced from a Halphen surface of index 2 with a multiple fiber of type ${\rm II}^*$, and hence the number of boundary component is one.

\smallskip

(2)\ {\bf The case of type $\tilde{E}_7+\tilde{A}_1^{(2)}$.}
The dual graph is given by the following Figure \ref{graphE7-2}.

\begin{figure}[!htb]
 \begin{center}
\xy
(-30,25)*{};
@={(80,10),(90,10),(0,10),(10,10),(20,10),(30,10),(40,10),(50,10),(60,10),(70,10),(30,20)}@@{*{\bullet}};
(0,10)*{};(80,10)*{}**\dir{-};
(90,10)*{};(80,10)*{}**\dir{=};
(30,10)*{};(30,20)*{}**\dir{-};
(60,6)*{E_8};(50,6)*{E_7};(40,6)*{E_6};
(30,6)*{E_4};(20,6)*{E_3};(10,6)*{E_2};
(0,6)*{E_1};(80,6)*{E_{10}};(70,6)*{E_9};(35,20)*{E_5};(90,6)*{E_{11}};
\endxy 
 \end{center}
 \caption{}
 \label{graphE7-2}
\end{figure}
It follows from lemma \ref{D4diagram} and \ref{roots2} (2) that the vertices 
$E_1,\ldots, E_{10}$ are represented by $(-2)$-curves. 
There exists a parabolic subdiagram of type $\tilde{E}_7\oplus \tilde{A}_1$
which corresponds to a quasi-elliptic fibration with two singular fibers of 
type ${\rm III}^*$ and of type ${\rm III}$ (Lemmas \ref{fibration-isotropic}, \ref{exclude2}).  If the fibration is obtained from a Jacobian
fibration, then
the vertex $E_9$ is represented by a $(-1)$-root which contradicts Lemma \ref{roots2} (2).  
Therefore the fibration is induced from a Halphen surface of index 2, $E_9$ is the proper transform of the curve of cusps and the boundary components are the 
proper transforms of the components of the fiber of type ${\rm III}$.  Thus the number of boundary components is two.

\smallskip

(3)\ {\bf The case of type $\tilde{E}_7+\tilde{A}_1^{(1)}$.}
The dual graph is given by the following Figure \ref{graphE7-1}.

\begin{figure}[!htb]
 \begin{center}
\xy
(-30,25)*{};
@={(80,10),(90,10),(0,10),(10,10),(20,10),(30,10),(40,10),(50,10),(60,10),(70,10),(30,20)}@@{*{\bullet}};
(0,10)*{};(80,10)*{}**\dir{-};
(90,10)*{};(80,10)*{}**\dir{=};
(30,10)*{};(30,20)*{}**\dir{-};
(70,10)*{};(90,10)*{}**\crv{(80,20)};
(60,6)*{E_8};(50,6)*{E_7};(40,6)*{E_6};
(30,6)*{E_4};(20,6)*{E_3};(10,6)*{E_2};
(0,6)*{E_1};(80,6)*{E_{10}};(70,6)*{E_9};(35,20)*{E_5};(90,6)*{E_{11}};
\endxy
 \end{center}
 \caption{}
 \label{graphE7-1}
\end{figure}

It follows from lemma \ref{D4diagram} and \ref{roots2} (2) that all vertices 
are represented by $(-2)$-curves. 
There exists a parabolic subdiagram of type $\tilde{E}_7\oplus \tilde{A}_1$
which corresponds to a quasi-elliptic fibration with reducible fibers of type 
$(2{\rm III}^*, {\rm III})$
(Lemmas \ref{fibration-isotropic}, \ref{exclude2}).  Lemma \ref{roots2} (2) implies that a fiber of type ${\rm II}$ is blown up, the fiber of type ${\rm III}$ is not
and the 2-section $E_9$ is the proper transform of the curve of cusps.
Thus the fibration is a Halphen surface of index 2 
with a multiple fiber of type ${\rm III}^*$ 
and the number of boundary components is one.

\smallskip

(4)\ {\bf The case of type $\tilde{E}_6+\tilde{A}_2$.}
The dual graph is given by the following Figure \ref{graphE6}.

\begin{figure}[!htb]
 \begin{center}
  \includegraphics[width=70mm]{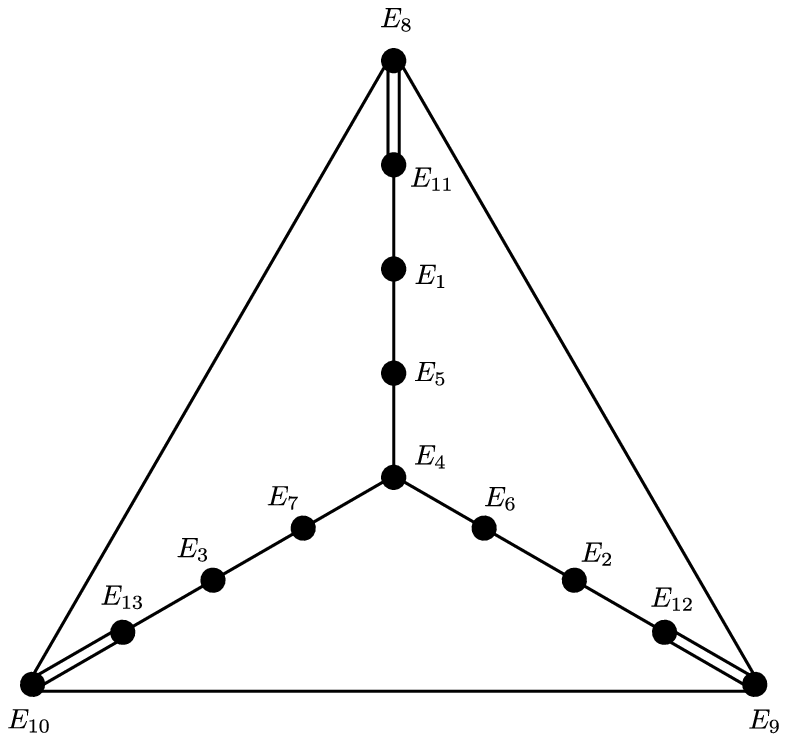}
 \end{center}
 \caption{}
 \label{graphE6}
\end{figure}

It follows from lemma \ref{D4diagram} and \ref{roots2} (2) that the vertices 
$E_1,\ldots, E_7, E_{11}, E_{12}, E_{13}$ are represented by $(-2)$-curves. 
There exists a parabolic subdiagram of type $\tilde{E}_6\oplus \tilde{A}_2$
which corresponds to an elliptic fibration with singular fibers of 
type $({\rm IV}^*, {\rm IV})$ or of type $({\rm IV}^*, {\rm I}_3, {\rm I}_1)$
by Propositions \ref{Lang}, \ref{Ito}.  If the fibration is obtained from a Jacobian fibration, then
a 2-section $E_{11}$, $E_{12}$ or $E_{13}$ is represented by a $(-1)$-root which 
contradicts Lemma \ref{roots2} (2).  
Therefore the fibration is obtained from a Halphen surface of index 2.
Note that the fiber of type ${\rm IV}$ does not occur because, otherwise,
the fiber of type ${\rm III}$ is blown up and then the $(-2)$-curve $E_{11}$ meets the
boundary component which is a contradiction.
Hence the number of boundary components is
one or three according to that the fiber blown up is of type ${\rm I}_1$ or of type
${\rm I}_3$.

\smallskip

(5) \ {\bf The case of type $\tilde{D}_8$.}
The dual graph is given by the following Figure \ref{graphD8}.

\begin{figure}[!htb]
 \begin{center}
\xy
(-50,25)*{};
@={(-10,10),(0,10),(10,10),(20,10),(30,10),(40,10),(50,10),(10,20),(60,10),(50,20)}@@{*{\bullet}};
(-10,10)*{};(60,10)*{}**\dir{-};
(10,10)*{};(10,20)*{}**\dir{-};
(50,10)*{};(50,20)*{}**\dir{-};
(60,6)*{E_2};(50,6)*{E_3};(40,6)*{E_4};
(30,6)*{E_5};(20,6)*{E_6};(10,6)*{E_7};
(0,6)*{E_9};(-10,6)*{E_{10}};(55,20)*{E_1};(15,20)*{E_8};
\endxy
 \end{center}
 \caption{}
 \label{graphD8}
\end{figure}

It follows from lemma \ref{D4diagram} and \ref{roots2} (2) that all vertices 
are represented by $(-2)$-curves. 
There exists a parabolic subdiagram of type $\tilde{D}_8$ which corresponds to 
a quasi-elliptic fibration with singular fibers of type $({\rm I}_4^*)$
(Lemmas \ref{fibration-isotropic}, \ref{exclude2}).  If the fibration is induced from a Jacobian fibration, then the vertex $E_{10}$ is a 
$(-1)$-root which contradicts Lemma \ref{roots2} (2).  Thus the fibration is obtained from a Halphen surface of index 2 with a multiple fiber of type ${\rm I}_4^*$, $E_{10}$ is the proper transform of the curve of cusps and the number of boundary component is one.

\smallskip

(6)\ {\bf The case of type VII.} 
The dual graph of type VII is given by the following Figure \ref{graphVII}.
\begin{figure}[!htb]
 \begin{center}
  \includegraphics[width=60mm]{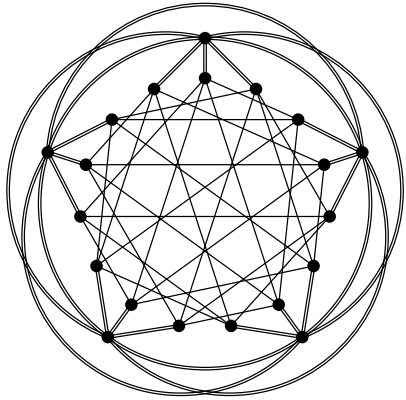}
 \end{center}
 \caption{Type VII}
 \label{graphVII}
\end{figure}

Note that the dual graph contains five vertices forming a complete graph with double edges
and the remaining fifteen vertices meet together with a single edge. 
Also the dual graph contains parabolic subdiagrams of type $\tilde{A}_4\oplus \tilde{A}_4$ 
corresponding to
elliptic fibrations of type $({\rm I}_5,{\rm I}_5, {\rm I}_1, {\rm I}_1)$ 
by Propositions \ref{Lang}, \ref{Ito}.  
By Lemma \ref{roots2} (2), all of the fifteen vertices are represented by only 
$(-1)$-roots or by only $(-2)$-curves.  
Therefore if they are represented by only $(-1)$-roots, then 
the Coble surface is obtained by blowing up all the singular points of two fibers of type 
${\rm I}_5$.  This implies that the number of boundary components is 10.  If they are represented by only $(-2)$-curves, then 
the surface is obtained by blowing up one or two fibers of type ${\rm I}_1$.  
To get the complete graph with the five vertices and double edges,
the number of the boundary components is 2.
Thus we conclude that the number of boundary components is 2 or 10.

\smallskip

(7)\ {\bf The case of type VIII.} 
The dual graph of type VIII is given by the following
Figure \ref{graphVIII}.

\begin{figure}[!htb]
 \begin{center}
  \includegraphics[width=70mm]{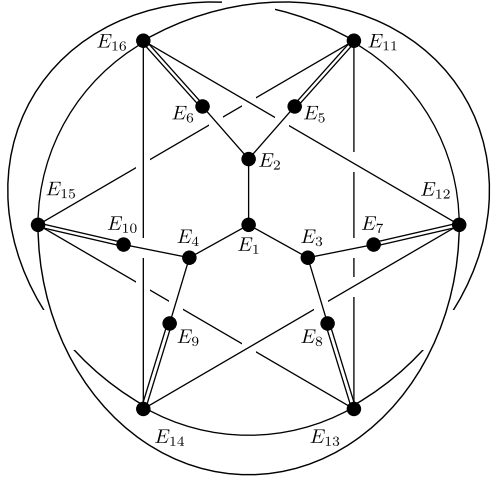}
 \end{center}
 \caption{Type VIII}
 \label{graphVIII}
\end{figure}

It follows from lemma \ref{D4diagram} and \ref{roots2} (2) that the vertices 
$E_1,\ldots, E_{10}$ are represented by $(-2)$-curves. 
The dual graph contains parabolic subdiagrams of type $\tilde{D}_5\oplus \tilde{A}_3$
which correspond to
elliptic fibrations with singular fibers of type $({\rm I}_1^*,{\rm I}_4)$ by Propositions \ref{Lang}, \ref{Ito}.  This 
implies that the number of boundary components is 4.  

Thus we have proved that the number of boundary components of a Coble surface with finite automorphism group is given as in Theorem \ref{main2}.
\end{proof}

\section{Existence and moduli}\label{sec6}

In this section we show the existence of Coble surfaces with finite automorphism group
in Table \ref{mainTable} and determine their moduli.
For each type, we give two constructions of such surfaces, that is, as blowing-ups of a pencil of sextic curves on ${\bf P}^2$ (Proposition \ref{Halphen}) and as a quotient of a non-singular surface by a rational derivation (\S \ref{derivation}).

\begin{example}\label{exE8} {\bf A Coble surface of type $\tilde{E}_8$ with one boundary component.}

A Coble surface of type $\tilde{E}_8$ with one boundary component has a unique quasi-elliptic fibration with a multiple fiber of type ${\rm II}^*$. It is obtained from a Halphen surface of index 2 with a multiple fiber of
type ${\rm II}^*$ by blowing up the singular point of a fiber $F$ of type ${\rm II}$.  
The vertex $E_{10}$ in Figure \ref{graphE8} is represented by the proper transform of the curve of cusps.
Contracting the $(-1)$-curve (the curve of cusps) on the Halphen surface and then the images of
$E_9,\ldots, E_5, E_3, E_2, E_1$ except $E_4$ successively, we obtain ${\bf P}^2$.  The image of $F$
is a sextic curve $C$ with the unique singular point $q$ and the image of $E_4$ is a line $\ell$.  The curves $C$ and $\ell$ meet only at the point $q$.
Let $(x:y:z)$ be homogeneous coordinates of ${\bf P}^2$.
We may assume that $q=(0:0:1)$ and $\ell$ is defined by $x=0$.
An elementary but long calculation shows that 
such curve $C$ is given by 
\begin{equation}\label{E8eq}
C : \ y^6 + x^5y + x^2z^4 = 0.
\end{equation}
Note that the pencil of lines through the point $q$ gives a purely inseparable
covering $C \to {\bf P}^1$ of degree 4 and hence $C$ is a rational curve.
Consider the pencil of sextic curves given by 
$$\calP = \{ sx^6 + t(y^6 + x^5x + x^2x^4)\}_{(s:t)\in {\bf P}^1}.$$
Blowing up the base points of $\calP$, 
we obtain a quasi-elliptic fibration (a Halphen surface of index 2) with a unique 
multiple singular fiber of type ${\rm II}^*$ over
the point $(s:t)=(1:0)$.  
Let $E$ be the $(-1)$-curve which is the exceptional curve of the last blowing-up.
Then $E$ is the curve of cusps of this fibration.  
Blowing up the singular point of the fiber $F$ of type ${\rm II}$ over the point 
$(s:t)=(0:1)$, we obtain a Coble surface
$S$ with a boundary component $B$ where $B$ is the proper transform of $F$.
The surface $S$ contains 10 $(-2)$-curves forming the desired dual graph.
Thus we obtain a Coble surface $S$ of type $\tilde{E}_8$.
By construction, any automorphism of $S$ is induced from one of ${\bf P}^2$ preserving 
the pencil of sextic curves.  
Let $\zeta_5$ be a primitive 5th root of unity.  Then 
$$\sigma: (x:y:z)\to (x:\zeta_5 y: \zeta_5^{-1}z)$$
induces an automorphism of $\calP$ of order 5.  A direct calculation shows that $\sigma$ is the only
such automorphisms.

Next we give a Coble surface of type $\tilde{E}_8$ as a quotient of
a rational surface by a derivation.  In the following
we use the same notation as in \cite[\S 10.2]{KKM}.  In \cite{KKM}, 
to construct classical Enriques surfaces of type $\tilde{E}_8$ (the ones with the dual graph in Figure \ref{graphE8}), we
considered the surface $Y$, blown up a smooth quadric surface, 
with the canonical divisor
$$K_{Y} = -(2F_0+2E_1+E_3+3E_5+4E_6+3E_7+2E_8+4E_9+5E_{10}+6E_{11}+8E_{12}+4E_{13}+6E_{14})$$
where $F_0, E_i$ are given in \cite[Figure 18]{KKM}, and the following derivation defined by
\begin{equation}\label{E8C-der}
D= {1\over x^3y^2} \left(x^4y^2\frac{\partial}{\partial x} + (x^2 + ax^4y^4 + y^4)\frac{\partial}{\partial y}\right), \quad  a  \in k^*,
\end{equation}
where $(x,y)$ is inhomogeneous coordinates of the quadric surface.

Now consider the case $a=0$.  Then we have the following Lemma (compare this with \cite[Lemmas 10.6, 10.7]{KKM}:

\begin{lemma}\label{E8pole}
{\rm (i)} The integral curves with respect to $D$ are 
$E_1, E_3, E_5, E_7, E_8, E_9, E_{10}$.

{\rm (ii)} $(D) =-(2F_0+3E_1-E_2+2E_3+4E_5+4E_6+3E_7+2E_8+4E_9+5E_{10}+6E_{11}+8E_{12}+4E_{13}+6E_{14})$.

{\rm (iii)} $(D)^2 = -12$. 

{\rm (iv)} $K_{Y}\cdot (D) = -4.$
\end{lemma}

Let $Y^D$ be the quotient surface of $Y$ by $D$ and $\pi:Y\to Y^D$ the canonical map.  
By the formula (\ref{euler}) and Lemma \ref{E8pole} (iii), (iv), 
$D$ is divisorial and hence $Y^D$ is smooth.
It follows from the canonical divisor formula
(\ref{canonical}) and Lemma \ref{E8pole} that
$\pi^*K_{Y^D}= E_1 -E_2 + E_3+E_5$.  Denote by $\bar{E}_i$ the image of $E_i$ on $Y^D$.
Since, by construction, $E_1^2=E_2^2=E_3^2=E_5^2=-2$ (see \cite[Figure 18]{KKM}) and $E_1, E_3, E_5$ are integral and $E_2$ is not, we have 
$\bar{E}_1^2=\bar{E}_3^2=\bar{E}_5^2=-1$ and $\bar{E}_2^2=-4$ (Proposition \ref{insep}).
Again by Proposition \ref{insep}, we have $2K_{Y^D}= 2\bar{E}_1-\bar{E}_2+2\bar{E}_3+2\bar{E}_5$.
By contracting $(-1)$-curves $\bar{E}_1, \bar{E}_3, \bar{E}_5$, we obtain a non-singular surface $S$ with $|-2K_S| = \{\bar{E}_2\}$.  The images of the curves 
$F_0, E_6,\ldots, E_{14}$ on $S$ are $(-2)$-curves forming the dual graph given 
in Figure \ref{graphE8}.
The image of $E_4$ is a $(-1)$-curve meeting with $\bar{E}_2$ at a point
with multiplicity 2.  Thus we have a Coble surface with the dual graph given in Figure \ref{graphE8} and have the following theorem.

\begin{theorem}\label{E8thm}
The surface $S$ is a unique Coble surface with the dual graph given in Figure {\rm \ref{graphE8}}.  The automorphism group 
${\rm Aut}(S)$ is isomorphic to ${\bf Z}/5{\bf Z}$.  The surface $S$ is a specialization of
classical Enriques surfaces of type $\tilde{E}_8$.
\end{theorem}

\end{example}

\begin{example}\label{exE7-2} {\bf A Coble surface of type $\tilde{E}_7+\tilde{A}_1^{(2)}$ with two boundary components.}

A Coble surface $S$ of type $\tilde{E}_7+\tilde{A}_1^{(2)}$ 
with two boundary components $B_1+B_2$ has a unique quasi-elliptic fibration $f$ 
with a multiple fiber of type ${\rm III}^*$.
It is obtained from a Halphen surface of index 2 with a multiple fiber of
type ${\rm III}^*$ and a fiber $F$ of type ${\rm III}$ 
by blowing up the singular point (and infinitely near point) of $F$ (see Figure \ref{Halsing} (type {\rm III})).  Then $E_9$ is the proper transform of
the curve of cusps of $f$ and $E_{10}$ is the proper transform of the exceptional 
curve of the first blowing-up.  Let $E$ be the exceptional curve of the second blowing-up.
Then $B_1, B_2$ are the proper transforms of the components of $F$
and $E_{11} = 2E + {1\over 2}B_1+{1\over 2}B_2$.

On the other hand, there exists a genus 1 fibration $f'$ on $S$ 
with a singular fiber of type ${\rm II}^*$ given by
$$2E_2+4E_3+6E_4+3E_5+5E_6+4E_7+3E_8+2E_9+E_{10}.$$
This fibration is quasi-elliptic because the fiber of type ${\rm II}^*$ does not
contain the component $E_1$ of the conductrix (in fact $E_1$ is the curve of cusps of 
$f'$).  Moreover the fibration $f'$ is Jacobian because $E$ is a section of $f'$.
Let $F_i$ be the $(-1)$-curve such that $2F_i + B_i$ is the fiber of $f'$ 
containing $B_i$ $(i=1,2)$.  We can easily see that 
$F_1\cdot E_1=F_2\cdot E_1 =1, F_1\cdot F_2=F_1\cdot E = F_2\cdot E=0$, and 
$F_i$ and $B_i$ meet at a point with multiplicity 2.

Now we blow down curves $F_1, F_2, E, E_{10}, E_9,\ldots, E_6, E_4, E_3, E_2$ successively.  
Then we obtain ${\bf P}^2$.  Let $(x:y:z)$ be homogeneous coordinate of ${\bf P}^2$.
The images of $B_1, B_2$ are cubics $C_1, C_2$ with a cusp
and those of $E_5, E_1$ are lines $\ell_1, \ell_2$, respectively. Denote by $p_1, p_2$ the
cusp of $C_1, C_2$ respectively. 
The two cubics $C_1, C_2$ meet together
at a point $p_0$ with multiplicity 9, the line $\ell_1$ meets $C_1, C_2$ at $p_0$ with multiplicity 3, and the line $\ell_2$ passes through $p_0, p_1, p_2$.
By changing the coordinates we may assume that $p_0=(0:0:1)$, $p_1=(1:0:0)$, $p_2=(1:0:1)$,
$\ell_1$ (resp. $\ell_2$) are defined by $x=0$ (resp. $y=0$).
It follows from elementary calculations that
$C_1$ and $C_2$, up to projective transformations, are given by
$$C_1: y^3+xz^2=0; \ C_2: x^3+y^3+xz^2=0.$$
Conversely, let 
$$\calP = \{s(x^2y)^2 + t(y^3+xz^2)(x^3+y^3+xz^2)\}_{(s:t)\in {\bf P}^1}$$
be a pencil of sextic curves.   Blowing up the base points $p_0, p_1, p_2$ successively, 
we obtain a quasi-elliptic fibration with a multiple fiber of type ${\rm III}^*$ over
the point $(s:t)=(1:0)$ and a fiber $\tilde{C}_1+\tilde{C}_2$ of type ${\rm III}$ over the point $(s:t)=(0:1)$ where $\tilde{C}_i$ is the proper transform of
$C_i$.  Thus we get
a Coble surface $S$ of type $\tilde{E}_7+\tilde{A}_1^{(2)}$ with two boundary components.
Note that the projective transformation 
$$(x:y:z)\to (x:y: x+z)$$
preserves each line $\ell_1, \ell_2$ and interchanging $C_1$ and $C_2$.
A direct calculation shows that ${\rm Aut}(S)$ is generated by this involution.

Next we give a Coble surface of type $\tilde{E}_7+\tilde{A}_1^{(2)}$ as a quotient of 
a rational surface by a derivation.
We use the same notation as in \cite[\S 11.2]{KKM}.
To construct classical Enriques surfaces with the dual graph in Figure \ref{graphE7-2}, we
considered a conic bundle $f : R\to {\bf P}^1$ defined by
$$S(aX_0^2+bX_2^2)+T(X_1^2+aX_1X_2+bX_0X_2)=0, \ (S:T)\in {\bf P}^1,\ a, b \in k^*$$
in the fiber space ${\bf P}^2\times {\bf P}^1\to {\bf P}^1$, and
the following derivation defined by
\begin{equation}\label{E7C-der}
D= {1\over s} \left(a(s^2+c)\frac{\partial}{\partial x} + (as^2x^2 + bc)\frac{\partial}{\partial y}\right), \quad  b\ne a^2c  \in k^*,
\end{equation}
where $(x,y,s)=(X_0/X_2, X_1/X_2, S/T)$ and $c$ is a root of the equation $t^2 +(b/a)t + 1=0$.

The surface $R$ has two rational double points and the derivation has isolated singularities.  By blowing up $R$ successively, we obtain a non-singular surface $Y$ with
$$K_Y= -(F_0+2E_3+E_6+E_7+2E_8+2E_9+2E_{10}+3E_{11}+4E_{12}+4E_{13}+2E_{14})$$
and a divisorial derivation.
In case $b=0, c=1$, we obtained Enriques surfaces of type 
$\tilde{E}_7+\tilde{A}_1^{(2)}$ birationally isomorphic to the quoteint of $Y$ by the derivation.

Now consider the case $b=c=0$, that is, 
$$D= s\left(\frac{\partial}{\partial x} + x^2\frac{\partial}{\partial y}\right).$$
Then we have the following Lemma.

\begin{lemma}\label{E7-2pole}
{\rm (i)} The integral curves with respect to $D$ are 
$F_0, E_4, E_6, E_7, E_8, E_9, E_{11}$ in \cite[Figure 24]{KKM}.

{\rm (ii)} $(D) =-(F_0-E_1-E_2+2E_3++2E_6+2E_7+2E_8+2E_9+2E_{10}+3E_{11}+4E_{12}+4E_{13}+2E_{14})$.

{\rm (iii)} $(D)^2 = -12$. 

{\rm (iv)} $K_{Y}\cdot (D) = -4.$
\end{lemma}

Let $Y^D$ be the quotient of $Y$ by $D$ and $\pi:Y\to Y^D$ the canonical map.  
By the formula (\ref{euler}) and Lemma \ref{E7-2pole}, $Y^D$ is smooth.
It follows from the canonical divisor formula
(\ref{canonical}) and Lemma \ref{E7-2pole} that
$\pi^*K_{Y^D}= -E_1 -E_2 + E_6+E_7$.  Denote by $\bar{E}_i$ the image of $E_i$ on $Y^D$.
Since, by construction, $E_1^2=E_2^2=E_6^2=E_7^2=-2$ (see \cite[Figure 24]{KKM}) and $E_6, E_7$ are integral and $E_1, E_2$ are not,
$\bar{E}_6^2=\bar{E}_7^2=-1$ and $\bar{E}_1^2=\bar{E}_2^2=-4$ (Proposition \ref{insep}).
Again by Proposition \ref{insep}, we have $2K_{Y^D}= -\bar{E}_1-\bar{E}_2+2\bar{E}_6+2\bar{E}_7$.
By contracting $(-1)$-curves $\bar{E}_6, \bar{E}_7$, we obtain a non-singular surface $S$ with $|-2K_S| = \{\bar{E}_1+\bar{E}_2\}$.  The images of the curves 
$E_3,E_8, \ldots, E_{14}$ on $S$ are $(-2)$-curves forming the fiber of type 
${\rm III}^*$.  The images of $F_0, E_5$ are $(-2)$-curves and that of $E_4$ is a $(-1)$-curve meeting with the boundary components $\bar{E}_1, \bar{E}_2$.  Thus we have a Coble surface of type $\tilde{E}_7+\tilde{A}_1^{(2)}$.
Thus we have proved the following theorem.

\begin{theorem}\label{E7B2}
A Coble surface $S$ of type $\tilde{E}_7+\tilde{A}_1^{(2)}$ with two boundary components
is obtained from the pencil $\calP$ as above, and is unique up to isomorphisms.
The automorphism group ${\rm Aut}(S)$ is isomorphic to ${\bf Z}/2{\bf Z}$.
The surface $S$ is a specialization of classical Enriques surfaces of type $\tilde{E}_7+\tilde{A}_1^{(2)}$.
\end{theorem}

\end{example}

\begin{example}\label{exE7-1} {\bf A Coble surface of type $\tilde{E}_7+\tilde{A}_1^{(1)}$ with one boundary component.}

A Coble surface $S$ of type $\tilde{E}_7+\tilde{A}_1^{(1)}$ 
with one boundary component $B$ has a unique quasi-elliptic fibration $f$ 
with a multiple fiber of type ${\rm III}^*$.
It is obtained from a Halphen surface of index 2 with singular fibers of
type $(2{\rm III}^*, {\rm III})$ by blowing up the singular point of a fiber of type 
${\rm II}$.  The proper transform of the fiber is the boundary component $B$.

Note that instead of blowing up the singular point of a fiber of type II of the Halphen surface, by blowing up the singular point of the fiber of type III we obtain a Coble surface $S$ of type $\tilde{E}_7+\tilde{A}_1^{(2)}$.  Thus we can use the construction of the previous Example.  Note that the last blowing-up depends on the choice of
a singular fiber of type II of the Halphen surface.

Next we give a Coble surface of type $\tilde{E}_7+\tilde{A}_1^{(1)}$ as a quotient of a rational surface by a derivation.
We use the same notation as in \cite[\S 11.1]{KKM}.
We assume that $c=0$ and $b\ne 0$ in the equation (\ref{E7C-der}).   Note that the conic bundle $R$
has a parameter $b$ which is the difference between this case and the previous case.
The canonical divisor of $Y$ is given by
$$K_Y=-(F_0+2E_2+E_7+E_8+2E_9+2E_{10}+2E_{11}+3E_{12}+4E_{13}+4E_{14}
+2E_{15}).$$
Then we have the following Lemma.

\begin{lemma}\label{E7-1pole}
{\rm (i)} The integral curves with respect to $D$ are 
$F_0, E_5, E_7, E_8, E_9, E_{10}, E_{12}$ in \cite[Figure 21]{KKM}.

{\rm (ii)} $(D) =-(F_0-E_1+2E_2+E_5+2E_7+2E_8+2E_9+2E_{10}+2E_{11}+3E_{12}+4E_{13}+4E_{14}
+2E_{15})$.

{\rm (iii)} $(D)^2 = -12$. 

{\rm (iv)} $K_{Y}\cdot (D) = -4.$
\end{lemma}

Let $Y^D$ be the quotient of $Y$ by $D$ and $\pi:Y\to Y^D$ the canonical map.  
By the formula (\ref{euler}) and Lemma \ref{E7-2pole}, $Y^D$ is smooth.
It follows from the canonical divisor formula
(\ref{canonical}) and Lemma \ref{E7-1pole} that
$\pi^*K_{Y^D}= -E_1 + E_5 + E_7+E_8$.  Denote by $\bar{E}_i$ the image of $E_i$ on $Y^D$.
Since, by construction, $E_1^2=E_5^2=E_7^2=E_8^2=-2$ (see \cite[Figure 21]{KKM}) and $E_5, E_7, E_8$ are integral and $E_1$ is not,
$\bar{E}_5^2=\bar{E}_7^2=\bar{E}_8^2=-1$ and $\bar{E}_1^2=-4$ (Proposition \ref{insep}).
Again by Proposition \ref{insep}, we have $2K_{Y^D}= -\bar{E}_1+2\bar{E}_5+2\bar{E}_7+2\bar{E}_8$.
By contracting $(-1)$-curves $\bar{E}_5, \bar{E}_7, \bar{E}_8$, we obtain a non-singular surface $S$ with $|-2K_S| = \{\bar{E}_1\}$.  The images of the curves 
$E_2, E_9, \ldots, E_{15}$ on $S$ are $(-2)$-curves forming the fiber of type 
${\rm III}^*$.  The images of $E_3, E_4$ are $(-2)$-curves forming the fiber of type III. 
The image of curve $F_0$ is a $(-2)$-curve and 
that of $E_6$ is a $(-1)$-curve meeting the boundary components $\bar{E}_1$ with multiplicity 2.  
Thus we have a Coble surface with the dual graph of type $\tilde{E}_7+\tilde{A}_1^{(1)}$.
Thus we have proved the following theorem.

\begin{theorem}\label{E7B1}
Coble surfaces $S$ of type $\tilde{E}_7+\tilde{A}_1^{(1)}$ with one boundary component
form a $1$-dimensional irreducible family. 
The automorphism group ${\rm Aut}(S)$ is isomorphic to ${\bf Z}/2{\bf Z}$.
The surfaces $S$ are specializations of classical Enriques surfaces of type $\tilde{E}_7+\tilde{A}_1^{(1)}$.
\end{theorem}
\end{example}

\begin{example}\label{exE6-3} {\bf A Coble surface of type $\tilde{E}_6+\tilde{A}_2$ 
with three boundary components.}

A Coble surface $S$ of type $\tilde{E}_6+\tilde{A}_2$ 
with three boundary components $B_1+B_2+B_3$ has a unique elliptic fibration $f$ 
with a multiple fiber of type ${\rm IV}^*$.
It is obtained from a Halphen surface of index 2 with singular fibers of
type $(2{\rm IV}^*, {\rm I}_3, {\rm I}_1)$ by blowing up the singular points of 
the fiber of type ${\rm I}_3$. 
The proper transforms of the components of the fiber of type ${\rm I}_3$ are the boundary components $B_1, B_2, B_3$.  Let $F_1, F_2, F_3$ be $(-1)$-curves such that
$2F_1 + {1\over 2}B_2+{1\over 2}B_3$, $2F_2 + {1\over 2}B_1+{1\over 2}B_3$ or 
$2F_3 + {1\over 2}B_1+{1\over 2}B_2$ gives $E_8, E_9$ or $E_{10}$ in Figure \ref{graphE6},  respectively.

On the other hand, there exist three genus 1 fibrations $f_i$ $(i=1,2,3)$
with  singular fibers of type $(\tilde{E}_7, 2\tilde{A}_1)$.  The linear system
$|2(E_8+E_{11})|$ (resp. $|2(E_9+E_{12})|$, $|2(E_{10}+E_{13})|$)
gives $f_1$ (resp. $f_2$, $f_3$).  
By Propositions \ref{Lang}, \ref{Ito}, they are quasi-elliptic.
Each of them is obtained from a Jacobian fibration with reducible singular fibers of
type $({\rm III}^*, {\rm III})$ by blowing up the singular point (and the infinitely near point) of the fiber of type
${\rm III}$ and the singular point of a fiber of type ${\rm II}$.
Thus there exist three $(-1)$-curves
$F_1', F_2', F_3'$ such that $F_i'$ meets $B_i$ at a point with multiplicity 2 and
$F_i'\cdot E_i=1$ $(i=1,2,3)$.  The Mordell-Weil group of $f_i$ acts on $S$ as automorphisms.  For example, $f_1$
has two sections $F_2, F_3$ which define an involution switching $F_2'$ and $F_3'$ and fixing $F_1'$ pointwisely.  This implies that all $F_1', F_2', F_3'$ meet at one point.
Moreover they generate a symmetric group $\mathfrak{S}_3$.

Now by contracting $F_1, F_2, F_3$ and then $E_{11}, E_{12}, E_{13}, E_1, E_2, E_3, E_5, E_6, E_7$ successively, we obtain ${\bf P}^2$.
The images of $B_1, B_2, B_3$ are smooth conics $Q_1, Q_2, Q_3$ and those of $E_4, F_1', F_2', F_3'$ are lines denoted by $\ell, \ell_1, \ell_2,\ell_3$ respectively.  Two conics $Q_i$ and $Q_j$ meet at a point $p_{ij}$ with multiplicity 4
and $\ell$ passes through three points $p_{12}, p_{13}, p_{23}$. Three lines $\ell_1, \ell_2, \ell_3$ meet at a point $p_0$.  The line $\ell_i$ tangents to $Q_i$ at a point $p_i$
and tangents to both $Q_j$, $Q_k$ at the point $p_{jk}$ $(\{i,j,k\}=\{1,2,3\})$.
By changing coordinates, we may assume that
$$p_1=(1:0:0),\ p_2=(0:1:0),\ p_3=(0:0:1),\ p_0=(1:1:1).$$
Then we have
$$\ell_1: y+z=0,\ \ell_2: x+z=0,\ \ell_3: x+y=0.$$
Since the action of $\mathfrak{S}_3$ descends to the one acting on ${\bf P}^2$ as permutations of coordinates and
$\ell$ is invariant under this action, we have
$$\ell: x+y+z=0.$$
Therefore we have 
$$p_{12}=(1:1:0),\ p_{13}=(1:0:1),\ p_{23}=(0:1:1).$$
Then one can easily see that the conics are given by
$$Q_1: y^2+z^2+xy + yz + zx=0,$$
$$Q_2: x^2+z^2+xy + yz + zx=0,$$
$$Q_3: x^2+y^2+xy + yz + zx=0.$$

Conversely consider the pencil $\calP$ of sextic curves 
\begin{equation}\label{E6pencil}
\calP = \{C_{(s:t)}\}=\{s\ell^6 + tQ_1Q_2Q_3\}_{(s:t)\in {\bf P}^1}.
\end{equation} 
Blowing up the base points of $\calP$, we have a Halphen surface with singular fibers of type $(2{\rm IV}^*, {\rm I}_3, {\rm I}_1)$, and hence we obtain
a Coble surface $X$ of type $\tilde{E}_6+\tilde{A}_2$ 
with three boundary components.
The three fibrations $f_i$ are Jacobian and their Mordell-Weil groups generate
$\mathfrak{S}_3$ a subgroup of ${\rm Aut}(S)$ which is also induced from permutations of the coordinates $(x:y:z)$.
One can see that there are no more projective transformations preserving the lines and conics.

Next we give a Coble surface of type $\tilde{E}_6+\tilde{A}_2$ as a quotient of 
a rational surface by a derivation.
We use the same notation as in \cite[\S 7.2]{KKM}.
To construct classical Enriques surfaces with the dual graph in Figure \ref{graphE6}, we
considered a rational elliptic surface defined by the Weierstrass equation
$$y^2+xy + t^2y = x^3$$
which has a singular of type ${\rm I}_6$ over the point $t=0$, a singular fiber of type ${\rm I}_2$ over the point $t=1$ and
a singular fiber of type ${\rm IV}$ over the point $t=\infty$, and 
the derivation $D$ defined by
\begin{equation}\label{E6-der}
D= (t+a)\frac{\partial}{\partial t} + (x + t^2)\frac{\partial}{\partial x}, \quad  a\in k\setminus \{0,1\}.
\end{equation}

Now consider the case $a=0$.  The derivation has an isolated singularity at the singular point of the fiber of type IV.  We blow up the surface successively and then obtain
a non-singular surface $Y$ with
$$K_Y=-(F_\infty +E_{\infty,1}+E_{\infty,2})-2(E_{\infty,3}+E_{\infty,4}+E_{\infty,5}+E_{\infty,6})$$
where the notations are as in \cite[Figure 7]{KKM}.
Then we have the following Lemma.

\begin{lemma}\label{E63pole}
{\rm (i)} The integral curves with respect to $D$ are 
$E_1, F_0, E_{0,3}, E_{0,4}, F_\infty, E_{\infty,1}$, $E_{\infty,2}$, $E_{\infty,3}$.

{\rm (ii)} $(D) =E_{0,1}+E_{0,2}+E_{0,5}-(E_1+ F_\infty+E_{\infty,1}+E_{\infty,2})- 2(E_{\infty,3}+E_{\infty,4}+E_{\infty,5}+E_{\infty,6})$.

{\rm (iii)} $(D)^2 = -12$. 

{\rm (iv)} $K_{Y}\cdot (D) = -4.$
\end{lemma}

Let $Y^D$ be the quotient of $Y$ by $D$ and $\pi:Y\to Y^D$ the canonical map.  
By the formula (\ref{euler}) and Lemma \ref{E7-2pole}, $Y^D$ is smooth.
It follows from the canonical divisor formula
(\ref{canonical}) and Lemma \ref{E63pole} that
$\pi^*K_{Y^D}= -E_{0,1} -E_{0,2}-E_{0,5} + E_1$.  Denote by $\bar{E}_{0,1}, \bar{E}_{0,2}, \bar{E}_{0,5}, \bar{E}_1$ the images of $E_{0,1}, E_{0,2}, E_{0,5}, E_1$ on $Y^D$, respectively.
Since $E_{0,1}, E_{0,2}, E_{0,5}, E_1$ are $(-2)$-curves (see \cite[Figure 7]{KKM}), 
$E_1$ is integral and $E_{0,1}, E_{0,2}, E_{0,5}$ are not (Lemma \ref{E63pole}),
$\bar{E}_{0,1}, \bar{E}_{0,2}, \bar{E}_{0,5}$ are $(-4)$-curves and $\bar{E}_1$ is a
$(-1)$-curve (Proposition \ref{insep}).
Again by Proposition \ref{insep}, we have $2K_{Y^D}= -\bar{E}_{0,1} -\bar{E}_{0,2}-\bar{E}_{0,5} + 2\bar{E}_1$.
By contracting $\bar{E}_1$, we obtain a non-singular surface $S$ with $|-2K_S| = \{\bar{E}_{0,1} +\bar{E}_{0,2}+\bar{E}_{0,5}\}$.  The images of the curves 
$s_1, s_2, s_3, F_\infty, E_{\infty,i} \ (i=1,\ldots, 6)$ on $S$ are $(-2)$-curves. The images of $F_1$ is the fiber of type ${\rm I}_1$. 
Thus we have a Coble surface of type $\tilde{E}_6+\tilde{A}_2$ and with 3 boundary components.

\begin{theorem}\label{E6B3}
The Coble surface $S$ of type $\tilde{E}_6+\tilde{A}_2$ with three 
boundary components is induced from the pencil $\calP$ of sextic curves and is unique up to isomorphisms.
The automorphism group ${\rm Aut}(S)$ is isomorphic to $\mathfrak{S}_3$.
The surface $S$ is a specialization of classical Enriques surfaces of type $\tilde{E}_6+\tilde{A}_2$.
\end{theorem}

\end{example}

\begin{example}\label{exE6-1} {\bf A Coble surface of type 
$\tilde{E}_6+\tilde{A}_2$ with one boundary component.}

We consider the same situation as in Example \ref{exE6-3}. 
Blowing up the base points of the pencil $\calP$ given by the equation (\ref{E6pencil}), 
we have obtained a Halphen surface with singular fibers of
type $(2{\rm IV}^*, {\rm I}_3, {\rm I}_1)$.
Then by blowing up the singular point of the fiber of type ${\rm I}_1$, we obtain
a Coble surface $S$ of type 
$\tilde{E}_6+\tilde{A}_2$ with one boundary component.
The proper transform of $C_{(1:1)}$ in (\ref{E6pencil}) is the boundary component of $X$.

Conversely one can easily see that any Coble surface of type 
$\tilde{E}_6+\tilde{A}_2$ with one boundary component is obtained 
from the pencil $\calP$ and hence it is unique.

Next we give a Coble surface of type $\tilde{E}_6+\tilde{A}_2$ with one boundary component  as a quotient of a rational surface by a derivation.
We use the same notation as in the previous Example \ref{exE6-3} and 
consider the case $a=1$.  Then we have the following Lemma.

\begin{lemma}\label{E6-1pole}
{\rm (i)} The integral curves with respect to $D$ are 
$F_1, E_{0,1}, E_{0,2}, E_{0,5}, F_\infty, E_{\infty,1}$, $E_{\infty,2}$, $E_{\infty,3}$.

{\rm (ii)} $(D) =E_1-(E_{0,1}+E_{0,2}+E_{0,5}+ F_\infty+E_{\infty,1}+E_{\infty,2})- 2(E_{\infty,3}+E_{\infty,4}+E_{\infty,5}+E_{\infty,6})$.

{\rm (iii)} $(D)^2 = -12$. 

{\rm (iv)} $K_{Y}\cdot (D) = -4.$
\end{lemma}

Let $Y^D$ be the quotient of $Y$ by $D$ and $\pi:Y\to Y^D$ the canonical map.  
By the formula (\ref{euler}) and Lemma \ref{E7-2pole}, $Y^D$ is smooth.
It follows from the canonical divisor formula
(\ref{canonical}) and Lemma \ref{E63pole} that
$\pi^*K_{Y^D}= E_{0,1} +E_{0,2}+E_{0,5}-E_1$.  Denote by $\bar{E}_{0,1}, \bar{E}_{0,2}, \bar{E}_{0,5}, \bar{E}_1$ the images of $E_{0,1}, E_{0,2}, E_{0,5}, E_1$ on $Y^D$, respectively.
Since $E_{0,1}, E_{0,2}, E_{0,5}, E_1$ are $(-2)$-curves (see \cite[Figure 7]{KKM}), 
$E_{0,1}, E_{0,2}, E_{0,5}$ are integral and $E_1$ is not (Lemma \ref{E63pole}),
$\bar{E}_{0,1}, \bar{E}_{0,2}, \bar{E}_{0,5}$ are $(-1)$-curves and $\bar{E}_1$ is a
$(-4)$-curve (Proposition \ref{insep}).
Again by Proposition \ref{insep}, we have $2K_{Y^D}= 2\bar{E}_{0,1} +2\bar{E}_{0,2}+2\bar{E}_{0,5} - \bar{E}_1$.
By contracting $\bar{E}_{0,1}, \bar{E}_{0,2}, \bar{E}_{0,5}$, we obtain a non-singular surface $S$ with $|-2K_S| = \{\bar{E}_1\}$.  
Thus we have a Coble surface of type $\tilde{E}_6+\tilde{A}_2$ and with one boundary component.

\begin{theorem}\label{E6B1}
The Coble surface $S$ of type $\tilde{E}_6+\tilde{A}_2$ with one
boundary component is induced from the pencil $\calP$  of sextic curves given in 
$(\ref{E6pencil})$, and is unique up to isomorphisms.
The automorphism group ${\rm Aut}(S)$ is isomorphic to $\mathfrak{S}_3$.
The surface $S$ is a specialization of classical Enriques surfaces of type $\tilde{E}_6+\tilde{A}_2$.
\end{theorem}

\end{example}

\begin{example}\label{exD8} {\bf A Coble surface of type $\tilde{D}_8$ with one boundary component.}

Let $S$ be a Coble surface of type $\tilde{D}_8$ with one boundary component $B$.
Note that $S$ has a unique quasi-elliptic fibration $f$ with a multiple fiber 
$$F=E_1+E_2+E_8+E_9+2(E_3+E_4+E_5+E_6 +E_7)$$ 
of type ${\rm I}_4^*$ and with a 2-section $E_{10}$, which is obtained from a rational 
quasi-elliptic fibration (a Halphen surface of index 2) 
by blowing up the singular point of a fiber of type ${\rm II}$.
Let $E_0$ be the $(-1)$-curve on $S$ such that $2E_0+B$ is the fiber of $f$.
Also $S$ has two genus 1 fibrations $f_i$ with a singular fiber 
$$F_i= 2E_{10}+4E_9+3E_8+6E_7+5E_6+4E_5+3E_4+2E_3+E_i$$ 
of type ${\rm II}^*$ $(i=1,2)$.  The support of the conductrix is contained in $F_i$ and hence these fibrations are elliptic (Lemma \ref{hascusp}).  Hence the singular fibers of $f_i$ are blown up 
the singular fibers of type $({\rm II}^*)$ or of type $({\rm II}^*, {\rm I}_1)$
(Proposition \ref{Lang}).
Obviously the first case does not occur.
Let $E_{11}$ (resp. $E_{12}$) be the $(-1)$-curve such that 
$2E_{11}+B$ (resp. $2E_{12}+B$) is the fiber of $f_1$ (resp. $f_2$).  
Since $S$ has only three genus 1 fibrations as above, curves with negative self-intersection number are $B$, 
$E_0, E_1,\ldots, E_{12}$.  

Now we show that one of $F_1, F_2$ is a multiple fiber and the other is not.
Assume that both of them are not multiple.  Then $E_{11}$ and $E_{12}$ do not meet
a 2-section $E_0$ of these fibrations.  By contracting $E_0, E_{10}, E_9, E_7, \ldots, E_3$ successively and $E_{11}, E_{12}$, we get a minimal rational surface with the Picard number
$11 -10=1$. On the other hand, the images of $E_1, E_2$ have the 
self-intersection number $0$ which is a contradiction.  Next assume that both of them are
multiple.  Then $E_{11}$ and $E_{12}$ meet $E_0$ with multiplicity 1.
By contracting $E_0, E_{10}, E_9, E_7, \ldots, E_3, E_2$ successively, we get a minimal rational surface with the Picard number
$11 - 9=2$.  Then there exists exactly one non-singular rational curve 
with non-positive self-intersection number (the image of $E_1$) which is a contradiction.

Thus we may assume that $F_2$ is multiple and $F_1$ is not.
We can easily check that $E_{12}\cdot E_0=1$, 
$E_{11}\cdot E_0=0$, $E_{11}\cdot E_2=1$, $E_{12}\cdot E_1=2$.  By contracting the curves $E_0, E_{10}, E_9, E_7, E_6, E_5, E_4, E_3, E_1$ and $E_{11}$ successively, we obtain 
${\bf P}^2$.
Denote by $p$, $q$ the image of $E_0$, $E_{11}$, and by $Q$, $L$, $C$ the image of $E_8$, $E_2$, $B$, respectively.  Then 
$L$ is a line, $Q$ a non-singular conic, 
$C$ a rational sextic curve with two singular points at $p, q$. 
The line $L$ meets $Q$ at $p$ with multiplicity 2, and $L$ meets $C$ at $p$ with multiplicity 4 and at
$q$ with multiplicity 2.  The conic $Q$ meets $C$ at $p$ with multiplicity 12.
We may assume that
$$L: x=0,\ Q: y^2+xz=0, \ p=(0:0:1).$$
It follows from an elementary but long calculation that 
$$q=(0:1:0), \quad C : x^6+(y^2+xz)^2(ax+z)z =0 \quad (a\ne 0 \in k).$$
Thus the pencil  
$$\{sx^2(y^2+xz)^2 + t(x^6+(y^2+xz)^2(x+az)z)\}_{(s:t)\in {\bf P}^1}$$
of sextic curves gives a family of Coble surfaces of type $\tilde{D}_8$ with one boundary component.
The projective transformation
$$(x:y:z)\to (x: \sqrt{a}x + y: ax +z)$$
acts on the pencil as an automorphism of order 2.

Next we give the Coble surface of type $\tilde{D}_8$ with one boundary component as a quotient of a rational surface by a derivation.
We use the same notation as in \cite{KKM}.
To construct classical Enriques surfaces with the dual graph in Figure \ref{graphD8}, we
considered 
the following derivation on a non-singular quadric ${\bf P}^1\times {\bf P}^1$
with homogeneous coordinates $((u_0:u_1),(v_0:v_1))$
 defined by
\begin{equation}\label{D8-der}
D= {1\over xy^2} \left(ax^2y^2\frac{\partial}{\partial x} + (x^4y^4+by^4+x^2y^2+x^2)\frac{\partial}{\partial y}\right), 
\end{equation}
where $a, b \in k, a, b\ne 0$, $x=u_0/u_1, y=v_0/v_1$.  Since $D$ has isolated singularities, we successively blow up the singularities and finally get a non-singular
surface $Y$ with
$$K_Y= -(2F_0+2E_2+E_3+3E_5+4E_6+3E_7+2E_8+2E_9+4E_{10}+E_{11}+2E_{12})$$
(see \cite[Figure 31]{KKM}) and a divisorial derivation denoted by the same symbol $D$.

Now we put $b=0$ in (\ref{D8-der}).
Then we have the following Lemma.

\begin{lemma}\label{D8pole}
{\rm (i)} The integral curves with respect to $D$ are 
$E_2, E_3, E_5, E_7, E_9, E_{11}$ in \cite[Figure 31]{KKM}.

{\rm (ii)} $(D) =-(2F_0-E_1+ 3E_2+2E_3+4E_5+4E_6+3E_7+2E_8+2E_9+4E_{10}+E_{11}+2E_{12})$.

{\rm (iii)} $(D)^2 = -12$. 

{\rm (iv)} $K_{Y}\cdot (D) = -4.$
\end{lemma}

Let $Y^D$ be the quotient of $Y$ by $D$ and $\pi:Y\to Y^D$ the canonical map.  
By the formula (\ref{euler}) and Lemma \ref{D8pole}, $Y^D$ is smooth.
It follows from the canonical divisor formula
(\ref{canonical}) and Lemma \ref{D8pole} that
$\pi^*K_{Y^D}= -E_1+E_2 +E_3+E_5$.  Denote by $\bar{E}_i$ the image of $E_i$ on $Y^D$ 
($i=1,2,3,5$).
Since $E_1, E_2, E_3, E_5$ are $(-2)$-curves (see \cite[Figure 7]{KKM}) and
$E_2, E_3, E_5$ are integral and $E_1$ is not (Lemma \ref{D8pole}),
$\bar{E}_2, \bar{E}_3, \bar{E}_5$ are $(-1)$-curves and $\bar{E}_1$ is a
$(-4)$-curve (Proposition \ref{insep}).
Again by Proposition \ref{insep}, we have $2K_{Y^D}= - \bar{E}_1 + 2\bar{E}_2 +2\bar{E}_3+2\bar{E}_5$.
By contracting $\bar{E}_2, \bar{E}_3, \bar{E}_5$, we obtain a non-singular surface $S$ with $|-2K_S| = \{\bar{E}_1\}$.  
Thus we have a Coble surface with the dual graph of type $\tilde{D}_8$ and with one boundary component.
\begin{theorem}\label{D8}
Coble surfaces of type $\tilde{D}_8$ with one boundary component
form a $1$-dimensional irreducible family.  
Its automorphism group is isomorphic to ${\bf Z}/2{\bf Z}$.
The surfaces $S$ are specializations of classical Enriques surfaces of type $\tilde{D}_8$.
\end{theorem}

\end{example}

\begin{example}\label{exVII-10} {\bf A Coble surface of type VII with ten boundary components.}

This example was given in \cite[Example 9.8.16]{DK}.
We only state the result.

\begin{theorem}\label{VII10thm}\cite[Theorem 9.8.18]{DK}
There exists a unique Coble surface $S$ of type ${\rm VII}$ with ten boundary components
which is obtained by blowing up the fifteen intersection points of ten lines on the quintic del Pezzo surface.
The automorphism group ${\rm Aut}(S)$ is isomorphic to $\mathfrak{S}_5$.
The surface $S$ is obtained as a quotient of the supersingular $K3$ surface $Y$ with Artin invariant $1$ by a derivation, and it is a specialization of classical Enriques surfaces of type ${\rm VII}$.
\end{theorem}

\end{example}

\begin{example}\label{exVII-2} {\bf A Coble surface of type VII with two boundary components.}

This example was also given in \cite[Example 9.8.18]{DK}.
We only state the result.

\begin{theorem}\label{VII2thm}\cite[Theorem 9.8.19]{DK}
There exists a unique Coble surface $S$ of type ${\rm VII}$ with two boundary components
which is obtained by blowing up seven points on the quintic del Pezzo surface.
The automorphism group ${\rm Aut}(S)$ is isomorphic to $\mathfrak{S}_5$.
The surface $S$ is obtained as a quotient of the supersingular $K3$ surface $Y$ with Artin invariant $1$ by a derivation, and it is a specialization of classical Enriques surfaces of type ${\rm VII}$.
\end{theorem}
 
\end{example}

\begin{example}\label{exVIII} {\bf A Coble surface of type VIII with four boundary components.}

The dual graph is given in Figure \ref{graphVIII}.
Recall that the automorphism group of this graph is $\mathfrak{S}_4$ (\cite[Theorem 9.4]{KKM}).  It contains parabolic subdiagrams of type 
$\tilde{D}_5\oplus \tilde{A}_3$, $\tilde{E}_6\oplus \tilde{A}_2$ and $\tilde{D}_6\oplus \tilde{A}_1\oplus \tilde{A}_1$ which correspond to
elliptic fibrations with singular fibers of type $({\rm I}_1^*,{\rm I}_4)$,
$({\rm IV}^*,{\rm I}_3, {\rm I}_1)$ (or $({\rm IV}^*,{\rm IV})$) and a quasi-elliptic fibration of type $({\rm I}^*_2,{\rm III},{\rm III})$, respectively, by Propositions \ref{Lang}, \ref{Ito}.  
Since $S$ has four boundary components, two fibers of genus 1 fibrations with singular fibers of type 
$({\rm IV}^*,{\rm I}_3, {\rm I}_1)$  and  of type $({\rm I}^*_2,{\rm III},{\rm III})$
are blown up.  Hence these are Jacobian fibrations.  Their Mordell-Weil groups generate 
$\mathfrak{S}_4$.  

Recall that the ten vertices $E_1,\ldots, E_{10}$ in Figure
\ref{graphVIII} are represented by $(-2)$-curves and the remaining six vertices
$E_{11},\ldots, E_{16}$ are represented by $(-1)$-roots.
There are six $(-1)$-curves $e_{11},\ldots, e_{16}$ and four boundary components such that
$$E_{11}=2e_{11}+{1\over 2}B_1+{1\over 2}B_3,\ E_{12}=2e_{12}+{1\over 2}B_1+{1\over 2}B_2,
\ E_{13}=2e_{13}+{1\over 2}B_3+{1\over 2}B_4,$$
$$E_{14}=2e_{14}+{1\over 2}B_2+{1\over 2}B_3,\ E_{15}=2e_{15}+{1\over 2}B_1+{1\over 2}B_4,
\ E_{16}=2e_{16}+{1\over 2}B_2+{1\over 2}B_4.$$
Note that by contracting $e_{13}, e_{15}, e_{16}$ and $B_4$, we obtain 
a cycle of type IV which defines a Halphen surface of index 2 with a multiple fiber
$E_5+E_7+E_9+2(E_2+E_3+E_4)+3E_1$ of type ${\rm IV}^*$ and the fiber of type IV.
Then by contracting $e_{11}, e_{12}, e_{14}, E_5, E_7, E_9, E_2, E_3, E_4$ successively,
we obtain ${\bf P}^2$.  Denote by $\ell_0, \ell_1, \ell_2, \ell_3, Q_1, Q_2, Q_3$ the
image of $E_1$, $E_6$, $E_8$, $E_{10}$, $B_1$, $B_2$, $B_3$, respectively.
Then $\ell_i$ are lines and $Q_j$ are conics.  
Let $q_0, q_1, q_2, q_3$ be the images of $e_{11}, e_{12}, e_{14}, B_4$, respectively.
We may assume that 
$$q_0=(1:1:1),\ q_1=(0:1:1),\ q_2=(1:0:1),\ q_3=(1:1:0).$$
Then 
$$\ell_0: x+y+z=0,\ \ell_1: y+z=0,\ \ell_2: x+z=0,\ \ell_3: x+y=0.$$
A simple calculation shows that
$$Q_1: x^2+y^2+z^2+yz=0,\ Q_2:  x^2+y^2+z^2 + zx=0,\ Q_3: x^2+y^2+z^2+xy =0.$$
Note that $Q_i$ and $Q_j$ meet at $q_k$ $(\{i,j,k\}=\{1,2,3\})$ with multiplicity 3 and
at $q_0$ transverselly.

Conversely
we define a pencil $\calP$ of sextic curves by
$$\calP = \{C_{(s,t)}\}_{(s:t)\in {\bf P}^1}=
\{t\ell_0^6+ sQ_1Q_2Q_3\}_{(s:t)\in {\bf P}^1}.$$
By blowing up the base points, we obtain a Halphen surface with a multiple fiber of type ${\rm IV}^*$ and a fiber of type IV.

As mentioned above ${\rm Aut}(S)$ contains $\mathfrak{S}_4$.
The stabilizer subgroup of $B_4$ in $\mathfrak{S}_4$ is isomorphic to $\mathfrak{S}_3$ which induces the permutation group of
coordinates on ${\bf P}^2$.  A direct calculation shows that the stabilizer of the pencil is $\mathfrak{S}_3$.  On the other hand, any automorphism of $S$ acting trivially on the dual graph given in Figure \ref{graphVIII} is contained in $\mathfrak{S}_3$.  Thus we have ${\rm Aut}(S) \cong \mathfrak{S}_4$.

Next we give the Coble surface of type ${\rm VIII}$ with four boundary components as a quotient of a rational surface $Y$ by a derivation.
We use the same notation as in \cite{KKM}.
To construct classical Enriques surfaces with the dual graph in Figure \ref{graphVIII}, we
considered a rational elliptic surface defined by the Weierstrass equation
$$y^2+txy + ty = x^3+x^2$$
which has a singular of type ${\rm III}$ over the point $t=0$ and
a singular fiber of type ${\rm I}_8$ over the point $t=\infty$, and
the following derivation $D$ defined by
\begin{equation}\label{E6-der}
D= t(at+1)\frac{\partial}{\partial t} + (x + 1)\frac{\partial}{\partial x}, \quad  a\ne 0 \in k.
\end{equation}
The derivation $D$ has an isolated singularity at the singular point of the fiber of type III.  We successively blow up the singularities of derivations 
and obtain a non-singular surface $Y$
with
$$K_Y= -(F_0+F_1+F_2+2F_3)$$
(see \cite[Figure 11]{KKM}) and a divisorial derivation denoted by the same symbol $D$.

Now consider the case $a=0$.  Then we have the following Lemma.

\begin{lemma}\label{VIIIpole}
{\rm (i)} The integral curves with respect to $D$ are 
$E_0, E_3, E_4, E_7, F_0, F_1, F_2, b_1, b_2$ in \cite[Figure 11]{KKM}.

{\rm (ii)} $(D) =-(F_0+F_1 + F_2 + 2F_3 - E_1 - E_2 - E_5 - E_6)$.

{\rm (iii)} $(D)^2 = -12$. 

{\rm (iv)} $K_{Y}\cdot (D) = -4.$
\end{lemma}

Let $Y^D$ be the quotient of $Y$ by $D$ and $\pi:Y\to Y^D$ the canonical map.  
By the formula (\ref{euler}) and Lemma \ref{VIIIpole}, $Y^D$ is smooth.
It follows from the canonical divisor formula
(\ref{canonical}) and Lemma \ref{VIIIpole} that
$\pi^*K_{Y^D}= -E_1-E_2 -E_5-E_6$.  Denote by $\bar{E}_i$ the images of $E_i$ on $Y^D$.
Since $E_1, E_2, E_5, E_6$ are $(-2)$-curves (see \cite[Figure 11]{KKM}) and
they are non-integral (Lemma \ref{D8pole}),
$\bar{E}_1, \bar{E}_2, \bar{E}_5, \bar{E}_6$ are 
$(-4)$-curve (Proposition \ref{insep}).
Again by Proposition \ref{insep}, we have $|-2K_{Y^D}|= \{\bar{E}_1 + \bar{E}_2 + 
\bar{E}_5 + \bar{E}_6\}$.
Thus we have a Coble surface $S=Y^D$ with the dual graph of type VIII and with four boundary components.
Thus we have obtained the following theorem.

\begin{theorem}\label{VIIIthm}
The surface $S$ is a unique Coble surface of type ${\rm VIII}$ with four boundary components.
The automorphism group ${\rm Aut}(S)$ is isomorphic to $\mathfrak{S}_4$.
The surface $S$ is a specialization of classical Enriques surfaces of type ${\rm VIII}$.
\end{theorem}

\end{example}

\end{document}